\documentclass[a4paper, 11pt]{article}
\usepackage{graphicx}
\usepackage{amssymb}
\usepackage{epstopdf}
\usepackage{fancyhdr}
\usepackage{amsmath}
\usepackage{amsthm}
\usepackage{hyperref}
\usepackage{makeidx}
\usepackage{mathdesign}
\usepackage{color}
\usepackage{soul}
\definecolor{purple}{rgb}{.9,0,.9}

\usepackage[top=2.2 cm,bottom=2.2 cm,left=1.9 cm, right=1.9 cm]{geometry}

\newcommand{\ve}{\varepsilon}

\newtheorem{theorem}{Theorem}[section]
\newtheorem{lemma}[theorem]{Lemma}

\newtheorem{definition}[theorem]{Definition}

\newtheorem{remark}[theorem]{Remark}

\def\Xint#1{\mathchoice
{\XXint\displaystyle\textstyle{#1}}%
{\XXint\textstyle\scriptstyle{#1}}%
{\XXint\scriptstyle\scriptscriptstyle{#1}}%
{\XXint\scriptscriptstyle\scriptscriptstyle{#1}}%
\!\int}
\def\XXint#1#2#3{{\setbox0=\hbox{$#1{#2#3}{\int}$ }
\vcenter{\hbox{$#2#3$ }}\kern-.58\wd0}}

\def\dashint{\Xint-}
\def\be{{\bf e}}

\title{Homogenization  of   biomechanical models  for plant tissues \thanks{This research was  supported by  Northern Research Partnership early career research exchange grant.
M.\ Ptashnyk  gratefully acknowledge the support of the EPSRC First Grant EP/K036521/1 ``Multiscale modelling and analysis of mechanical properties of plant cells and tissues''.} }
\author{Andrey Piatnitski\thanks{The Arctic University of Norway, Campus in Narvik, P.O.Box 385, Narvik 8505, Norway \ \ and \ \ Institute for information transmission problems of RAS, 19, Bolshoy Karetny per.,  Moscow,  127051, Russia   ({\tt andrey@sci.lebedev.ru}).}
 \and Mariya Ptashnyk\thanks{Department of Mathematics, University of Dundee, DD1 4HN Dundee, Scotland, UK,  ({\tt mptashnyk@maths.dundee.ac.uk}). } }

\begin{document}
%\date{}
\setcounter{footnote}{1}

\maketitle

\begin{abstract}
In this paper homogenization    of a mathematical model for plant tissue biomechanics is presented.  The microscopic model  constitutes  a strongly coupled system of  reaction-diffusion-convection  equations for chemical processes in plant cells,   the equations of  poroelasticity for elastic deformations of  plant cell walls and middle lamella,  and Stokes equations for fluid flow inside the cells.  The chemical process in  cells and the elastic properties of cell walls and middle lamella are coupled because elastic moduli depend on  densities involved in chemical reactions, whereas chemical reactions depend on  mechanical stresses.
 Using homogenization techniques we derive rigorously a macroscopic  model for plant biomechanics.  To pass to the limit in the nonlinear reaction terms, which depend on  elastic strain,   we prove the strong two-scale convergence of the displacement gradient and velocity field.
\end{abstract}

{\small  \quad {\bf Key words. }  {Homogenization; two-scale convergence; periodic unfolding method;
poroelasticity; \\  \phantom{a} \quad  \quad Stokes system; biomechanics of plant tissues.}}

\section{Introduction}

Analysis of interactions between mechanical properties  and chemical processes, which   influence the elasticity and extensibility of plant cell  tissues, is important  for better understanding  of  plant growth and  development, as well as their   response to environmental changes.   Plant tissues are composed of cells  surrounded by cell walls and connected by  a cross-linked  pectin network of middle lamella.  Plant cell walls  must be very strong   to resist high internal hydrostatic pressure and at the same time  flexible to permit growth.
 It is supposed that calcium-pectin cross-linking chemistry is one of the main regulators of plant cell wall elasticity and extension \cite{WHH}.  Pectin is  deposited to  cell walls in  a methylesterified form. In cell walls and middle lamella  pectin  can be modified  by the enzyme pectin methylesterase (PME), which removes  methyl groups by  breaking ester bonds. The de-esterified pectin  is able to form  calcium-pectin cross-links, and thus stiffen the cell wall and reduce its expansion, see e.g.~\cite{WG}. On the other hand,  mechanical stresses can break calcium-pectin cross-links and hence increase the extensibility  of plant cell walls and middle lamella.  It has been shown that   chemical properties of pectin  and  the control of the density of  calcium-pectin cross-links  greatly influence the  mechanical deformations of plant cell walls \cite{P2010},  and the interference with PME activity causes dramatic changes in growth behavior of plant  tissues \cite{WG}.

To analyze  the interactions between  calcium-pectin dynamics and  deformations  of a plant tissue, we derive a mathematical  model for plant tissue biomechanics  at  the  length scale of plant cells.  In the microscopic model we consider  a system of reaction-diffusion-convection equations describing the dynamics of the  methylesterified pectin, demethylesterified pectin, calcium ions, and calcium-pectin cross-links.
 Elastic deformations and water flow  are modelled by the equations of poroelasticity  for   cell walls  and middle lamella coupled with the  Stokes system for the flow velocity  inside  cells. The interplay between the mechanics and the chemistry   comes in by assuming that the elastic  properties of  cell walls and middle lamella depend on the density of the calcium-pectin cross-links and that  the  stress within  cell walls and middle lamella   can break the cross-links.  Thus the  microscopic problem is   a strongly coupled system of the Stokes equations, reaction-diffusion-convection equations, with reaction terms  depending on the displacement gradient,   and  equations of poroelasticity, with    elastic moduli  depending on  the density of cross-links.  To address the   situations when a  plant tissue is completely and not completely saturated by water, we consider both evolutional and  quasi-stationary equations  of poroelasticity.

  To show the existence of a weak solution of the  microscopic equations  we use a classical approach  and apply the  Banach fixed-point theorem.  However, due to quadratic non-linearities of reaction-terms,  the proof of the contraction inequality is not  standard and relies on delicate \textit{a priori} estimates for  the $L^\infty$-norm  of a solution of  the reaction-diffusion-convection system in terms of  the $L^2$-norm of displacement gradient and flow velocity.  The Alikakos \cite{Alikakos}  iteration technique is applied to show the  uniform boundedness of some of components of solutions of the microscopic equations.

 To analyze effective  mechanical properties  of  plant  tissues,     we derive rigorously a  macroscopic model for plant biomechanics  using homogenization techniques.
  The two-scale convergence, e.g.~\cite{allaire,Nguetseng}, and the periodic unfolding  method, e.g.~\cite{CDDGZ}, are applied to obtain the macroscopic equations.    The main   mathematical difficulty in the derivation of the macroscopic problem  arises from the strong coupling between the equations of poroelasticity     and the system of reaction-diffusion-convection equations.  In order  to pass to the limit in the nonlinear reaction terms  we prove the  strong two-scale convergence for the displacement gradient and fluid flow velocity, essential for the homogenization of the  coupled problem considered here. Due to the dependence of the elasticity tensor on the time variable,  in the proof of the strong two-scale convergence  a specific form of   the energy functional is considered.

Similar to  the microscopic problem,  to prove uniqueness of a solution of the macroscopic equations we derive a contraction inequality involving the  $L^\infty$-norm of the difference of two solutions of the reaction-diffusion-convection equations.  This contraction inequality also ensures  also the well-posedness of the limit system.

The poroelasticity  equations, modelling   interactions   between  fluid flow  and elastic stresses  in  porous media, were first  obtained  by Biot  using  a phenomenological approach \cite{Biot_1972, Biot_1962, Biot_1941} and subsequently  derived by applying  techniques of homogenization theory. Formal asymptotic expansion  was undertaken by \cite{Auriault, Burridge, Levy, SP} to derive Biot equations from microscopic description of elastic deformations of  a solid matrix and fluid flow in porous space.
The rigorous homogenization of the coupled system of equations of  linear elasticity  for a solid matrix combined with the Stokes or Navier-Stokes equations for the fluid part    was conducted   in \cite{Clopeau, GM,  Meirmanov_2007, Nguetseng_1990} by using the two-scale convergence method.  Depending on the ratios between the physical parameters,  different macroscopic equations were obtained, e.g.\ Biot's equations of poroelasticity, the system consisting of the anisotropic Lam\'e equations for the solid component, and the acoustic equations for the fluid component, the equations of viscoelasticity.
The homogenization of  a mathematical model for  elastic deformations, fluid flow  and chemical processes  in a cell  tissue was considered   in \cite{JMN}. In contrast to the problem considered in the present paper,  in \cite{JMN} the  coupling between the  equations of linear elasticity  and reaction-diffusion-convection equations for a concentration was given only through the dependence of  the elasticity tensor on the chemical concentration.  The existence and uniqueness  of a solution for equations of poroelasticity  were studied in \cite{Showalter, Zenisek_1984}.

Compared to the many   results for  the equations of  poroelasticity, there exist  only few studies  of  interactions between a free fluid and a deformable porous medium.    In \cite{Showalter_2005} nonlinear semigroup method was used for mathematical analysis of a system of poroelastic equations coupled with the Stokes equations for free fluid flow.  A rigorous derivation of  interface conditions between a poroelastic medium  and an elastic body  was considered in \cite{MW}. Numerical methods  for  coupled Biot poroelastic system  and Navier-Stokes equations were derived in  \cite{Badia}.  The Nitsche method for enforcing interface conditions was applied  in \cite{Bukac} for numerical simulation of  the Stokes--Biot coupled system.

Several results  on homogenization of equations of  linear elasticity can be found  in  \cite{BLP,JKO,OShY,SP} (and the references therein).  Homogenization of the  microscopic model for plant cell wall biomechanics,  composed of equations of linear elasticity and    reaction-diffusion equations for chemical processes,   has been  studied in  \cite{Ptashnyk}.

This paper is organized as follows. In Section~\ref{model} we  derive  the microscopic model for plant tissue biomechanics.   {\it A priori} estimates as well as the  existence and uniqueness  of a   weak solution of the microscopic problem are obtained  in  Section~\ref{a_priori}.  In Section~\ref{convergence} we prove the convergence results for  solutions of the microscopic problem. The multiscale analysis   of the coupled poroelastic  and Stokes problem  is conducted in Sections~\ref{macro_elasticity}. In Section~\ref{strong_convergence} we show strong two-scale convergence  of displacement  gradient and flow velocity.  The macroscopic equations for the system of reaction-diffusion-convection equations are derived  in Section~\ref{macro_diffusion}. The well-posedness and  uniqueness of a solution of the macroscopic problem are proven in Section~\ref{uniqueness_macro}.  In Section~\ref{incompresible} we consider the incompressible and quasi-stationary cases for the equations of poroelasticity.

\section{Microscopic model}\label{model}
 In the mathematical model for plant tissue biomechanics we consider   interactions between the mechanical properties of a plant tissue  and the chemical processes in plant cells.  A plant tissue is composed of  cell interior (intracellular space), the plasma membrane, plant cell walls,  and the cross-linked pectin network of the middle lamella joining individual cells together.  Primary plant cell walls  consist mainly of oriented cellulose microfibrils (that strongly influence the cell wall stiffness), pectin, hemicellulose, proteins, and water.
  It is supposed that calcium-pectin chemistry, given by the de-esterification of pectin and  creation and breakage of calcium-pectin  cross-links, is  one of the main regulators of cell wall elasticity, see e.g.~\cite{WHH}.    Hence in our mathematical model we consider the interactions and  two-way coupling  between calcium-pectin chemistry  and elastic deformations of a plant tissue. To describe the coupling between the mechanics and chemistry we consider the dynamics of pectins, calcium, and calcium-pectin cross-links, water flow in a plant tissue and the poroelastic nature of cell walls and middle lamella.

 To derive a mathematical model for plant tissue biomechanics, we denote a domain occupied by  a  plant tissue by $\Omega\subset \mathbb R^3$, where $\Omega$ is a bounded domain with $C^{1,\alpha}$ boundary for some $\alpha>0$. Notice that all results  also hold  for  a two-dimensional domain.   Then the   time-independent  domains  $\Omega_f\subset \Omega$ and $\Omega_e\subset\Omega$, with $\overline\Omega= \overline\Omega_e\cup \overline\Omega_f$  and $\Omega_e\cap\Omega_f = \emptyset$,   represent the reference (Lagrangian) configurations of the  intracellular (cell inside) and   intercellular  (cell walls and middle lamella) spaces, respectively,  and $\Gamma$ denotes the boundaries between the cell inside and cell walls and corresponds to the plasma membrane.
 Since  $\Gamma$ represents  the interface between elastic material and fluid in the Lagrangian configuration, it is also independent of time.

 Pectin is deposited into the cell wall in a highly methylesterified state and is modified by the wall enzyme pectin-methylesterase (PME), which removes methyl groups \cite{WG}.   It was observed experimentally that pectins can diffuse in a plant cell wall matrix, see e.g.\ \cite{Cosgrove,  Proseus_2005, Somerville}.
 Thus in the balance equation for the density of the methylesterified pectin $b_{e,1}$ and demethylesterified pectin $b_{e,2}$
 $$
 \partial_t b_{e,j} + {\rm div} J_{b,j}= g_{b,j} \quad \text{in } \Omega_e, \quad j=1,2,
 $$
  we assume the flux to be determined by Fick's law, i.e.\ $J_{b,j} = -D_{b_e,j} \nabla b_{e,j}$, for $j=1,2$, and $D_{b_e,j} >0$. The term $g_{b,j} $ models chemical reactions, that  correspond to the demethylesterification processes and creation and breakage of calcium-pectin cross-links.
  In general, diffusion coefficients for pectins and calcium depend on the microscopic structure of cell wall  given by the cell wall microfibrils and hemicellulose network, which is assumed
to be given and not to change in time, as well as on the density of pectins and calcium-pectin
cross-links. For presentation simplicity we assume here that the diffusion coefficient does not depend on the dynamics
of pectin and calcium-pectin cross-link densities. However the analysis can be conducted in the same way for the generalised model in which the diffusion of pectin, calcium, and cross-links depends on pectin and cross-link densities, assuming that diffusion coefficients are uniformly bounded from below and above, which is biologically
sensible.
The modification of methylesterified pectin by PME is modelled by the reaction term $g_{b,1}=-\mu_1 b_{e,1}$ with  some $\mu_1>0$.
  For simplicity  we assume that there is a constant concentration of PME enzyme in the cell wall.  By simple modifications of the analysis considered here, the same results can be obtained for a generalized model including the dynamics of PME and   chemical reactions between PME and pectin, see \cite{Ptashnyk} for the derivation of the corresponding system of equations.

 The deposition of the methylesterified pectin is described by the inflow boundary condition on the cell plasma membrane. We also assume that the demetylesterified pectin cannot move back into the cell interior:
 $$
 D_{b_e,1} \nabla b_{e,1} \cdot n= P_1(b_{e,1}, b_{e,2}, b_{e,3}), \quad  D_{b_e,2} \nabla b_{e,2} \cdot n =  0 \quad \text{ on } \Gamma.
 $$
  To account for mechanisms controlling the amount of pectin in the cell wall, we assume that the inflow of new methylesterified pectin depends on the density of methylesterified and demethylesterified  pectin, i.e.\ $b_{e,1}$ and $b_{e,2}$,  and calcium-pectin cross-links $b_{e,3}$.

    We consider the diffusion and transport by water flow of calcium molecules  in the symplast (in the cell inside) and  diffusion of calcium in  apoplast (cell walls and middle lamella), see e.g.\ \cite{White_2001}. Then the balance equations for calcium densities $c_f$  and $c_e$ in $\Omega_f$ and $\Omega_e$, respectively, are given by
    $$
    \begin{aligned}
  &  \partial_t c_f- {\rm div} (D_f \nabla c_f- \mathcal G(\partial_t u_f)c_f) = g_f  \quad && \text{ in } \Omega_f,  \\
    & \partial_t c_e - {\rm div} (D_e \nabla c_e) = g_e  \quad && \text{ in } \Omega_e,
    \end{aligned}
    $$
   where the chemical reaction term $g_f=g_f(c_f)$ in $\Omega_f$ describes the decay and/or  buffering of calcium inside the plant cells, see e.g.\ \cite{Xiong},  $g_e$ models the interactions between calcium and demethylesterified pectin in cell walls and middle lamella and the creation and breakage of calcium-pectin cross-links,   and $\mathcal G$ is a bounded function of the intracellular flow velocity $\partial_t u_f$.  The condition that $\mathcal G$ is bounded is natural from  the biological and physical point of view, because the flow velocity in plant tissues is bounded. This condition is also essential for a rigorous mathematical  analysis of the model.
   We assume that as the result of  the breakage of a calcium-pectin cross-link by mechanical stresses we obtain  one calcium molecule and two galacturonic acid monomers  of demethylesterified pectin. A detailed  derivation of the chemical reaction term $g_e$  is given in \cite{Ptashnyk}. See also Remark~\ref{remark1}  for the detailed form of the reaction terms.
     We assume  a passive  flow of calcium between cell walls and cell inside and assume that the exchange of  calcium between cell insides and cell walls  is facilitated only on parts of the cell membrane $\Gamma \setminus \widetilde \Gamma$, i.e.\
   $$
   \begin{aligned}
 &c_f = c_e, \qquad &&  (D_f \nabla c_f - \mathcal G(\partial_t u_f)c_f)  \cdot n = D_e \nabla c_e\cdot n \; \;   && \text{ on } \; \;   \Gamma \setminus \widetilde \Gamma, \\
&D_e \nabla c_e\cdot n =0, \; && (D_f \nabla c_f - \mathcal G(\partial_t u_f)c_f) \cdot n =0 \qquad  \; \;    && \text{ on }\;  \; \widetilde \Gamma.
   \end{aligned}
   $$
    The regulation  of calcium flow  by mechanical properties of the cell wall will be considered in future studies.

  Calcium-pectin cross-links $b_{e,3}$ are created by electrostatic and ionic binding between two galactu\-ronic acid monomers of pectin chains and calcium molecules. It is also known that these cross-links are very stable and can be disturbed mainly by thermal treatments and/or mechanical forces, see e.g.\ \cite{PB, Proseus_2007}.  Thus assuming a constant temperature, the calcium-pectin chemistry can be described as an reaction between calcium molecules and pectins, where the breakage of cross-links depends on the deformation gradient of the cell walls.
 Hence we assume that the cross-links are disturbed by the mechanical stresses in  cell walls and middle lamella, see \cite{Ptashnyk} for a  detailed description of the modelling of the calcium-pectin chemistry.  A similar approach was used in \cite{Rojas} to model a chemically mediated mechanical expansion of the cell wall of a pollen tube.  There are no experimental observations of diffusion of calcium-pectin cross-links $b_{e,3}$, however since  most calcium-pectin cross-links are not attached to cell wall microfibrils \cite{Cosgrove}  it is possible that cross-links can move inside the cell wall matrix by a very slow diffusion
$$
\partial_t b_{e,3} - {\rm div}(D_{b_e,3} \nabla b_{e,3}) = g_{b,3}  \quad \text{ in } \Omega_e,
$$
where $D_{b_e,3}>0$ and  the reaction term $g_{b,3}$  models the creation and breakage by mechanical stresses of calcium-pectin cross-links, see Remark~\ref{remark1} for a detailed form of $g_{b,3}$.
 For the analysis presented here the diffusion term in the equations for calcium-pectin cross-link density is important. However the same results can be obtained if one
assumes that calcium-pectin cross-links do not diffuse and that the reaction terms in equations
for  pectin, calcium and calcium-pectin cross-links  depend on a local average of the deformation gradient, reflecting the fact that in a  dense pectin network mechanical forces have a non-local effect on the calcium-pectin chemistry, see \cite{Ptashnyk}.

To describe  elastic deformations of plant cell walls and middle lamella  we consider the equations of poroelasticity reflecting the microscopic structure of cell walls  and middle lamella permeable to  fluid flow:
$$
\begin{aligned}
&\rho_e\partial^2_t u_e- {\rm div} ( {\bf E}(b_{e,3}) \be( u_e)) + \alpha \nabla p_e  = 0  \qquad &&  \text{ in } \Omega_{e},  \\
&\rho_p\partial_t p_e- {\rm div} ( K_p \nabla p_e - \alpha \,  \partial_t u_e)  = 0  \qquad &&  \text{ in } \Omega_{e}.
\end{aligned}
$$
Here  $u_e$ denotes the displacement from the equilibrium position,  ${\bf e}(u_e)$ stands for the symmetrized gra\-dient of $u_e$, and  $\rho_e$ denotes the poroelastic wall density. Since we consider the equations of poroelasticity, one more unknown function that should be determined is the pressure,  denoted by $p_e$.  The mass storativity coefficient is denoted by $\rho_p$ and $K_p$ denotes the hydraulic conductivity of cell walls and middle lamella. In what follows, we  assume that the Biot-Willis constant is   $\alpha=1$.

It is observed experimentally that  the  load bearing  calcium-pectin  cross-links reduce cell wall expansion, see e.g.~\cite{WHH}.   Hence  elastic properties of  cell walls and middle lamella depend on  the chemical configuration of pectin and density of calcium-pectin cross-links, see e.g. \cite{Zsivanovits_2004}.    This is reflected in the dependence of the elasticity tensor ${\bf E}$ of the cell wall and middle lamella on the density of calcium-pectin cross-links $b_{e,3}$.  The differences in the elastic properties of cell walls and middle lamella  result in the dependence of the elasticity tensor ${\bf E}$  on the spatial variables.  Since the properties of calcium-pectin cross-links are changing during the deformation and  the stretched cross-links have different impact (stress drive hardening) on the elastic properties of the cell wall matrix than newly-created cross-links, see e.g.\ \cite{Broedersz, Proseus_2006, Schuster_2012},  we consider a non-local in time dependence of the Young's modulus of the cell wall matrix  on the density of calcium-pectin  cross-links, see Assumption ${\bf A1}$.
A similar  approach was used in \cite{JMN} to model the dependence of cell deformations on the  concentration of a  chemical substance.   We assume that the hydraulic conductivity tensor varies between cell wall and middle lamella and, hence, $K_p$ depends on  the spacial variables.

In the  cell inside, that is in $\Omega_{f}$,  the water flow is modelled by the Stokes system
$$
\rho_f  \partial^2_t u_f - \mu \, \text{div} ( \be(\partial_t u_f)) + \nabla p_f = 0,  \qquad   {\rm div} \partial_t u_f =0 \qquad   \text{ in } \Omega_{f},
$$
where  $\partial_t u_f$ denotes the fluid velocity,  $p_f$ the fluid pressure, $\mu$ the fluid viscosity,  and $\rho_f$ the fluid density.

As transmission  conditions between free fluid  and poroelastic  domains we consider  the continuity of normal flux, which corresponds to mass conservation,  and the continuity of the normal component of total stress on the interface $\Gamma$,  i.e.\ the total stress of the poroelastic medium is balanced  by the total stress of the fluid,   representing  the conservation of momentum,
\begin{equation}\label{trans1}
\begin{aligned}
&(-K_p \nabla p_e +  \partial_t u_e)\cdot n = \partial_t u_f\cdot n \quad &&  \text{ on } \Gamma,  \\
  & ({\bf E}(b_{e,3}) \, \be( u_e) -  p_e I)\,  n = (  \mu \, \be(\partial_t u_f)  - p_f I)\,  n  \quad &&  \text{ on } \Gamma.
\end{aligned}
\end{equation}
Also taking into account the channel structure of a cell membrane separating cell inside and cell wall, given by the presence of aquaporins, see e.g.\  \cite{CC_1994},  we assume that the  water flow between the poroelastic cell wall and cell inside is in the direction normal to  the interface between the free fluid and the poroelastic medium. Hence we assume the no-slip interface condition,  which is appropriate for  problems where at the interface the fluid flow in the tangential direction is not allowed, see e.g.\  \cite{Bukac},
$$
 \Pi_\tau \partial_t u_e = \Pi_\tau \partial_t u_f  \qquad \text{ on } \Gamma.
 $$
 By $\Pi_\tau w$ we define the tangential projection of a vector $w$, i.e.\ $\Pi_\tau w= w - (w\cdot n) n$, where $n$ is a normal vector and $\tau$ indicates the tangential subspace to the boundary.
The balance  of the normal components of the stress in the fluid phase across the interphase is given by
\begin{equation}\label{trans2}
n \cdot ( \mu \,  \be(\partial_t u_f) - p_f I) \, n  = - p_e \qquad \text{ on } \Gamma.
\end{equation}
Notice that the  transmission conditions \eqref{trans1} and \eqref{trans2} imply   ${\bf E}(b_{e,3}) \, \be( u_e) \, n\cdot n =0$ on $\Gamma$.
  The transmission conditions are motivated by the models describing coupling between Biot and Navier-Stokes or Stokes equations   considered in e.g.\ \cite{Badia, Bukac, Murad_2000, Murad_2001, Showalter_2005}.
  The coupling between elastic deformations and fluid flow is described in  the Lagrangian configuration and hence $\Gamma$ is a fixed interface between fluid domain and elastic material. Since
   in our model we consider only the linear elastic nature of the solid skeleton of the cell walls,  the transmission conditions   \eqref{trans1} and \eqref{trans2}  are the corresponding linearizations of the fluid-solid interface conditions, i.e.\
$|\det(I +\nabla u_e)| (\mu \, \be(\partial_t u_f(t,x+u_e))  - p_f(t,x+u_e) I) (I+\nabla u_e)^{-T}n$ is approximated by $(\mu \, \be(\partial_t u_f(t,x))  - p_f(t,x) I)n$ on $\Gamma$ and the first Piola-Kirchhoff stress tensor is equal to the Cauchy stress tensor in the first order approximation.

Then the  model for the densities  of  calcium,  pectins and calcium-pectin cross-links reads as
\begin{equation}\label{eq_codif_dim}
\begin{aligned}
&\partial_t b_e = {\rm div}(D_b\nabla b_e) + g_b(c_e, b_e, \be(u_e) )  \quad &&  \text{ in }  \Omega_{e}, \;  t>0\\
&\partial_t c_e = {\rm div}(D_e\nabla c_e) + g_e ( c_e, b_e, \be(u_e))  \quad &&  \text{ in } \Omega_{e}, \; t>0 ,\\
&\partial_t c_f = {\rm div}(D_f\nabla c_f - \mathcal{G}(\partial_t u_f) c_f)  + g_f(c_f)\qquad &&  \text{ in } \Omega_{f}, \;   t>0,\\
&D_b\nabla b_e \cdot n = P(b_e) && \text{ on }   \Gamma, \; \; \,  t>0,   \\
&c_e = c_f,  \qquad \qquad
D _e\nabla c_e \cdot n = (D_f\nabla c_f -   \mathcal G(\partial_t u_f) c_f)\cdot n  \quad  && \text{ on } \Gamma\setminus\widetilde \Gamma,\;  t>0, \\
&D_e\nabla c_e \cdot n =0, \qquad (D_f \nabla c_f - \mathcal G(\partial_t u_f) c_f) \cdot n = 0 \quad && \text{ on } \widetilde \Gamma, \; \; \, t>0, \\
& b_e(0,x) = b_{e0}(x), \quad  c_e(0,x) = c_{0}(x) && \text{ in } \Omega_e, \\
& c_f(0,x) = c_{0}(x) && \text{ in } \Omega_f,
\end{aligned}
\end{equation}
where $b_e= (b_{e,1}, b_{e,2}, b_{e,3})$, $D_f>0$, $D_e>0$,  and $D_b= \text{diag} (D_{b_{e,1}}, D_{b_{e,2}}, D_{b_{e,3}})$  with $D_{b_{e,j}} >0$, $j=1,2,3$, stands for   the diagonal matrix of diffusion coefficients for $b_{e,1}$,  $b_{e,2}$, and   $b_{e,3}$.\\
For elastic deformations  of cell walls  and middle lamella and  fluid  flow inside the cells  we have a coupled system of Stokes equations and poroelastic (Biot)  equations:
\begin{equation}\label{equa_cla_dim}
\begin{aligned}
&\rho_e\partial^2_t u_e- {\rm div} ( {\bf E}(b_{e,3}) \be( u_e)) +  \nabla p_e  = 0  \quad &&  \text{ in } \Omega_{e}, \; t>0, \\
&\rho_p\partial_t p_e- {\rm div} ( K_p \nabla p_e -  \partial_t u_e)  = 0  \quad &&  \text{ in } \Omega_{e}, \; t>0, \\
&\rho_f  \partial^2_t u_f - \mu \, \text{div} ( \be(\partial_t u_f)) + \nabla p_f = 0&&  \text{ in } \Omega_{f}, \; t>0, \\
&\text{div}\, \partial_t u_f = 0 &&  \text{ in } \Omega_{f}, \; t>0, \\
& ({\bf E}(b_{e,3}) \, \be( u_e) -  p_e I)\,  n = (  \mu \, \be(\partial_t u_f)  - p_f I)\,  n &&  \text{ on } \Gamma, \; t>0,  \\
& \Pi_\tau \partial_t u_e= \Pi_\tau \partial_t u_f, \qquad  n \cdot ( \mu \,  \be(\partial_t u_f) - p_f I) \, n  = - p_e \quad
  &&  \text{ on } \Gamma,\; t>0, \\
&(-K_p \nabla p_e +  \partial_t u_e)\cdot n = \partial_t u_f\cdot n  &&  \text{ on } \Gamma,\; t>0, \\
& u_e(0,x) = u_{e0}(x), \quad  \partial_t u_e (0,x) = u_{e0}^1(x), \quad   p_e (0,x) = p_{e0}(x) && \text{ in } \Omega_e, \\
 & \partial_t u_f (0,x) =u^1_{f0}(x) \quad   &&  \text{ in } \Omega_f.
\end{aligned}
\end{equation}
For multiscale analysis of the mathematical model \eqref{eq_codif_dim}--\eqref{equa_cla_dim} we derive the nondimen\-sio\-na\-lized equations  from the dimensional model by considering $t = \hat t t^\ast$, $x = \hat x x^\ast$,
$b_e= \hat b b_e^\ast$, $c_j=\hat b c_j^\ast$, $u_j = \hat u u_j^\ast$,  $p_j = \hat p p_j^\ast$, with $j=e,f$, ${\bf E} = \hat E {\bf E}^\ast$,  $K_p = \hat K K_p^\ast$,  $\mu= \hat \mu \mu^\ast$, $\rho_p = \hat \rho_p \rho_p^\ast$,   $\rho_j = \hat \rho \rho_j^\ast$, with $j=e,f$,   $D_j = \hat D D_j^\ast$ for $j=b,e,f$, $P(b_e)=\hat R \hat b P^\ast(b_e^\ast)$,  $g_j(c_e, b_e, \be(u_e) )= \hat g \hat b g^\ast_j(c_e^\ast, b_e^\ast, \be(u_e^\ast))$, for $j=b,e$,
and $g_f(c_f) = \hat g_f \hat b g_f^\ast(c^\ast_f)$.  The dimensionless small parameter $\ve=l/L$ represents the ratio between the representative size of a plant cell $l$ and considered size  of a plant tissue $L$ and reflects  the size of the microstructure.  For a plant  root cell we can consider $l = 10\mu$m and $L=1$m and, hence,    $\ve$   is of order  $10^{-5}$.
We consider  $\hat x = L$, $\hat p = \Lambda \ve$, with $\Lambda =1$MPa,  and $\hat u = l$. For the time scale we take    $\hat t = \hat \mu/(\Lambda \ve^2)$,  which together  with  $\hat \mu = 10^{-2}$Pa$\cdot$s  corresponds approximately to $1.7$min.
We also consider   $\hat E = \Lambda $, $\hat K =\hat x^2 \ve/(\hat p \hat t)= l^2/\hat\mu$,
$\hat \rho = (\Lambda  \hat t^2)/ \hat x^2 =\hat \mu^2/( \Lambda \ve^4 L^2)$, $\hat \rho_p = 1/ \Lambda$,
$\hat D= \hat x^2/\hat t= l^2 \Lambda/ \hat \mu$, and $\hat R =\hat x \ve /\hat t= \ve^3 L \Lambda /\hat \mu$.
Hydraulic conductivity $K_p$ is of order  $10^{-9} - 10^{-8}$ m$^2\cdot$ s$^{-1} \cdot $ Pa$^{-1}$ and the minimal value of the elasticity tensor is of order $10$MPa  \cite{Zsivanovits_2004}. Hence the minimal value of the nondimensionalized elasticity tensor ${\bf E}^\ast$ is approximately $10$  and $K_p^\ast\sim 0.01-0.1$.  The  parameters  in the inflow boundary condition, i.e.\ in  $P(b_e)$,  are of order $10^{-7}$m/s,  and with $\hat R= 10^{-7}$m/s we obtain  the nondimensional  parameters in the  boundary condition for $b_e$ to be of order $1$.
Here we assume   that   $\rho_j>0$, with $j=e,p,f$, are fixed.  The case when the density $\rho_e$ and/or $\rho_p$  is of order $\ve^2$ can be analyzed in the same  way as the case when $\rho_e=0$ and $\rho_p=0$, considered in Section~\ref{incompresible}.

To describe the microscopic structure of a plant tissue,  we assume that  cells in the  tissue  are distributed periodically  and have a diameter of the order  $\ve$.
The stochastic case will be analyzed in a future paper. We consider a unit cell $\overline Y= \overline Y_e \cup \overline Y_f$, with   $\overline Y=[0,a_1]\times [0, a_2]\times [0, a_3]$, for $a_j >0$ with $j=1,2,3$,  where $Y_e$ represents the cell wall and  a part  of  middle lamella,
  and $Y_f$, with $\overline Y_f \subset Y$, defines the inner part of a cell.
  We denote  $\partial Y_f = \Gamma$ and let  $\widetilde \Gamma$ be an open on $\Gamma$ regular subset  of  $\Gamma$.

Then the   time-independent  domains  $\Omega_f^\ve$ and $\Omega_e^\ve$, representing the reference (Lagrangian) configuration of the  intracellular (cell inside) and   intercellular  (cell walls and middle lamella) spaces,  are  defined by
\begin{equation}\label{def_ome_ep_f}
\Omega_f^\ve =\text{Int}\big( \bigcup_{\xi \in \Xi^\ve} \ve(\overline Y_f + \xi)\big) \quad \text{ and } \quad
\Omega_e^\ve = \Omega \setminus \overline \Omega_f^\ve,
\end{equation}
respectively, where $\Xi^\ve=\{\xi=(a_1\eta_1, a_2\eta_2, a_3 \eta_3), \; \eta=(\eta_1, \eta_2, \eta_3) \in \mathbb Z^3: \; \ve(\overline Y_f+ \xi) \subset \Omega \}$, and $\Gamma^\ve =  \bigcup_{\xi \in \Xi^\ve} \ve(\Gamma + \xi)$ defines the boundaries between cell inside and cell walls,
$\widetilde \Gamma^\ve =  \bigcup_{\xi \in \Xi^\ve} \ve(\widetilde \Gamma + \xi)$, see Figure~\ref{geom}.

We shall use the following notation for time-space domains: $\Omega_s = (0,s)\times \Omega$,  $(\partial\Omega)_s = (0,s)\times \partial \Omega$, \,   $\Omega_{j,s}^\ve = (0,s)\times \Omega_j^\ve$,  for $j=e,f$,  \,  $\Gamma_s^\ve = (0,s) \times \Gamma^\ve$, and $\widetilde\Gamma_s^\ve = (0,s) \times \widetilde \Gamma^\ve$  for $s\in (0, T]$.

\begin{figure}
\begin{center}
\includegraphics[width=13cm]{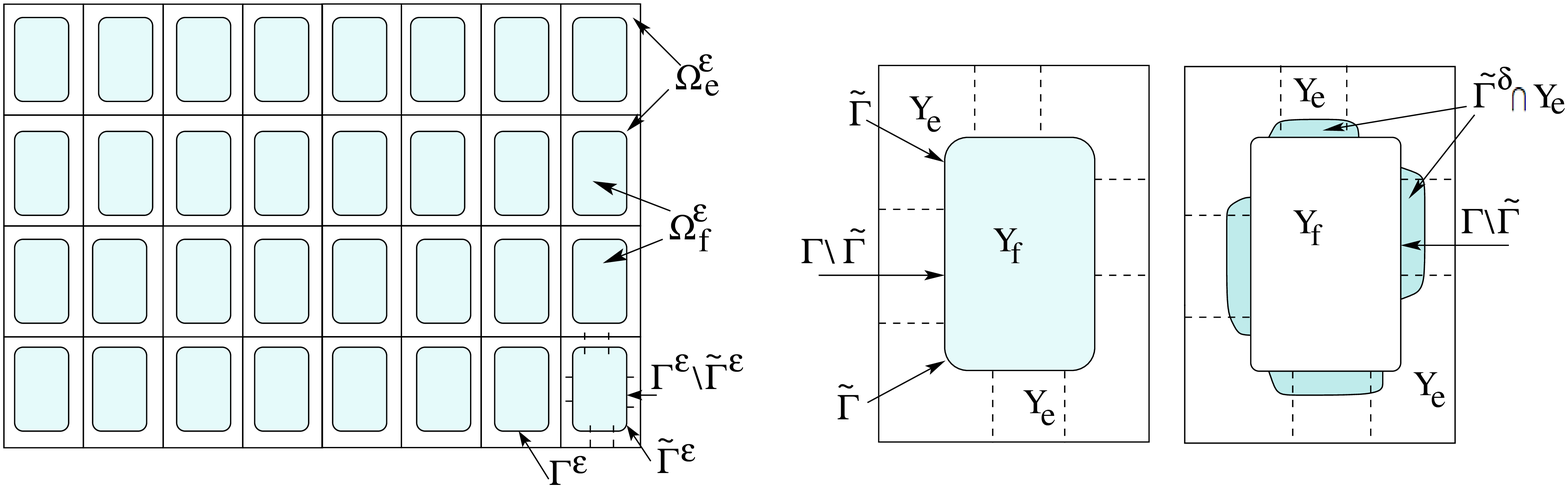}
\caption{Schematic diagram of the geometry of a plant tissue and unit cell}\label{geom}
\end{center}
\end{figure}

Neglecting $\phantom{}^\ast$   we obtain the nondimensonalised microscopic model for plant tissue biomechanics
\begin{equation}\label{eq_codif}
\begin{aligned}
&\partial_t b_e^\ve = {\rm div}(D_b\nabla b^\ve_e) + g_b(c_e^\ve, b_e^\ve, \be(u_e^\ve) )  \qquad &&  \text{ in } \Omega_{e,T}^\ve,\\
&\partial_t c_e^\ve = {\rm div}(D_e\nabla c^\ve_e) + g_e ( c^\ve_e, b_e^\ve, \be(u_e^\ve))  \qquad &&  \text{ in } \Omega_{e,T}^\ve ,\\
&\partial_t c^\ve_f = {\rm div}(D_f\nabla c_f^\ve - \mathcal{G}(\partial_t u_f^\ve) c_f^\ve)  + g_f(c_f^\ve)\qquad &&  \text{ in } \Omega_{f,T}^\ve,\\
&D_b\nabla b_e^\ve \cdot n =\ve \, P(b_e^\ve) && \text{ on }   \Gamma^\ve_T,  \\
&c_e^\ve = c_f^\ve,  \qquad
D _e\nabla c_e^\ve \cdot n = (D_f\nabla c_f^\ve -   \mathcal G(\partial_t u_f^\ve) c_f^\ve)\cdot n  \qquad  && \text{ on } \Gamma^\ve_T\setminus\widetilde \Gamma^\ve_T, \\
&D_e\nabla c_e^\ve \cdot n =0, \qquad (D_f \nabla c_f^\ve- \mathcal G(\partial_t u_f^\ve) c_f^\ve) \cdot n = 0 \quad && \text{ on } \widetilde \Gamma^\ve_T, \\
& b_e^\ve(0,x) = b_{e0}(x), \quad  c_e^\ve(0,x) = c_{0}(x) && \text{ in } \Omega_e^\ve, \\
& c_f^\ve(0,x) = c_{0}(x) && \text{ in } \Omega_f^\ve
\end{aligned}
\end{equation}
and
\begin{equation}\label{equa_cla}
\begin{aligned}
&\rho_e\partial^2_t u^\ve_e- {\rm div} ( {\bf E}^\ve(b^\ve_{e,3}) \be( u^\ve_e)) +  \nabla p_e^\ve  = 0  \qquad &&  \text{ in } \Omega_{e,T}^\ve, \\
&\rho_p\partial_t p^\ve_e- {\rm div} ( K^\ve_p \nabla p^\ve_e -  \partial_t u_e^\ve)  = 0  \qquad &&  \text{ in } \Omega_{e,T}^\ve, \\
& \rho_f\partial^2_t u_f^\ve -\ve^2 \mu\, \text{div} (\be(\partial_t u^\ve_f)) + \nabla p_f^\ve = 0&&  \text{ in } \Omega_{f,T}^\ve, \\
&\text{div}\, \partial_t u_f^\ve = 0 &&  \text{ in } \Omega_{f,T}^\ve , \\
& ({\bf E}^\ve(b^\ve_{e,3}) \, \be( u^\ve_e) - p_e^\ve I)\,  n = ( \ve^2  \mu\, \be(\partial_t u^\ve_f)  - p_f^\ve I)\,  n &&  \text{ on } \Gamma^\ve_T, \\
& \Pi_\tau \partial_t u_e^\ve= \Pi_\tau \partial_t u_f^\ve, \qquad  n \cdot (\ve^2  \mu \, \be(\partial_t u^\ve_f) - p_f^\ve I) \, n  = - p_e^\ve \qquad
  &&  \text{ on } \Gamma^\ve_T, \\
&(-K^\ve_p \nabla p^\ve_e + \partial_t u_e^\ve)\cdot n = \partial_t u_f^\ve\cdot n  &&  \text{ on } \Gamma^\ve_T, \\
& u_e^\ve(0,x) = u^\ve_{e0}(x), \quad  \partial_t u_e^\ve (0,x) = u_{e0}^1(x), \quad   p_e^\ve (0,x) = p^\ve_{e0}(x) && \text{ in } \Omega_e^\ve, \\
 & \partial_t u_f^\ve (0,x) =u^1_{f0}(x) \quad   &&  \text{ in } \Omega_f^\ve.
\end{aligned}
\end{equation}

On the external boundaries  we prescribe the following  force  and flux conditions:
\begin{equation}\label{exbou_co}
\begin{aligned}
&D_b\nabla b_e^\ve \cdot n = F_b(b_e^\ve), \qquad  D_e\nabla c_e^\ve \cdot n = F_c(c_e^\ve) \quad && \text{ on } (\partial \Omega)_T, \\
& {\bf E}^\ve(b^\ve_{e,3}) \be( u^\ve_e) \,  n = F_u  \quad && \text{ on } (\partial \Omega)_T, \\
& (K^\ve_p \nabla p^\ve_e - \partial_t u_e^\ve) \cdot n = F_p  &&  \text{ on }  (\partial \Omega)_T.
\end{aligned}
\end{equation}

The  elasticity  and permeability tensors are defined  by $Y$-periodic functions
$$
{\bf E}^\ve(x, \xi) = {\bf E}(x/\ve, \xi)\quad \hbox{and} \quad K_p^\ve(x) = K_p(x, x/\ve),
$$
where   ${\bf E}(\cdot, \xi)$ and $K_p(x, \cdot)$ are $Y$-periodic for a.a.\ $\xi \in \mathbb R$ and $x \in \Omega$.

 We emphasize that the diffusion coefficients $D_b$, $D_e$, and $D_f$ in   equations  \eqref{eq_codif} are supposed to be constant just for   presentation simplicity. The method developed in this paper also applies to  the case of  non-constant uniformly elliptic diffusion coefficients.\\

\noindent
We suppose the following conditions hold.

\begin{itemize}
\item [\bf A1.]  Elasticity tensor ${\bf E}(y, \zeta) = ( E_{ijkl}(y, \zeta))_{1\leq i,j,k,l\leq 3}$ satisfies  $E_{ijkl} = E_{klij} = E_{jikl} = E_{ijlk}$   and
$\alpha_1 |A|^2  \leq  {\bf E}(y, \zeta) A \cdot A   \leq \alpha_2 |A|^2 $ for  all symmetric matrices $A \in \mathbb R^{3 \times 3}$,   $\zeta \in \mathbb R_+$, and  $y \in Y$,  and  for some $\alpha_1$ and $\alpha_2$ such that $0<\alpha_1 \leq \alpha_2 < \infty$,\\
 ${\bf E}(y, \zeta) ={\bf E}_1(y, \mathcal F(\zeta))$,   where
 $$
  {\bf E}_1\in C_{\text{per}}(Y; C^2_b( \mathbb R))\quad \hbox{and}  \quad \mathcal F(\zeta) = \int_0^t \kappa (t-\tau) \zeta (\tau, x) d \tau,
 $$
with a smooth function $\kappa: \mathbb R_+ \to \mathbb R_+$  such that $\kappa(0)=0$, and $x \in \Omega$.
\item [\bf A2.]  $K_p \in C(\overline \Omega; L^\infty_{\text{per}}(Y))$ and $K_p(x,y)\eta \cdot \eta \geq k_1 |\eta|^2$ for $\eta \in \mathbb R^3$,  a.a.\   $y \in Y$ and $x \in \Omega$,  and $k_1 >0$.

\item [\bf A3.]  $\mathcal{G}$ is a Lipschitz continuous function on $\mathbb R^3$ such that $|\mathcal{G}(r)|\leq R$, for some $R>0$ and all $r \in \mathbb R^3$.
\item [\bf A4.] For functions $g_b$, $g_e$, $g_f$, $P$, $F_b$, and $F_c$ we assume
$$
g_b\in C(\mathbb R\times\mathbb R^3\times \mathbb R^6;\mathbb R^3), \quad
g_e\in C(\mathbb R\times\mathbb R^3\times \mathbb R^6), \quad
F_b,\,P\in C(\mathbb R^3;\mathbb R^3),
$$
and  $F_c$  and $g_f$  are Lipschitz continuous.   Moreover, the following estimates hold:
$$
\begin{aligned}
&|g_b(s,r, \xi )|\leq C_1(1+|s|+|r|) + C_2|r||\xi|,    \\
 &|g_e(s,r,\xi)|\leq C_3(1+|s|+ |r|)+ C_4 (|s|+|r|)|\xi|,
\end{aligned}
$$
$$
\begin{aligned}
& |F_b(r)| + |P(r)| \leq C(1+ |r|),  \\
& |F_c(s)| + |g_f(s)|\leq C(1+|s|),
\end{aligned}
$$

where $s\in\mathbb R_+$, $r \in \mathbb R^3_+$,  and $\xi$ is a symmetric $3\times 3$ matrix. Here and in what follows we identify the space of symmetric $3\times3$ matrices with $\mathbb R^6$.

  It is also assumed that for any symmetric $3\times 3$ matrix $\xi$  we have  that $g_{b,j}(s,r,\xi)$, $F_{b,j}(r)$, and $P_j(r)$  are non-negative for  $r_j=0$,  $s\geq 0$, and  $r_i\geq 0$, with $i=1,2,3$ and $j\neq i$,   and $g_e(s,r,\xi)$, $g_f(s)$, and $F_c(s)$  are non-negative for $s=0$ and $r_j\geq 0$,  with $j=1,2,3$.

We assume also  that $g_b(\cdot, \cdot, \xi)$, $g_e(\cdot, \cdot, \xi)$,  $F_b$ and $P$ are locally Lipschitz continuous   and
$$
\begin{aligned}
&  |g_b(s_1,r_1, \xi_1 ) - g_b(s_2,r_2, \xi_2 )| \leq C_1(|r_1|+|r_2| )|s_1-s_2|  \\ &
 \hspace{ 0.2  cm } + C_2 (|s_1|+ |s_2|+|\xi_1|+|\xi_2|)|r_1-r_2| +  C_3(|r_1|+|r_2| ) |\xi_1-\xi_2|, \\[2mm]
  &  |g_e(s_1,r_1, \xi_1 ) - g_e(s_2,r_2, \xi _2)|\leq C_1(|r_1|+|r_2| + |\xi_1|+ |\xi_2|)|s_1-s_2|  \\
  & \hspace{ 0.1  cm }  +C_2(|s_1|+ |s_2|+|\xi_1| + |\xi_2|)|r_1-r_2| + C_3 (|r_1|+|r_2| + |s_1|+ |s_2|) |\xi_1-\xi_2|,
\end{aligned}
$$
for $s_1, s_2\in\mathbb R_+$, $r_1, r_2 \in \mathbb R^3_+$,  and $\xi, \xi_1, \xi_2$ are symmetric $3\times 3$ matrices.

\item [\bf A5.]
$b_{e0}\in   L^\infty(\Omega)^3$,\,  $c_{0} \in  L^\infty(\Omega)$, \  and $b_{e0,j}\geq 0$, $c_{0}\geq 0$ a.e.\ in $\Omega$,  where  $j=1,2,3$,
\\
$ u_{e0}^1 \in H^1(\Omega)^3$,    $u_{f0}^1\in H^2(\Omega)^3$  and
   ${\rm div }\,  u^{1}_{f0} = 0$  in $\Omega_f^\ve$, \\
$u_{e0}^{\ve} \in H^1(\Omega_e^\ve)^3$,  \,   $p^{\ve}_{e0} \in H^1(\Omega)$,   are defined as solutions of
$$
\begin{aligned}
& {\rm div}({\bf E}^{\ve}(b_{e0,3}) \be( u^{\ve}_{e0})) = f_u  \quad  &&  \text{in } \Omega_e^\ve, &&\\
&\Pi_\tau({\bf E}^{\ve}(b_{e0,3}) \be( u^{\ve}_{e0}) \, n ) = \ve^2 \mu \,  \Pi_\tau ( \be(u_{f0}^1) n) &&  \text{on } \Gamma^\ve,\\
& n \cdot {\bf E}^{\ve}(b_{e0,3}) \be( u^{\ve}_{e0}) \, n  =0&& \text{on } \Gamma^\ve, &&  u^\ve_{e0} = 0 \quad  \text{on } \partial \Omega,
\\
& {\rm div} (K^{\ve}_p \nabla p^{\ve}_{e0}) = f_p \quad && \text{in } \Omega,  &&   p^\ve_{e0} =0 \quad \text{on } \partial \Omega,
 \end{aligned}
$$
for given  $f_u  \in L^2(\Omega)^3$ and $f_p \in L^2(\Omega)$,   \\
$F_p \in H^1(0,T; L^2(\partial \Omega))$, $F_u \in H^2(0,T; L^2(\partial \Omega))^3$.
\end{itemize}

\begin{remark} Under  the  assumptions on $u^\ve_{e0}$ and $p^\ve_{e0}$  by the standard homogenization results   we obtain
$$
\begin{aligned}
&\tilde  u^\ve_{e0} \to u_{e0}, \qquad     p^\ve_{e0} \to  p_{e0} && \text{strongly  in } L^2(\Omega),  \\
& \be(u^\ve_{e0}) \to \be(u_{e0}) + \be_y(\hat u_{e0}) && \text{strongly  two-scale}, \;  \hat u_{e0} \in  L^2(\Omega; H^1(Y_e)/\mathbb R)^3,
\end{aligned}
$$
where $\tilde u_{e0}^\ve$  is an extension of  $u_{e0}^\ve$,  and  $u_{e0}\in H^1(\Omega)^3$ and $p_{e0} \in H^1(\Omega)$ are solutions of the corresponding macroscopic (homogenized) equations.
\end{remark}

\begin{remark}\label{remark_iniosc}
Our approach also applies to the case when the initial velocity $u^1_{f0}$ has the form
$u^{1, \ve}_{f0}(x)=U^1_{f0}(x, x/\varepsilon)$, where the vector function $U^1_{f0}(x,y)$ is periodic in $y$,
sufficiently regular, and such that $\mathrm{div}_x\: U^1_{f0}(x,y)=0,\ \mathrm{div}_y\: U^1_{f0}(x,y)=0$.
\end{remark}

\begin{remark}\label{remark1}
The reaction terms for $c_f^\ve$, $b_{e,1}^\ve$, $b_{e,2}^\ve$, $b_{e,3}^\ve$,  and $c_e^\ve$ can be considered in the following form:
 \begin{equation*}
 \begin{aligned}
 & g_f(c_f^\ve)= - \mu_2 c_f^\ve, \qquad  g_{b,1}(b_e^\ve, c_e^\ve,   {\bf e}(u_e^\ve))= - \mu_1 b_{e,1}^\ve, \\
& g_{b,2}(c^\ve_e, b_e^\ve, {\bf e}(u_e^\ve))=  \mu_1 b_{e,1}^\ve - 2r_{dc} \frac{b_{e,2}^\ve c_e^{\ve}}{\kappa+ c_e^\ve} + 2R_b(b_{e,3}^\ve)({\rm tr} {\bf E}^\ve(b^\ve_{e,3}){\bf e}(u^\ve_e))^{+}- r_d b^\ve_{e,2}, \\
& g_{b,3}(c^\ve_e, b_e^\ve, {\bf e}(u_e^\ve)) = r_{dc} \frac{b_{e,2}^\ve c_e^{\ve}}{\kappa+ c_e^\ve}- R_b(b_{e,3}^\ve)({\rm tr}  {\bf E}^\ve(b^\ve_{e,3}){\bf e}(u^\ve_e))^{+}, \\
& g_e(c^\ve_e, b_e^\ve, {\bf e}({\bf u}_e^\ve))= - r_{dc} \frac{b_{e,2}^\ve c_e^{\ve}}{\kappa+ c_e^\ve} + R_b(b_{e,3}^\ve)({\rm tr}   {\bf E}^\ve(b^\ve_{e,3}){\bf e}(u^\ve_e))^+,
\end{aligned}
\end{equation*}
where $\mu_1,\mu_2,  r_{dc} , r_d , \kappa>0$, and  $R_b(b_{e,3}^\ve)$ is a Lipschitz continuous function of calcium-pectin cross-links density, e.g.\ $R_b(b_{e,3}^\ve)= r_b b_{e,3}^\ve$ with some constant $r_b>0$.
We assume that the concentration of  the enzyme PME  is constant and hence methylesterified pectin is de-esterified at a constant rate.   The demethylesterified pectin is produced through the de-esterification  of methylesterified pectin by PME,  demethylesterified  pectin can decay and through the interaction between two galacturonic acid groups of  pectin chains and a calcium molecule  a calcium-pectin cross-link is produced.   If  a cross-link breaks due to mechanical forces we regain two acid groups of demethylesterified pectin and one calcium molecule.  We consider the decay of calcium  inside the cells.  The positive part of the trace of the elastic stress reflects the fact that extension rather than compression causes the breakage of  calcium-pectin cross-links. See \cite{Ptashnyk} for more details on the derivation of a microscopic  model  for the  biomechanics of a plant cell wall.
\end{remark}

In what follows we  use the notation  $\langle \cdot, \cdot \rangle_{H^1(A)^\prime, H^1}$  for  the duality product between $L^2(0,s; (H^1(A))^\prime)$ and $L^2(0,s; H^1(A))$,  and
$$
\langle \phi, \psi \rangle_{A_s} =  \int_0^s \int_A  \phi \, \psi \, dx dt\quad \text{ for } \quad  \phi \in L^q(0,s; L^p(A)) \text{ and } \psi \in L^{q^\prime}(0,s; L^{p^\prime} (A)),
$$
where $1/q+ 1/q^\prime=1$ and   $1/p + 1/p^\prime =1$,  for any $s>0$ and  domain $A\subset \mathbb R^3$.

We also  use the notation
$$
c^\ve = \left\{ \begin{aligned}
c_e^\ve \quad \text{in } \Omega_{e,T}^\ve,\\
c_f^\ve \quad \text{in } \Omega_{f,T}^\ve.
\end{aligned}
\right.
$$

Next we define a weak solution of the coupled system \eqref{eq_codif}--\eqref{exbou_co}.
\begin{definition}\label{def_weak_micro}
Functions
$$
\begin{aligned}
&u_e^\ve \in \big[ L^2(0,T; H^1(\Omega_e^\ve))\cap H^2(0,T; L^2(\Omega_e^\ve))\big]^3,\\
&p_e^\ve \in L^2(0,T; H^1(\Omega_e^\ve))\cap H^1(0,T; L^2(\Omega_e^\ve)),\\
 &\partial_t u_f^\ve \in \big[L^2(0,T; H^1(\Omega_f^\ve))\cap H^1(0,T; L^2(\Omega_f^\ve))\big]^3,  && \quad  p_f^\ve \in L^2((0,T)\times \Omega_f^\ve),  \\
 & \Pi_\tau \partial_t u_e^\ve=\Pi_\tau \partial_t u_f^\ve \qquad  {\rm on } \; \;  \Gamma_{T}^\ve,  && \quad \hbox{\rm div}\, \partial_t u_f^\ve =0    \quad {\rm in } \; \; \Omega_{f,T}^\ve
 \end{aligned}
 $$
 and
$$
\begin{array}{l}
b_e^\ve\in \big[L^2 (0,T; H^1(\Omega_e^\ve))\cap L^\infty (0,T; L^2(\Omega_e^\ve))\big]^3,\\[2mm]
c^\ve\in L^2 (0,T; H^1(\Omega\setminus \widetilde\Gamma^\ve))\cap L^\infty (0,T; L^2(\Omega))
\end{array}
$$
are a weak solution of \eqref{eq_codif}--\eqref{exbou_co}  if\\
(i)  $(u_e^\ve,  p_e^\ve, \partial_t u_f^\ve, p_f^\ve)$ satisfy the  integral relation:
\begin{equation}\label{weak_u_ef}
\begin{aligned}
&\langle \rho_e\, \partial^2_t u^\ve_e, \phi \rangle_{\Omega_{e,T}^\ve} + \left\langle{\bf E}^\ve(b_{e,3}^\ve) \be( u^\ve_e), \be(\phi) \right\rangle_{\Omega_{e,T}^\ve}  +
\langle \nabla p_e^\ve, \phi \rangle_{\Omega_{e,T}^\ve}
\\
+ &
\langle \rho_p \, \partial_t p^\ve_e, \psi \rangle_{\Omega_{e,T}^\ve} +\left \langle K_p^\ve \nabla p^\ve_e - \partial_t u_e^\ve, \nabla \psi \right\rangle_{\Omega_{e,T}^\ve}  +\langle \partial_t u_f^\ve\cdot n , \psi \rangle_{\Gamma^\ve_T}  - \langle p_e^\ve, \eta\cdot n \rangle_{\Gamma^\ve_T}  \\
+ &  \langle \rho_f\,  \partial^2_t u^\ve_f, \eta \rangle_{\Omega_{f,T}^\ve} +  \ve^2 \mu\left \langle \be(\partial_t u^\ve_f), \be(\eta)\right \rangle_{\Omega_{f,T}^\ve}
 = \langle F_u, \phi \rangle_{(\partial \Omega)_T} +  \langle F_p, \psi \rangle_{(\partial \Omega)_T}
\end{aligned}
\end{equation}
for all  $\psi  \in L^2(0,T; H^1(\Omega_e^\ve))$,   $\phi \in L^2(0,T; H^1(\Omega_e^\ve))^3$,   and $\eta \in   L^2(0,T; H^1(\Omega_f^\ve))^3$,
with $\Pi_\tau\phi=\Pi_\tau\eta$ on $\Gamma^\ve_T$ and  ${\rm div} \eta =0$ in $(0,T)\times\Omega_f^\ve$,
\\
(ii)  $(b_e^\ve,  c^\ve)$ satisfy the  integral relations:
\begin{equation}\label{cd_one}
\begin{aligned}
\langle \partial_t b^\ve_e, \varphi_1 \rangle_{H^1(\Omega_{e}^\ve)^\prime, H^1}+\langle D_b \nabla b^\ve_e, \nabla\varphi_1 \rangle_{\Omega_{e,T}^\ve}
- \left \langle g_b(c^\ve_e,b_e^\ve, \be(u_e^\ve)), \varphi_1 \right\rangle_{\Omega_{e,T}^\ve}\\ = \ve \langle P(b_e^\ve), \varphi_1 \rangle_{\Gamma^\ve_T}  + \langle F_b(b_e^\ve), \varphi_1 \rangle_{(\partial\Omega)_T}
\end{aligned}
\end{equation}
and
\begin{equation}\label{cd_two}
\begin{aligned}
 \langle \partial_t c^\ve_e, \varphi_2 \rangle_{H^1(\Omega_{e}^\ve)^\prime, H^1}+\langle D_e \nabla c^\ve_e, \nabla\varphi_2 \rangle_{\Omega_{e,T}^\ve}
-\left\langle g_e(c^\ve_e,b_e^\ve, \be(u_e^\ve)), \varphi_2\right \rangle_{\Omega_{e,T}^\ve} \qquad \quad \\
 +\langle \partial_t c^\ve_f, \varphi_2 \rangle_{H^1(\Omega_{f}^\ve)^\prime, H^1}+\langle D_f \nabla c^\ve_f
- \mathcal G(\partial_t u^\ve_f) c_f^\ve, \nabla\varphi_2 \rangle_{\Omega_{f,T}^\ve}
-\langle g_f(c^\ve_f), \varphi_2 \rangle_{\Omega_{f,T}^\ve}
\\  = \langle F_c(c_e^\ve), \varphi_2 \rangle_{(\partial\Omega)_T}
\end{aligned}
\end{equation}
for all $\varphi_1\in L^2(0,T; H^1(\Omega^\ve_{e}))^3$ and $\varphi_2\in L^2(0,T; H^1(\Omega\setminus\widetilde\Gamma^\ve))$,
\\
(iii) the corresponding  initial conditions are satisfied. Namely, as $t\to0$, \\
$u_e^\ve(t, \cdot) \to u^\ve_{e0}(\cdot)$ and $\partial_t u_e^\ve(t, \cdot) \to u^1_{e0}(\cdot)$  in $L^2(\Omega^\ve_e)^3$,\,  $p_e^\ve(t, \cdot) \to p^\ve_{e0}(\cdot)$ in $L^2(\Omega^\ve_e)$,\\
 $\partial_t u_f^\ve(t, \cdot) \to u^1_{f0}(\cdot)$  in $L^2(\Omega^\ve_f)^3$, \\
 $b^\ve_e(t, \cdot) \to b_{e0}(\cdot)$ in $L^2(\Omega_e^\ve)^3$, and $c^\ve(t, \cdot) \to c_0(\cdot)$ in $L^2(\Omega)$.
\end{definition}

\section{A priori estimates, existence and uniqueness of  a solution of  the microscopic problem}\label{a_priori}
We begin by proving the existence of a weak solution of the microscopic model  \eqref{eq_codif}--\eqref{exbou_co}  and uniform in $\ve$ {\it a priori} estimates.
In order to obtain uniform in $\ve$ estimates  we shall  extend $H^1$-functions from a perforated domain into the whole domain.
\begin{lemma}\label{extension}
\begin{itemize}
\item There exist extensions $\overline b_e^\ve$ and $\overline c_e^\ve $ of  $b_e^\ve$ and $c_e^\ve$, respectively,  from $L^2(0,T; H^1(\Omega_e^\ve))$  to $L^2(0, T; H^1(\Omega))$ such that
\begin{equation}\label{estim_ext_1}
\| \overline b_e^\ve \|_{L^2(\Omega_T)} \leq  C \|b_e^\ve \|_{L^2(\Omega^\ve_{e,T})}, \quad \| \nabla \overline b_e^\ve \|_{L^2(\Omega_T)} \leq  C \|\nabla b_e^\ve \|_{L^2(\Omega^\ve_{e,T})},
\end{equation}
\begin{equation}\label{estim_ext_11}
\| \overline c_e^\ve \|_{L^2(\Omega_T)} \leq  C \|c_e^\ve \|_{L^2(\Omega^\ve_{e,T})}, \quad \| \nabla \overline c_e^\ve \|_{L^2(\Omega_T)} \leq  C \|\nabla c_e^\ve \|_{L^2(\Omega^\ve_{e,T})}.
\end{equation}
\item There exists an extension $\overline c^\ve $ of  $c^\ve$ from  $L^2(0,T; H^1(\widetilde \Omega_{ef}^\ve))$  to $L^2(0,T; H^1(\Omega))$ such that
\begin{equation}\label{estim_ext_2}
\begin{aligned}
& \| \overline c^\ve \|_{L^2(\Omega_T)} \leq  C \|c^\ve \|_{L^2(\widetilde  \Omega_{ef, T}^\ve)} , \quad  \| \nabla \overline c^\ve \|_{L^2(\Omega_T)} \leq  C \|\nabla c^\ve \|_{L^2(\widetilde \Omega_{ef, T}^\ve)}.
\end{aligned}
\end{equation}
Here  the constant $C$ is independent of $\ve$ and   $\widetilde \Omega^\ve_{ef}=\Omega \setminus \widetilde \Omega^\ve$, with  $\widetilde \Omega^\ve= \bigcup_{\xi \in \Xi^\ve} \ve( \widetilde \Gamma^\delta\cap Y_e+ \xi)$,   where  $\widetilde \Gamma^\delta$ is a $\delta$-neighborhood of $\widetilde \Gamma$ such that   $\widetilde \Gamma^\delta \cap \partial Y = \emptyset$ and $Y \setminus \overline{\widetilde  \Gamma^\delta\cap Y_e}$ is a connected set.
\end{itemize}
\end{lemma}
\begin{proof}[Proof Sketch]
The assumptions on the geometry of $\Omega_e^\ve$ and  $\widetilde \Omega^\ve_{ef}$ and a standard extension operator, see  e.g.\ \cite{Acerbi, CiorPaulin99}, ensure the existence of extensions    of $b_e^\ve$, $c_e^\ve$, and $c^\ve $ satisfying estimates \eqref{estim_ext_1},
\eqref{estim_ext_11}, and \eqref{estim_ext_2}, respectively.
\end{proof}
\noindent
{\bf Remark.} Notice that we have a jump in $c^\ve$ across $\widetilde \Gamma$. Thus in order to construct an extension of $c^\ve$ in  $H^1(\Omega)$ we have to consider $c^\ve$ outside a $\delta$-neighborhood of $\widetilde \Gamma$.  Also since we would like  to have  an extension of $c^\ve_f$ from $\Omega_f^\ve$ to $\Omega$, we have to consider $\widetilde \Gamma^\delta \cap Y_e$, see Figure~\ref{geom}.
\\
 Notice that, since $Y_f \subset Y$ with  $\partial Y_f \cap \partial Y = \emptyset$ and $\Gamma= \partial Y_f$,  for $\delta>0$ sufficiently small $\widetilde \Gamma^\delta$ will satisfy the assumption  of lemma.

\begin{lemma}\label{Lemma:apriori}
Under assumptions {\bf A1.}--{\bf A5.}  solutions of the microscopic problem \eqref{eq_codif}--\eqref{exbou_co} satisfy the following  \textit{a priori} estimates:\\
 for elastic deformation, pressures and  flow velocity  we have
\begin{equation}\label{estim_u_p_u}
\begin{aligned}
&\| u_e^\ve\|_{L^\infty(0,T; H^1(\Omega_e^\ve))} + \|\partial_t u_e^\ve\|_{L^\infty(0,T; H^1(\Omega_e^\ve))} + \| \partial^2_t u_e^\ve \|_{L^\infty(0,T; L^2(\Omega_{e}^\ve))} \leq C , \\
&\| p_e^\ve\|_{L^2(0,T; H^1(\Omega_e^\ve))}+  \| \partial_t p_e^\ve\|_{L^\infty(0,T; L^2(\Omega_e^\ve))} + \| \partial_t p_e^\ve \|_{L^2(0,T; H^1(\Omega_{e}^\ve))} \leq C, \\
& \| \partial_t u_f^\ve\|_{L^\infty(0,T; L^2(\Omega_f^\ve))} + \| \partial^2_t u_f^\ve\|_{L^\infty(0,T; L^2(\Omega_{f}^\ve))}
 + \ve \|\nabla \partial_t u_f^\ve\|_{H^1(0,T; L^2(\Omega_{f}^\ve))} \\ & \hbox{ }\hskip 7.5cm   + \| p_f^\ve \|_{L^2(\Omega_{f,T}^\ve)} \leq C,
\end{aligned}
\end{equation}
 for the  densities   we have
\begin{equation}\label{estim_b_c}
\begin{aligned}
& b_{e, i}^\ve \geq 0, \quad c_e^\ve \geq 0 \quad \text{a.e.\ in }  \Omega_{e,T}^\ve, \quad  c_f^\ve \geq 0 \quad \text{a.e.\  in } \Omega_{f,T}^\ve,  && i =1,2,3, \\
& \| b_e^\ve\|_{L^2(0,T; H^1(\Omega_e^\ve))} +  \ve^{1/2} \| b_e^\ve\|_{L^2(\Gamma_T^\ve)}+ \|b_e^\ve\|_{L^\infty(0,T; L^\infty(\Omega_{e}^\ve))} \leq C , \\
&\| c_j^\ve\|_{L^2(0,T; H^1(\Omega_j^\ve))}  + \| c_j^\ve\|_{L^\infty(0,T; L^2(\Omega_j^\ve))}+ \| c_j^\ve\|_{L^\infty(0,T; L^4(\Omega_j^\ve))} \leq C, \; && j = e,f,
\end{aligned}
\end{equation}
 and
\begin{equation}\label{estim_time_h}
\begin{aligned}
& \| \theta_h b_e^\ve -  b^\ve_e\|_{L^2((0, \tilde T)\times\Omega_{e}^\ve)} +  \| \theta_h c_j^\ve -  c^\ve_j\|_{L^2((0, \tilde T)\times \Omega_{j}^\ve)}  \leq C h^{1/4}, \; \qquad \quad j=e,f,
\end{aligned}
\end{equation}
 for  $\tilde T\in (0,T-h]$, where  $\theta_h v(t,x) = v(t+h, x)$ for  $(t,x) \in (0,T-h]\times \Omega_j^\ve$, with  $j=e,f$,  and  the constant $C$ is independent of $\ve$.
\end{lemma}

\begin{proof}
The non-negativity of $c^\ve_e$, $c_f^\ve$,  and $b_e^\ve$ is justified in the proof of Theorem~\ref{th:exist_uniq_micro} on  the existence and uniqueness of a weak solution of the microscopic problem \eqref{eq_codif}--\eqref{exbou_co}.

To derive the estimates in \eqref{estim_u_p_u}, we first  take $(\partial_t u_e^\ve, \, p_e^\ve, \, \partial_t u_f^\ve)$ as test functions  in \eqref{weak_u_ef}  and obtain
\begin{equation*}
\begin{aligned}
&\rho_e\| \partial_t u^\ve_e(s) \|^2_{L^2(\Omega_{e}^\ve)} + \langle {\bf E}^\ve(b_{e,3}^\ve) \be( u^\ve_e(s)), \be(u^\ve_e(s)) \rangle_{\Omega_{e}^\ve} -\langle\partial_t {\bf E}^\ve(b^\ve_{e,3}) \be(u^\ve_e), \be(u^\ve_e)\rangle_{\Omega_{e, s}^\ve} \\ + &
2 \langle \nabla p_e^\ve,  \partial_t u^\ve_e\rangle_{\Omega_{e, s}^\ve}
+
\rho_p \|p^\ve_e(s) \|^2_{L^2(\Omega_{e}^\ve)}
 + 2 \langle K_p^\ve \nabla p^\ve_e, \nabla p^\ve_e \rangle_{\Omega_{e, s}^\ve}    - 2 \langle  \partial_t u_e^\ve, \nabla p_e^\ve \rangle_{\Omega_{e, s}^\ve} \\
+ &\rho_f  \| \partial_t u^\ve_f(s) \|^2_{L^2(\Omega_{f}^\ve)} + 2  \ve^2  \mu \| \be(\partial_t u^\ve_f)\|^2_{L^2(\Omega_{f, s}^\ve)}
\\ = &2 \langle F_u,   \partial_t u^\ve_e \rangle_{(\partial \Omega)_s} + 2\langle F_p,   p^\ve_e \rangle_{(\partial \Omega)_s}
+   \rho_e\| \partial_t u^\ve_e(0) \|^2_{L^2(\Omega_{e}^\ve)} \\ +& \rho_p\|p^\ve_e(0)\|^2_{L^2(\Omega_e^\ve)} + \rho_f\| \partial_t u^\ve_f(0) \|^2_{L^2(\Omega_{f}^\ve)}+  \langle {\bf E}^\ve(b_{e,3}^\ve) \be( u^\ve_e(0)), \be(u^\ve_e(0)) \rangle_{\Omega_{e}^\ve}
\end{aligned}
\end{equation*}
for $s \in (0, T]$.
As was defined just after formula \eqref{def_ome_ep_f}, $\Omega^\ve_{j,s} := (0,s)\times \Omega_j^\ve$ for $j=e,f$.
Using assumptions {\bf A1.}, {\bf A2.}, and  {\bf A5.}  yields
\begin{equation}\label{estim_uef_p_1}
\begin{aligned}
 \|\partial_t u_e^\ve(s)\|^2_{L^2(\Omega_e^\ve)}  +  \| \be(u_e^\ve(s) )\|^2_{L^2(\Omega_e^\ve)} +  \|\partial_t u_f^\ve(s)\|^2_{L^2(\Omega_f^\ve)}
 + \ve^2 \| \be(\partial_t u_f^\ve)\|^2_{L^2(\Omega_{e,s}^\ve)} \\
+  \|p^\ve_e(s) \|^2_{L^2(\Omega_{f}^\ve)}
+ \|\nabla p_e^\ve\|^2_{L^2(\Omega_{e, s}^\ve)}
\\
 \leq   \delta \big[\|u^\ve_e(s) \|^2_{L^2(\partial \Omega)} +  \|p^\ve_e \|^2_{L^2((0,s)\times\partial \Omega)} \big] + C_1 \langle |\partial_t {\bf E}^\ve(b^\ve_{e,3})| \be(u_e^\ve), \be(u_e^\ve) \rangle_{\Omega_{e,s}^\ve}
 \\
+  C_\delta\big[\| F_u \|^2_{L^\infty(0,s; L^2(\partial \Omega))} + \| \partial_t F_u \|^2_{L^2((0,s)\times\partial \Omega)} + \| F_p \|^2_{L^2((0,s)\times\partial \Omega)}\big]+ C_2
\end{aligned}
\end{equation}
for $s \in (0,T]$. Under our standing assumptions {\bf A1.}  on ${\bf E}$ we have
$$
\|  \partial_t {\bf E}^\ve(b^\ve_{e,3}) \|_{L^\infty((0,T)\times \Omega_e^\ve) }  \leq C.
$$
Applying the trace and Korn inequalities \cite{OShY}  and using extension properties of $u_e^\ve$ we obtain
\begin{equation}\label{bou_u}
\begin{aligned}
\| u_e^\ve(s)\|_{ L^2(\partial\Omega)}  \leq  C \big[ \| u_e^\ve(s)\|_{ L^2(\Omega_e^\ve)} + \| \be(u_e^\ve(s))\|_{ L^2(\Omega_e^\ve)} \big].
\end{aligned}
\end{equation}
Our assumptions {\bf A5.} on the initial conditions ensure
\begin{equation}\label{bou_init_u}
\|u_e^\ve(s) \|_{L^2(\Omega_e^\ve)} \leq  \|\partial_t u_e^\ve\|_{L^2(\Omega_{e,s}^\ve)} + \|u_{e0}^\ve \|_{L^2(\Omega_e^\ve)} \leq C + \|\partial_t u_e^\ve\|_{L^2(\Omega_{e,s}^\ve)}
\end{equation}
for  $s \in (0,T]$. Then applying  the trace and Gronwall inequalities  in \eqref{estim_uef_p_1} yields the following estimate
\begin{equation}\label{estim_upu_22}
\begin{aligned}
& \|\partial_t u_e^\ve\|_{L^\infty(0,T; L^2(\Omega_{e}^\ve))} + \| \be(u_e^\ve)\|_{L^\infty(0, T; L^2(\Omega_e^\ve))} + \|p^\ve_e \|_{L^\infty(0,T; L^2(\Omega_{e}^\ve))}
 \\ & + \|\nabla p_e^\ve\|_{L^2(\Omega_{e,T}^\ve)}
+  \|\partial_t u_f^\ve\|_{L^\infty(0,T; L^2(\Omega_f^\ve))}  + \ve  \| \be(\partial_t u_f^\ve)\|_{L^2(0,T; L^2(\Omega_f^\ve))}
 \leq  C,
\end{aligned}
\end{equation}
where the constant $C$ is independent of $\ve$. Using the Korn inequality \cite{OShY} for  deformation and velocity, together with a scaling argument,  we obtain
\begin{equation}\label{bou_uef_11}
\begin{aligned}
  \| u_e^\ve\|_{L^\infty(0,T; L^2(\Omega_e^\ve))}&+ \|\nabla  u_e^\ve \|_{L^\infty(0, T; L^2(\Omega_e^\ve))}
 \\ &\leq C_1 \big(\| \be(u_e^\ve)\|_{L^\infty(0,T; L^2(\Omega_e^\ve))}  +   \| u_e^\ve\|_{L^\infty(0,T; L^2(\Omega_e^\ve))} \big) \leq C, \\[3mm]
\|\partial_t u_f^\ve\|_{L^2(\Omega_{f,T}^\ve)}  & + \ve  \|\nabla \partial_t u_f^\ve \|_{L^2(\Omega_{f,T}^\ve)}
\\ &\leq C_2 \big( \ve \| \be(\partial_t u_f^\ve)\|_{L^2(\Omega_{f,T}^\ve)}  +   \|\partial_t u_f^\ve\|_{L^2(\Omega_{f,T}^\ve)} \big) \leq C.
\end{aligned}
\end{equation}
Differentiating all  equations in (\ref{equa_cla}) with respect to time $t$ and  taking  $(\partial^2_t u_e^\ve,\,  \partial_t p_e^\ve,\, \partial^2_t u_f^\ve)$ as test functions in the resulting equations,  we obtain
\begin{equation}\label{estim_time_upu}
\begin{aligned}
&\rho_e \| \partial^2_t u^\ve_e(s) \|^2_{L^2(\Omega_{e}^\ve)} +
\langle {\bf E}^\ve(b_{e,3}^\ve) \be( \partial_t u^\ve_e(s)), \be(\partial_t u^\ve_e(s)) \rangle_{\Omega_{e}^\ve}  -  \rho_e \| \partial^2_t u^\ve_e(0) \|^2_{L^2(\Omega_{e}^\ve)}\\ & +
\rho_p \|\partial_t p^\ve_e(s) \|^2_{L^2(\Omega_{e}^\ve)}
 + 2 \langle K_p^\ve \nabla \partial_t p^\ve_e, \nabla \partial_t p^\ve_e \rangle_{\Omega_{e,s}^\ve}-
 \rho_p\|\partial_t p^\ve_e(0) \|^2_{L^2(\Omega_{e}^\ve)}
\\
& + \rho_f \| \partial^2_t u^\ve_f(s) \|^2_{L^2(\Omega_{f}^\ve)} + 2 \,  \ve^2 \mu \| \be(\partial^2_t u^\ve_f)\|^2_{L^2(\Omega_{f, s}^\ve)} - \rho_f\| \partial^2_t u^\ve_f(0) \|^2_{L^2(\Omega_{f}^\ve)}
\\ & =
  \big\langle {\bf E}^\ve(b_{e,3}^\ve(0)) \be(\partial_t u^\ve_e(0)), \be(\partial_t u^\ve_e(0)) \big\rangle_{\Omega_{e}^\ve}
+2 \big\langle \partial_t {\bf E}^\ve(b^\ve_{e,3}(s)) \be( u^\ve_e(s)), \be(\partial_t u^\ve_e(s)) \big\rangle_{\Omega_{e}^\ve}\\
&
-2 \big \langle \partial_t {\bf E}^\ve(b^\ve_{e,3}(0)) \be( u^\ve_e(0)), \be(\partial_t u^\ve_e(0)) \big\rangle_{\Omega_{e}^\ve} + 2 \langle \partial_t F_u,   \partial^2_t u^\ve_e \rangle_{(\partial \Omega)_s} \\
& -\big\langle 2 \partial^2_t {\bf E}^\ve(b_{e,3}^\ve) \, \be( u^\ve_e)+ \partial_t {\bf E}^\ve(b_{e,3}^\ve)\, \be(\partial_t u^\ve_e), \be(\partial_t u^\ve_e)\big \rangle_{\Omega_{e, s}^\ve} +2 \langle \partial_t F_p,   \partial_t p^\ve_e \rangle_{(\partial \Omega)_s}
\end{aligned}
\end{equation}
for  $s\in (0,T]$. Here we used the following equality
\begin{equation*}
\begin{aligned}
\big\langle \partial_t {\bf E}^\ve(b_{e,3}^\ve) \,\be( u^\ve_e), \be(\partial^2_t u^\ve_e)\big \rangle_{\Omega_{e,s}^\ve}  =
\big\langle \partial_t {\bf E}^\ve(b_{e,3}^\ve(s)) \be( u^\ve_e(s)), \be(\partial_t u^\ve_e(s)) \big\rangle_{\Omega_{e}^\ve}
\\
- \big\langle \partial_t {\bf E}^\ve(b_{e,3}^\ve (0)) \be( u^\ve_e(0)), \be(\partial_t u^\ve_e(0)) \big\rangle_{\Omega_{e}^\ve}
\\  -\big\langle \partial^2_t {\bf E}^\ve(b^\ve_{e,3}) \, \be( u^\ve_{e})+  \partial_t {\bf E}^\ve(b^\ve_{e,3})\, \be(\partial_t u^\ve_e), \be(\partial_t u^\ve_e)\big \rangle_{\Omega_{e,s}^\ve} .
\end{aligned}
\end{equation*}
Assumptions {\bf A5.}  on the initial conditions  together with  the microscopic  equations in  \eqref{equa_cla} ensure that
\begin{equation}\label{init_estim_h}
\begin{aligned}
&\| \partial^2_t u^\ve_e(0) \|^2_{L^2(\Omega_{e}^\ve)}+  \|\partial_t p^\ve_e(0) \|^2_{L^2(\Omega_{e}^\ve)}+ \|\partial^2_t u^\ve_f(0) \|^2_{L^2(\Omega_{f}^\ve)}
\leq  C,
\end{aligned}
\end{equation}
where   the constant $C$ is independent of $\ve$. To justify \eqref{init_estim_h}, first we consider   the Galerkin approximations of $u^\ve_e$ and $\partial_t u^\ve_f$ and a function $\phi^k$ in the corresponding finite dimensional   subspace, with   $\phi^k = 0$ on $\partial \Omega$ and ${\rm div } \, \phi^k = 0$ in $\Omega_f^\ve$,
$$
\begin{aligned}
&\langle \rho_e\, \partial^2_t u^{\ve,k}_e, \phi^k \rangle_{\Omega_{e}^\ve} + \left\langle{\bf E}^\ve(b_{e,3}^\ve) \be( u^{\ve,k}_e), \be(\phi^k) \right\rangle_{\Omega_{e}^\ve}  +
\langle \nabla p_e^{\ve, k}, \phi^k \rangle_{\Omega_{e}^\ve} \\
& +
\langle \rho_f \, \partial^2_t u_f^{\ve, k}, \phi^k \rangle_{\Omega_f^\ve}   +\ve^2 \mu\, \langle  \be(\partial_t u^{\ve, k}_f), \be(\phi^k)\rangle_{\Omega_f^\ve}
+
\langle p_e^{\ve,k} , \phi^k \cdot n \rangle_{\Gamma^\ve}  = 0.
\end{aligned}
$$
Taking $t\to 0$ and using the regularity of  $u^{\ve,k}_e$,  $\partial_t u_f^{\ve, k}$,  and $b_{e,3}^\ve$ with respect to the time variable  we obtain
$$
\begin{aligned}
&\langle \rho_e\, \partial^2_t u^{\ve,k}_e(0), \phi^k \rangle_{\Omega_{e}^\ve} + \left\langle{\bf E}^\ve(b_{e0,3}) \be( u^{\ve,k}_e(0)), \be(\phi^k) \right\rangle_{\Omega_{e}^\ve}  +
\langle \nabla p_e^{\ve, k}(0), \phi^k \rangle_{\Omega_{e}^\ve} \\
& +
\langle \rho_f \, \partial^2_t u_f^{\ve, k}(0), \phi^k \rangle_{\Omega_f^\ve}   +\ve^2 \mu\, \langle  \be(\partial_t u^{\ve, k}_f(0)), \be(\phi^k)\rangle_{\Omega_f^\ve}
+
\langle p_e^{\ve,k}(0) , \phi^k \cdot n \rangle_{\Gamma^\ve}  = 0.
\end{aligned}
$$
Then the integration by parts in the last two terms and   the assumptions on the initial values ensure
$$
\begin{aligned}
&|\langle  \partial^2_t u^{\ve,k}_e(0), \phi^k \rangle_{\Omega_{e}^\ve} |  + |\langle \partial^2_t u_f^{\ve, k}(0), \phi^k \rangle_{\Omega_f^\ve}|
\leq  |\left\langle f_u,  \phi^k  \right\rangle_{\Omega_{e}^\ve}|  +
|\langle \nabla p_{e0}^{\ve, k}, \phi^k \rangle_{\Omega}| \\
 &+  \ve^2 \mu\, | \langle  {\rm div }\,  \be(\partial_t u^{1,k}_{f0}), \phi^k\rangle_{\Omega_f^\ve}| + \ve^2 \mu\, \| \nabla^2 \partial_t u^{1,k}_{f0}\|_{L^2(\Omega)}\|\phi^k\|_{L^2(\Omega_f^\ve)}
\leq C \|\phi^k\|_{L^2(\Omega)},
\end{aligned}
$$
and hence
$$
\| \partial^2_t u^{\ve,k}_e(0) \|_{L^2(\Omega_e^\ve)} + \| \partial^2_t u_f^{\ve, k}(0)\|_{L^2(\Omega_f^\ve)} \leq C,
$$
where the constant $C$ is independent of $k$ and ${\rm div}\,  \partial^2_t u_f^{\ve, k}(0) =0$ in $\Omega_f^\ve$.  In a similar way we also obtain  the boundedness of $\|\partial_t p^{\ve, k}_e(0)\|_{L^2(\Omega_e^\ve)}$ uniformly in $k$.

Then the estimates similar to \eqref{estim_time_upu} for the Galerkin approximations of $u_e^\ve$, $p_e^\ve$, and $\partial_t u_f^\ve$ imply that  $p^\ve_e \in C([0,T]; L^2(\Omega^\ve_e))$,
  $\nabla p^\ve_e, \partial_t u^\ve_e  \in C([0,T]; L^2(\Omega^\ve_e))^3$,  $\be(u_e^\ve) \in C([0,T]; L^2(\Omega^\ve_e))^{3\times 3}$,  $\partial_t u^\ve_f \in C([0,T]; L^2(\Omega^\ve_f))^3$,
${\bf e}(\partial_t u_f^\ve) \in C([0,T]; L^2(\Omega^\ve_f))^{3\times 3}$.

Then from the equations for $u_e^\ve$  and $p_e^\ve$ and the continuity of  $\be(u_e^\ve)$, $\partial_t u^\ve_e$, and  $\nabla p^\ve_e$  with respect to the time variable we obtain the continuity of $\partial^2_t u^\ve_e$ and  $\partial_t p_e^\ve$ with respect to  the time variable. Then the assumptions on   $u_{e0}^\ve$, $u^1_{e0}$,  and $p_{e0}^\ve$ ensure the boundedness of $\| \partial^2_t u^\ve_e(0) \|_{L^2(\Omega_{e}^\ve)}$ and   $\|\partial_t p^\ve_e(0) \|_{L^2(\Omega_{e}^\ve)}$ uniformly in $\ve$.

 For $\phi \in H^1_0(\Omega)$, with ${\rm div} \, \phi =0$ in $\Omega_f^\ve$, we have
$$
\begin{aligned}
&\langle \rho_e\, \partial^2_t u^\ve_e, \phi \rangle_{\Omega_{e}^\ve} + \left\langle{\bf E}^\ve(b_{e,3}^\ve) \be( u^\ve_e), \be(\phi) \right\rangle_{\Omega_{e}^\ve}  +
\langle \nabla p_e^\ve, \phi \rangle_{\Omega_{e}^\ve} \\
& +
\langle \rho_f\partial^2_t u_f^\ve, \phi \rangle_{\Omega_f^\ve}   +\ve^2 \mu\, \langle  \be(\partial_t u^\ve_f), \be(\phi)\rangle_{\Omega_f^\ve}
+
\langle p_e^\ve , \phi \cdot n \rangle_{\Gamma^\ve}  = 0
\end{aligned}
$$

Considering the continuity of  $\be(u_e^\ve)$, $\partial_t^2 u_e^\ve$, $\nabla p^\ve_e$ and $\be(\partial_t u_f^\ve)$  with respect to the time variable and taking $t \to 0$  we obtain the continuity of $\partial^2_t u^\ve_f$ and
$$
\begin{aligned}
&\langle \rho_e\, \partial^2_t u^\ve_e(0), \phi \rangle_{\Omega_{e}^\ve} + \left\langle{\bf E}^\ve(b_{e0,3}) \be( u^\ve_{e0}), \be(\phi) \right\rangle_{\Omega_{e}^\ve}  +
\langle \nabla p_{e0}^\ve, \phi \rangle_{\Omega_{e}^\ve} \\
&+ \langle \rho_f\, \partial^2_t u_f^\ve(0), \phi \rangle_{\Omega_f^\ve}  +\ve^2 \mu\, \langle  \be(u_{f0}^1), \be(\phi)\rangle_{\Omega_f^\ve}
+
\langle p_{e0}^\ve , \phi \cdot n \rangle_{\Gamma^\ve}  = 0.
\end{aligned}
$$
The integration by parts, the boundary conditions for  $u_{e0}^\ve$,  and the assumptions on $\phi$ imply
$$
\begin{aligned}
\langle \rho_f\partial^2_t u_f^\ve(0), \phi \rangle_{\Omega_f^\ve}  =
- \langle \rho_e\, \partial^2_t u^\ve_e(0), \phi \rangle_{\Omega_{e}^\ve}
+ \left\langle {\rm div}({\bf E}^\ve(b_{e0,3}) \be( u^\ve_{e0})), \phi \right\rangle_{\Omega_{e}^\ve}  -
\langle \nabla p_{e0}^\ve, \phi \rangle_{\Omega}\\ +
\langle   \ve^2   \mu \,  {\rm div}  (\be(u_{f0}^1)), \phi\rangle_{\Omega_f^\ve}
-
\ve^2 \langle \mu \, n\cdot \be(u_{f0}^1) n, \phi \cdot n \rangle_{\Gamma^\ve} .
\end{aligned}
$$
From the assumptions on  $u_{e0}^\ve$ and  $p_{e0}^\ve$   we have   that   ${\rm div}({\bf E}^\ve(b_{e0,3}) \be( u^\ve_{e0}))
= f_u$, with $f_u \in L^2(\Omega)$, and  $\|\nabla p_{e0}^\ve \|_{L^2(\Omega)}
\leq C_1$, where $C_1$ is independent of $\ve$. 
\\
The assumptions on $u_{f0}^1$ ensure that  $\ve^2  \mu\| {\rm div}  (\be(u_{f0}^{1}))\|_{L^2(\Omega_f^\ve)} \leq C_2$ and  there exists $\psi^\ve \in H^1(\Omega_f^\ve)$,   such that  $\|\nabla\psi^\ve\|_{L^2(\Omega_f^\ve)} \leq C_3$ and
$$
\begin{aligned}
|\ve^2 \langle \mu \, n\cdot \be(\partial_t u_{f0}^1) n, \phi \cdot n \rangle_{\Gamma^\ve} | =
| \langle \nabla \psi^\ve, \phi  \rangle_{\Omega_f^\ve}| \leq C_4 \|\phi\|_{L^2(\Omega_f^\ve)},
\end{aligned}
$$
where the constants $C_2$, $C_3$, and $C_4$ are  independent of $\ve$.   Using the density of $\phi$ in $\mathcal H = \{v\in L^2(\Omega_f^\ve) : \, {\rm div} \, v = 0 \; \text{ in } \Omega_f^\ve \}$ we obtain the  boundedness of $\partial^2_t u_f^\ve(0)$  in  $\mathcal H$ uniformly in $\ve$.

Then considering assumptions  {\bf A1.}--{\bf A2.} and   applying the H\"older and Gronwall inequalities in \eqref{estim_time_upu},  we obtain the estimates for $\partial^2_t u_e^\ve$,  $\partial_t p^\ve_e$ and $\partial^2_t u_f^\ve$ stated in \eqref{estim_u_p_u}.
Here we used the fact that assumptions {\bf A1.} on ${\bf E}$ imply the following upper bound:
$$
\|  \partial^2_t {\bf E}^\ve(b^\ve_{e,3}) \|_{L^\infty((0,T)\times \Omega_e^\ve) }  \leq C.
$$
Testing the first and third equations  in \eqref{equa_cla}  with $\phi \in L^2(0,T; H^1(\Omega))^3$  and using the a priori estimates for $u_e^\ve$,  $p_e^\ve$, and $\partial_t u_f^\ve$, we obtain
\begin{equation}\label{estim_pf}
\begin{aligned}
\langle p_f^\ve, {\rm div} \, \phi \rangle_{\Omega_{f,T}^\ve} +  \langle p_e^\ve, {\rm div }\, \phi   \rangle_{\Omega_{e, T}^\ve} = \langle \ve^2 \mu \, \be(\partial_t u_f^\ve),  \be(\phi)  \rangle_{\Omega_{f,T}^\ve} +  \rho_f\langle \partial^2_t u_f^\ve, \phi\rangle_{\Omega_{f, T}^\ve}
\\
 +\rho_e \langle \partial^2_t u_e^\ve, \phi\rangle_{\Omega_{e, T}^\ve}
+  \langle {\bf E}^\ve(b^\ve_{e,3}) \be(u_e^\ve), \be(\phi) \rangle_{\Omega_{e, T}^\ve}+\langle p_e^\ve\,  n - F_u, \phi \rangle_{(\partial \Omega)_T} \\
\leq C \|\phi\|_{L^2(0,T; H^1(\Omega))^3}.
\end{aligned}
\end{equation}
Here we used the properties of an extension of $p_e^\ve$  from $\Omega_e^\ve$ to $\Omega$, see Lemma~\ref{extension},  and  the trace estimate
$\|p_e^\ve\|_{L^2((0,T)\times\partial\Omega)}\leq C_1 \|p_e^\ve\|_{L^2(0,T; H^1(\Omega))} \leq C_2 \|p_e^\ve\|_{L^2(0,T; H^1(\Omega_e^\ve))}$.

For any $q \in L^2(\Omega_T)$ there exists $\phi \in  L^2(0,T; H^1(\Omega))^3$ satisfying ${\rm div} \, \phi =q$ in $\Omega$ and $\phi \cdot n = \frac 1{|\partial \Omega|}\int_\Omega q(\cdot,x) dx $ on $\partial \Omega$ and $\| \phi\|_{L^2(0,T; H^1(\Omega))^3}\leq C \|q\|_{L^2(\Omega_T)}$.
Thus for
$$
\widetilde p^\ve =\begin{cases}  p_f^\ve & \text{ in } \; \;  (0,T)\times  \Omega_f^\ve \\
p_e^\ve & \text{ in } \; \;  (0,T)\times(  \Omega \setminus  \Omega_f^\ve)
\end{cases}
$$
 using \eqref{estim_pf}  we obtain
$$
\langle \widetilde p^\ve,  q \rangle_{\Omega_{T}} \leq C\|q\|_{L^2((0,T)\times \Omega)},
$$
where  the constant $C$ is independent of $\ve$.  This implies, by the definition of the $L^2$-norm and the estimates for $p^\ve_e$, that $\|p_f^\ve\|_{L^2((0,T)\times\Omega_f^\ve)} \leq C$.

To  justify estimates  \eqref{estim_b_c}  we take $b_e^\ve$ and $c^\ve$ as  test functions in \eqref{cd_one} and  \eqref{cd_two}, respectively. Using   assumptions {\bf A3.}--{\bf A5.}, we obtain
\begin{equation*}
\begin{aligned}
\| b_e^\ve(s)&\|^2_{L^2(\Omega_e^\ve)}  + \| \nabla b_e^\ve\|^2_{L^2(\Omega_{e,s}^\ve)}
\\ \leq & \| b_e^\ve(0)\|^2_{L^2(\Omega_e^\ve)} + C_1\|\be(u_e^\ve)\|_{L^\infty(0,s; L^2(\Omega_{e}^\ve))} \|b_e^\ve\|^2_{L^2(0,s; L^4(\Omega_{e}^\ve))}  \\
+ & C_2\big[\|c_e^\ve\|^2_{L^2(\Omega_{e, s}^\ve)}  +  \|b_e^\ve\|^2_{L^2(\Omega_{e, s}^\ve)}\big]+ C_3 \big[1+ \ve \|b_e^\ve\|^2_{L^2(\Gamma_s^\ve)}   + \| b_e^\ve\|^2_{L^2((0,s)\times \partial \Omega)} \big]
\end{aligned}
\end{equation*}
and
\begin{equation*}
\begin{aligned}
 \| c_e^\ve(s)&\|^2_{L^2(\Omega_e^\ve)} + \| c_f^\ve(s)\|^2_{L^2(\Omega_f^\ve)}   + \| \nabla c_e^\ve\|^2_{L^2(\Omega_{e,s}^\ve)}+  \| \nabla c_f^\ve\|^2_{L^2(\Omega_{f, s}^\ve)}
\\ \leq & \| c_e^\ve(0)\|^2_{L^2(\Omega_e^\ve)}   +   \| c_f^\ve(0)\|^2_{L^2(\Omega_f^\ve)}\\
+ & C_1\|\be(u_e^\ve)\|_{L^\infty(0,s; L^2(\Omega_{e}^\ve))}\big[\|b_e^\ve\|^2_{L^2(0,s; L^4(\Omega_{e}^\ve))}
+ \|c_e^\ve\|^2_{L^2(0,s; L^4(\Omega_{e}^\ve))} \big]
\\  + & C_2 \big[1+  \|\mathcal G(\partial_t u_f^\ve)\|^2_{L^\infty(\Omega_{f,s}^\ve)}\big] \|c_f^\ve\|^2_{L^2(\Omega_{f, s}^\ve)}
\\ + & C_3\big[\|b_e^\ve\|^2_{L^2(\Omega_{e, s}^\ve)}+  \| c_e^\ve\|^2_{L^2(\Omega_{e, s}^\ve)} +  \| c_e^\ve\|^2_{L^2((0,s)\times\partial \Omega)} \big] .
\end{aligned}
\end{equation*}
The  Gagliardo-Nirenberg and trace  inequalities,  together with the extension properties of $b^\ve_e$ and $c^\ve$, see Lemma \ref{extension},  yield
\begin{equation}\label{estim_GN_trace}
\begin{aligned}
& \|b_e^{\ve}\|^2_{L^4(\Omega_e^\ve)}   \leq \|b_e^{\ve}\|^2_{L^4(\Omega)}  \leq  \delta_1 \|\nabla  b_e^{\ve}\|^2_{L^2(\Omega)}
+ C_{\delta_1} \| b_e^{\ve}\|^2_{L^2(\Omega)} \\ & \hskip 3.4 cm \leq    \delta_2 \|\nabla  b_e^{\ve}\|^2_{L^2(\Omega_e^\ve)}
+ C_{\delta_2} \| b_e^{\ve}\|^2_{L^2(\Omega_e^\ve)} ,
\\
& \|c_e^{\ve}\|^2_{L^4(\Omega_e^\ve)} +  \|c_f^{\ve}\|^2_{L^4(\Omega_f^\ve)}  \leq \delta\left[ \|\nabla  c_e^{\ve}\|^2_{L^2(\Omega_e^\ve)}+  \|\nabla  c_f^{\ve}\|^2_{L^2(\Omega_f^\ve)} \right]
\\
& \hskip 3.6cm + C_\delta \left[\|c_e^{\ve}\|^2_{L^2(\Omega_e^\ve)} +  \|c_f^{\ve}\|^2_{L^2(\Omega_f^\ve)}\right],\\
& \|b_e^{\ve}\|^2_{L^2(\partial\Omega)}   \leq \delta \|\nabla  b_e^{\ve}\|^2_{L^2(\Omega_e^\ve)}
+ C_\delta \| b_e^{\ve}\|^2_{L^2(\Omega_e^\ve)} , \quad\\
&  \|c_e^{\ve}\|^2_{L^2(\partial\Omega)}   \leq \delta \|\nabla  c_e^{\ve}\|^2_{L^2(\Omega_e^\ve)}
+ C_\delta \| c_e^{\ve}\|^2_{L^2(\Omega_e^\ve)} , \\
&\ve \|b_e^{\ve}\|^2_{L^2(\Gamma^\ve)}   \leq C\big[\ve^2 \|\nabla  b_e^{\ve}\|^2_{L^2(\Omega_e^\ve)}
+  \| b_e^{\ve}\|^2_{L^2(\Omega_e^\ve)} \big],
 \end{aligned}
\end{equation}
for an arbitrary $\delta >0$,  and $C_\delta$ depending on $\delta$ and independent of $\ve$. Notice that since the Gagliardo-Nirenberg inequality is applied to the extension of  $b^\ve_e$ and $c^\ve$ defined in $\Omega$, the constant in the Gagliardo-Nirenberg inequality  is independent of $\ve$.  Then applying the Gronwall inequality and using the assumptions {\bf A3.} on $\mathcal G$  yields
\begin{equation}\label{bou_b_c}
\begin{aligned}
&\| b_e^\ve\|_{L^\infty(0,T; L^2(\Omega_e^\ve))}  + \| \nabla b_e^\ve\|_{L^2((0,T) \times \Omega_e^\ve)}  \leq C,\\
& \| c_j^\ve\|_{L^\infty(0,T; L^2(\Omega_j^\ve))}  + \| \nabla c_j^\ve\|_{L^2((0,T)\times \Omega_j^\ve)}   \leq C,  \quad j = e,f.
\end{aligned}
\end{equation}
The uniform boundedness of $b_e^\ve$, i.e.\
\begin{equation}\label{estim_Linfty_b_e}
\|b_e^\ve\|_{L^\infty(0,T; L^\infty(\Omega_e^\ve))} \leq C,
\end{equation}
with a constant $C$ independent of $\ve$,  is proved by applying  the Alikakos iteration Lemma \cite[Lemma 3.2]{Alikakos}.
Since the derivation of estimate \eqref{estim_Linfty_b_e} is rather involved, we present the detailed proof of this estimate in Appendix, see Lemma~\ref{Linf_L4}.
In the same Lemma in Appendix we also prove the estimate
 $$
 \|c^\ve_e\|_{L^\infty(0,T; L^4(\Omega_e^\ve))} + \| c_f^\ve\|_{L^\infty(0,T; L^4(\Omega_f^\ve))} \leq C,
 $$
where the constant $C$ does not depend on $\ve$.

To justify the last estimate  (\ref{estim_time_h}), we integrate  the equation for $b_e^\ve$ in \eqref{eq_codif} over $(t, t+h)$ and consider $\theta_h b^\ve_e - b^\ve_e $  as  a test function:
 \begin{equation*}
\begin{aligned}
 \| \theta_h b^\ve_e & - b^\ve_e \|^2_{L^2((0, \tilde T)\times \Omega_e^\ve)} +
 \Big\langle D_b\int_t^{t+h} \nabla b_e^\ve(s)  \, ds,  \nabla (\theta_h b_e^\ve) -\nabla b_e^\ve \Big \rangle_{(0, \tilde T)\times \Omega_{e}^\ve} \\
 & =
 \Big \langle \int_t^{t+h} g_b( b_e^\ve(s), c_e^\ve(s), \be(u_e^\ve(s))) \, ds,  \theta_h b_e^\ve - b_e^\ve   \Big\rangle_{(0, \tilde T)\times \Omega_{e}^\ve}
 \\ & + \ve \Big  \langle \int_t^{t+h} P( b_e^\ve(s)) \, ds,  \theta_h b_e^\ve- b_e^\ve  \Big\rangle_{(0,\tilde T)\times \Gamma^\ve}\\
&  + \Big  \langle \int_t^{t+h} F_b(b_e^\ve) \, ds,  \theta_h b_e^\ve - b_e^\ve  \Big\rangle_{(0, \tilde T)\times\partial \Omega}
\end{aligned}
\end{equation*}
for all $\tilde T \in (0, T-h]$. Then using  the a priori estimates for $u_e^\ve$, $b_e^\ve$ and $c^\ve_e$ in \eqref{estim_u_p_u} and \eqref{estim_b_c}  together with the H\"older inequality   implies the estimate for $b^\ve_e(t+h,x) - b^\ve_e(t,x)$.
Similar calculations yield the estimates for $c^\ve_e(t+h,x) - c^\ve_e(t,x)$ and $c^\ve_f(t+h,x) - c^\ve_f(t,x)$.
\end{proof}

\begin{theorem}\label{th:exist_uniq_micro}
Under assumptions {\bf A1.}--{\bf A5.} for every $\ve>0$ there exists a unique weak solution of the coupled problem \eqref{eq_codif}--\eqref{exbou_co}.
\end{theorem}

\begin{proof}
We shall use a  contraction argument to show the existence of a solution of the coupled system.
We consider an operator $\mathcal K$ over $L^\infty(0,s; H^1(\Omega_e^\ve)^3)\times L^\infty(0,s; L^2(\Omega_f^\ve)^3)$ defined by  $(u^{\ve,j}_{e}, \partial_t u^{\ve,j}_{f})  = \mathcal K (u_{e}^{\ve,j-1}, \partial_t u^{\ve,j-1}_{f})$, where for given $(u_{e}^{\ve,j-1},\partial_tu_{f}^{\ve,j-1})$ we first define $(b^{\ve,j}_{e}, c^{\ve,j}_{e}, c^{\ve,j}_{f})$ as a solution of  system \eqref{eq_codif} with functions $(u_{e}^{\ve,j-1},\partial_tu_{f}^{\ve,j-1})$ in place of
$(u_{e}^\ve,\partial_tu_{f}^\ve)$ and with
external boundary conditions in \eqref{exbou_co}, and then  $(u^{\ve,j}_{e}, p_e^{\ve,j}, \partial_t u_{f}^{\ve,j}, p_f^{\ve, j})$ are  solutions of \eqref{equa_cla} with $b_{e}^{\ve,j}$ in place of
$b_{e}^\ve$.

For each $j=2,3,\ldots$, the proof of  existence and uniqueness of $(b_{e}^{\ve,j}, c_{e}^{\ve,j}, c_{f}^{\ve,j})$ for given $(u_{e}^{\ve,j-1}, \partial_t u^{\ve,j-1}_{f})$ follows the same arguments (with a number of simplifications)  as the proof that $\mathcal K$ is a contraction for $(u^{\ve,j}_{e}, \partial_t u^{\ve,j}_{f})$, i.e.\ using  the Galerkin method and fixed point arguments.  Notice that the fixed point  argument  for the system for $b_{e}^{\ve,j}$ and  $c^{\ve,j}$ allows us to consider the equations for $b_e^{\ve,j}$ and $c^{\ve,j}$ recursively. Thus using the  non-negativity of initial data $b_{e0}$, $c_{e0}$, and $c_{f0}$ and   assumptions  \textbf{A4.}   on   the reaction   and boundary terms and   applying iteratively   the Theorem on positively invariant regions \cite{Redlinger,Smoller}  we obtain the non-negativity of   all components of $b^{\ve,j}_{e}$ and $c^{\ve,j}$.

We choose the first iteration $(u^{\ve,1}_{e}, p_e^{\ve,1}, \partial_t u^{\ve,1}_{f}, p_f^{\ve, 1})$ to satisfy the initial and boundary conditions in \eqref{equa_cla} and  \eqref{exbou_co}.   Then applying the Galerkin method  (using the basis functions for $H^1(\Omega_e^\ve)\times H^1(\Omega\setminus \widetilde \Gamma^\ve)$)  and    fixed-point argument we obtain the existence of solutions $(b_e^{\ve,2}, c_e^{\ve,2}, c_f^{\ve,2})$ of  system \eqref{eq_codif} with
external boundary conditions in \eqref{exbou_co} and  have
\begin{equation}\label{bou_b_c_itera}
\begin{aligned}
&\| b_e^{\ve,2}\|_{L^\infty(0,T; L^2(\Omega_e^\ve))}  + \| \nabla b_e^{\ve,2}\|_{L^2( \Omega_{e,T}^\ve)} + \| b_e^{\ve,2}\|_{L^\infty(0,T;  L^\infty(\Omega_e^\ve))}  \leq C,\\
& \| c_l^{\ve,2}\|_{L^\infty(0,T; L^2(\Omega_l^\ve))}  + \| \nabla c_l^{\ve,2}\|_{L^2(\Omega_{l,T}^\ve)}  +  \| c_l^{\ve,2}\|_{L^\infty(0,T; L^4(\Omega_l^\ve))} \leq C,  \quad l = e,f,
\end{aligned}
\end{equation}
where the constant $C$  depends only on $\|\be(u^{\ve,1}_e)\|_{L^\infty(0,T;L^2(\Omega^\ve_e))}$ and the constants in assumptions {\bf A1.}--{\bf A4.}
The estimates  \eqref{bou_b_c_itera} can be justified in the same way as those in \eqref{estim_b_c}.

Next we consider   system \eqref{equa_cla} with $b_{e}^{\ve,2}$ in place of $b_{e}^\ve$.  To show the existence result we use  the Galerkin method with the basis functions $\{\phi_j, \psi_j, \eta_j\}_{j\in \mathbb N}$ for  the space
$$
\begin{aligned}
W= \{(v  ,p,w) \in  H^1(\Omega_e^\ve)^3\times H^1(\Omega_e^\ve)\times H^1(\Omega_f^\ve)^3 : \;  \;  {\rm div}\, w =0 \; \text{ in } \; \Omega_f^\ve,
\\
 \Pi_\tau v= \Pi_\tau w \;   \text{ on }\;  \Gamma^\ve, \;  \;   \; {\rm div}(K^\ve_p \nabla p)  \in L^2(\Omega_e^\ve),
 \\
 \langle (v-K^\ve_p \nabla p -w ) \cdot n, \psi \rangle_{H^{-1/2}(\Gamma^\ve), H^{1/2}(\Gamma^\ve)}  = 0 \},
\end{aligned}
$$
and consider the approximate solutions in the form
$$
u^{\ve,2}_{e,k}= \sum_{j=1}^k q^k_j(t) \phi_j, \quad p^{\ve,2}_{e,k}= \sum_{j=1}^k \frac d{dt} q^k_j(t) \psi_j, \quad \partial_t u^{\ve,2}_{f,k}= \sum_{j=1}^k \frac d{dt} q^k_j(t) \eta_j, \quad k \in \mathbb N.
$$
The linearity of equations for $\big(u_{e}^{\ve,2}, \,p_e^{\ve,2},\,\partial_t u^{\ve,2}_{f}\big)$ ensures the existence of  unique solutions $q^k_j(t)$ of the corresponding linear system  of second order ordinary differential equations with initial conditions $q^k_j(0) = \alpha_j^k$ and $\dfrac d {dt} q^k_j(0) = \beta_j^k$, where $\alpha_j^k$ and $\beta_j^k$ are derived from the initial conditions in  \eqref{equa_cla},  and hence,  the existence of a unique solution $\big(u^{\ve,2}_{e,k},  p^{\ve,2}_{e,k}, \partial_t u^{\ve,2}_{f,k}\big)$ for $k \in \mathbb N$.
Then using  the  a priori estimates derived in the same way as in Lemma \ref{Lemma:apriori} (by considering assumptions {\bf A1.},  {\bf A2.}, and {\bf A5.}),  and taking the limit as $k \to \infty$ we obtain the existence of  $u_{e}^{\ve,2} \in [H^1(0,T; H^1(\Omega_e^\ve))\cap H^2(0,T; L^2(\Omega_e^\ve))]^3$,  $p_e^{\ve,2} \in    H^1(0,T; H^1(\Omega_e^\ve))$, and
$\partial_t u^{\ve,2}_{f} \in H^1(0,T;  H^1(\Omega_f^\ve))^3\cap L^2(0,T; V)$, with $V=\{v \in H^1(\Omega_f^\ve)^3:  \,  {\rm div }\, v=0 \, \text{ in } \,  \Omega_{f}^\ve \}$,  satisfying
\eqref{weak_u_ef} with $b_{e,3}^{\ve,2}$ in place of $b_{e,3}^\ve$.
Taking $\psi\in L^2(0,T; H^1_0(\Omega_e^\ve))$, $\phi\in L^2(0,T; H^1_0(\Omega_e^\ve))^3$,  and $\eta \in L^2(0,T; V_0)$, where   $V_0=\{v \in H^1_0(\Omega_f^\ve)^3:  \, {\rm div }\, v=0 \,  \text{ in } \, \Omega_{f}^\ve \}$, as test functions in the weak formulation we obtain the equations for $u_{e}^{\ve,2}$ and $p_e^{\ve,2}$ in   \eqref{equa_cla} and  $\langle\rho_f \partial^2_t u^{\ve,2}_f - \ve^2\mu\, {\rm div}({\bf e}(\partial_t u^{\ve,2}_f)), \eta \rangle =0$ for any $\eta \in L^2(0,T; V_0)$.
Then  De~Rham's theorem applied to $-\rho_f \partial^2_t u_f^{\ve,2} + \ve^2\mu\, {\rm div}({\bf e}(\partial_t u^{\ve,2}_f))$  implies the existence of $p_f^{\ve,2} \in L^2((0,T)\times \Omega_f^\ve)$ such that $-\rho_f \partial^2_t u_f^{\ve,2} + \ve^2 \mu\,{\rm div}({\bf e}(\partial_t u^{\ve,2}_f))= \nabla p_f^{\ve, 2}$.
 Using first $\psi=0$, $\phi=0$,  and $\eta \in L^2(0,T; H^1(\Omega_f^\ve))^3$, with $\Pi_\tau \eta =0$ on $(0,T)\times\Gamma^\ve$, as a test function in the weak formulation of the equations for
$\big(u_{e}^{\ve,2}, \,p_e^{\ve,2},\,\partial_t u^{\ve,2}_{f}\big)$ we obtain the transmission condition $-n\cdot \ve^2\mu \, {\bf e}(\partial_t u^{\ve,2}_f) \, n + p_f^{\ve,2} =p_e^{\ve,2}$ on $(0,T)\times\Gamma^\ve$, satisfied in the distribution sense.  Choosing  $\psi=0$, $\phi\in L^2(0,T; H^1(\Omega_e^\ve))^3$ and $\eta \in L^2(0,T; H^1(\Omega_f^\ve))^3$, with $\phi = \eta$ on $(0,T)\times\Gamma^\ve$, as test functions and using the equations for $u_e^{\ve,2}$ and $\partial_t u _f^{\ve,2}$ ensure $(\ve^2\mu \, {\bf e}(\partial_t u^{\ve,2}_f)  - p_f^{\ve,2}I)n =({\bf E}^\ve(b^{\ve,2}_{e,3}) {\bf e}(u_e^{\ve,2}) - p_e^{\ve,2}I)n$ on $(0,T)\times\Gamma^\ve$.
Then, using the equations for $ u_{e}^{\ve,2}$, $p_e^{\ve,2}$ and $\partial_t u^{\ve,2}_{f}$ and considering $\psi\in L^2(0,T; H^1(\Omega_e^\ve))$, $\phi\in L^2(0,T; H^1(\Omega_e^\ve))^3$  and $\eta \in L^2(0,T; H^1(\Omega_f^\ve))^3$, with $\Pi_\tau \phi = \Pi_\tau \eta$ on $(0,T)\times\Gamma^\ve$ and $\psi=0$, \, $\phi=0$ on $(0,T)\times \partial\Omega$,
as test functions  we obtain the transmission condition $(-K^\ve_p \nabla p^{\ve,2}_e + \partial_t u_e^{\ve,2}) \cdot n = \partial_t u_f^{\ve, 2} \cdot n$ on $(0,T)\times\Gamma^\ve$ in the distribution sense.  Taking  $\psi\in L^2(0,T; H^1(\Omega_e^\ve))$, $\phi\in L^2(0,T; H^1(\Omega_e^\ve))^3$ and $\eta \in L^2(0,T; H^1(\Omega_f^\ve))^3$, with $\Pi_\tau \phi = \Pi_\tau \eta$ on $(0,T)\times\Gamma^\ve$, as test functions we obtain  the boundary conditions on $(0,T)\times \partial\Omega$.
Hence we obtain that $\big(u_{e}^{\ve,2}, \,p_e^{\ve,2},\,\partial_t u^{\ve,2}_{f}, p^{\ve, 2}_f\big)$ is a weak solution of  \eqref{equa_cla}, with $b_{e,3}^{\ve,2}$ in place of $b_{e,3}^\ve$, together with   the corresponding external boundary conditions in \eqref{exbou_co}.  Standard arguments pertaining to the consideration of  two solutions of  \eqref{equa_cla} imply the uniqueness of a weak solution of \eqref{equa_cla}, \eqref{exbou_co}.
 The transmission condition   $-n\cdot \ve^2\mu \, {\bf e}(\partial_t u^{\ve,2}_f) \, n + p_f^{\ve,2} =p_e^{\ve,2}$ on $(0,T)\times\Gamma^\ve$ ensures that $p_f^{\ve,2}$ is defined uniquely.

Also,   we obtain that the estimates similar to \eqref{estim_u_p_u}  are valid for the functions $\big(u_{e}^{\ve,2}, \,p_e^{\ve,2},\,\partial_t u^{\ve,2}_{f}, \, p_f^{\ve, 2}\big)$ uniformly with respect to solutions
of  equations \eqref{eq_codif} with boundary conditions in \eqref{exbou_co}:
\begin{equation}\label{bou_u_p_u_itera}
\begin{aligned}
&\| \partial_t u_e^{\ve,2}\|_{L^\infty(0,T; L^2(\Omega_e^\ve))}  + \| \nabla u_e^{\ve,2}\|_{L^\infty(0,T;  L^2(\Omega_e^\ve))}  \leq C,\\
& \| p_e^{\ve,2}\|_{L^\infty(0,T; L^2(\Omega_e^\ve))}  + \| \nabla p_e^{\ve,2}\|_{L^2( \Omega_{e,T}^\ve)}   \leq C, \\
&\| \partial_t u_f^{\ve,2}\|_{L^\infty(0,T; L^2(\Omega_f^\ve))}  + \ve \| \nabla \partial_t u_f^{\ve,2}\|_{L^2(\Omega_{f,T}^\ve)}  + \|p_f^{\ve, 2}\|_{L^2(\Omega_{f,T}^\ve)} \leq C.
\end{aligned}
\end{equation}
Iterating this step, we conclude  the existence of a solution  $\big(b_{e}^{\ve,j},\, c_{e}^{\ve,j},\, c_{f}^{\ve,j}\big)$ of   \eqref{eq_codif}  with $(u_e^{\ve, j-1}, \partial_t u_f^{\ve, j-1})$ instead of   $(u_e^{\ve}, \partial_t u_f^{\ve})$  and  a solution $\big(u_{e}^{\ve,j},\,p_e^{\ve,j},\,\partial_t u^{\ve,j}_{f}, \, p_f^{\ve, j}\big)$   of system  \eqref{equa_cla} with $b_{e}^{\ve,j}$ instead of $b_{e}^{\ve}$,  and  that the estimates similar to (\ref{bou_b_c_itera}) and (\ref{bou_u_p_u_itera}) are fulfilled for  $\big(b_{e}^{\ve,j},\, c_{e}^{\ve,j},\, c_{f}^{\ve,j}\big)$ and $\big(u_{e}^{\ve,j},\,p_e^{\ve,j},\,\partial_t u^{\ve,j}_{f}, \, p_f^{\ve, j}\big)$,   with $j\geq 2$.
\\
To show the contraction property of $\mathcal K$,  we consider two  iterations
$$
(b_e^{\ve,j-1}, c_e^{\ve,j-1}, c_f^{\ve,j-1}),\    (\partial_t u_e^{\ve,j-2},\partial_t u_f^{\ve, j-2})\;\; \text{and}\;  \;(b_e^{\ve,j}, c_e^{\ve,j}, c_f^{\ve,j}),\ (\partial_t u_e^{\ve,j-1 },\partial_t u_f^{\ve, j-1}).
$$
Then the   differences   $\widetilde b_e^{\ve,j} = b^{\ve,j-1}_e - b_e^{\ve,j}$,  $\widetilde c_e^{\ve,j} = c^{\ve,j-1}_e - c_e^{\ve,j}$,  and $\widetilde c_f^{\ve,j} = c^{\ve,j-1}_f - c_f^{\ve,j}$
satisfy the following equations:
\begin{equation}\label{eq_diff}
\begin{aligned}
&\partial_t \widetilde b_e^{\ve,j} - \text{div}(D_b\nabla \widetilde b^{\ve,j}_e)
 \\&\hskip 1cm =  g_b(c_e^{\ve,j-1}, b_e^{\ve,j-1}, \be(u_e^{\ve,j-2}))  - g_b(c_e^{\ve,j}, b_e^{\ve,j}, \be(u_e^{\ve,j-1}))
  \   &&  \text{ in } \Omega_{e,T}^\ve,\\
&\partial_t \widetilde c_e^{\ve,j} - \text{div}(D_c\nabla \widetilde c^{\ve,j}_e)
\\&\hskip 1cm =  g_e ( c_e^{\ve,j-1}, b_e^{\ve,j-1}, \be(u_e^{\ve,j-2}))- g_e ( c_e^{\ve,j}, b_e^{\ve,j}, \be(u_e^{\ve,j-1}))  \   &&  \text{ in } \Omega_{e,T}^\ve ,\\
&\partial_t \widetilde c^{\ve,j}_f -
\mathrm{div}\big(D_f\nabla \widetilde c_f^{\ve,j} - \mathcal G(\partial_t u_f^{\ve,j-2}) \widetilde c_f^{\ve,j}\big)
\\ &\hskip 1cm+\mathrm{div}\big( c_f^{\ve,j}\big[\mathcal G(\partial_t u_f^{\ve,j-2}) - \mathcal G(\partial_t u_f^{\ve,j-1})\big] \big) \\
&\hskip 1cm= g_f(c_f^{\ve,j-1})- g_f(c_f^{\ve,j})\   &&  \text{ in } \Omega_{f,T}^\ve,
\end{aligned}
\end{equation}
together with the boundary conditions
\begin{equation}\label{eq_diff_2}
\begin{aligned}
&D_b\nabla \widetilde b_e^{\ve,j} \cdot n =\ve\big(P(b_e^{\ve, j-1}) -P(b_e^{\ve, j}) \big) && \text{ on }   \Gamma^\ve_T,  \\
&\widetilde c_f^{\ve,j} = \widetilde c_e^{\ve,j}  \quad  && \text{ on } \Gamma_T^\ve\setminus\widetilde \Gamma^\ve_T, \\
&D_e \nabla \widetilde c_e^{\ve,j} \cdot n = \big[D_f\nabla \widetilde c_f^{\ve,j} -  \mathcal G(\partial_t u_f^{\ve,j-2} )\widetilde c_f^{\ve,j}\big]\cdot n  \\
&\hskip 1.8cm -\big[
 \big(\mathcal G(\partial_t  u_f^{\ve,j-2}) - \mathcal G(\partial_t  u_f^{\ve,j-1})\big)  c_f^{\ve,j} \big]\cdot n  \quad  && \text{ on } \Gamma_T^\ve\setminus
  \widetilde \Gamma^\ve_T, \\
&D_e\nabla \widetilde c_e^{\ve,j} \cdot n =0 && \text{ on } \widetilde \Gamma^\ve_T,\\
& \big[D_f\nabla \widetilde c_f^{\ve,j} -  ( \mathcal G(\partial_t  u_f^{\ve,j-2}) c_f^{\ve,j-1} - \mathcal G(\partial_t u_f^{\ve,j-1} ) c_f^{\ve,j})\big]\cdot n   = 0 && \text{ on } \widetilde \Gamma^\ve_T, \\
&D_b\nabla \widetilde b_e^{\ve,j} \cdot n =F_b(b_e^{\ve, j-1}) -F_b(b_e^{\ve, j})  && \text{ on }   (\partial \Omega)_T,  \\
&D_e\nabla \widetilde c_e^{\ve,j} \cdot n =F_c(c_e^{\ve, j-1}) -F_c(c_e^{\ve, j})  && \text{ on }   (\partial \Omega)_T.
\end{aligned}
\end{equation}
 Using $\widetilde b_e^{\ve,j}$, $\widetilde c_e^{\ve,j}$ and $\widetilde c_f^{\ve,j}$ as test functions in  the weak formulation of  \eqref{eq_diff} and \eqref{eq_diff_2} we obtain, for any $\delta_1>0$,
\begin{equation}\label{eq_diff_b_weak}
\begin{aligned}
\partial_t\| \widetilde b_e^{\ve,j}\|^2_{L^2(\Omega_e^\ve)} +
\|\nabla \widetilde b^{\ve,j}_e\|^2_{L^2(\Omega_e^\ve)}
  \leq   (\ve^2+ \delta_1)   \|\nabla \widetilde b_e^{\ve,j}\|^2_{L^2(\Omega_e^\ve)} \qquad \\
   +C_1
\big(\|b_e^{\ve,j-1}\|_{L^\infty(\Omega_e^\ve)} +C_{\delta_1}\big) \left[ \|\widetilde b_e^{\ve,j}\|^2_{L^2(\Omega_e^\ve)} +
 \|\widetilde c_e^{\ve,j}\|^2_{L^2(\Omega_e^\ve)} \right] \\
 +
C_2\|b_e^{\ve,j-1}\|_{L^\infty(\Omega_e^\ve)} \left[ \| \be(\widetilde u_e^{\ve,j-1})\|^2_{L^2(\Omega_e^\ve)}  +  \|\widetilde b_e^{\ve,j}\|^2_{L^2(\Omega_e^\ve)} \right]
\\
 + C_3\big[\|c_e^{\ve,j-1}\|_{L^2(\Omega_e^\ve)}+ \|\be(u_e^{\ve,j-2})\|_{L^2(\Omega_e^\ve)}\big]\|\widetilde b_e^{\ve,j}\|^2_{L^4(\Omega_e^\ve)}
\\  + C_4\big[\ve \|\widetilde b_e^{\ve,j}\|^2_{L^2(\Gamma^\ve)}  + \|\widetilde b_e^{\ve,j}\|^2_{L^2(\partial \Omega)}  \big]
\end{aligned}
\end{equation}
and
 \begin{equation}\label{eq_diff_c_weak}
\begin{aligned}
&\partial_t\| \widetilde c_e^{\ve,j}\|^2_{L^2(\Omega_e^\ve)} + \|\nabla \widetilde c^{\ve,j}_e\|^2_{L^2(\Omega_e^\ve)}
+ \partial_t\| \widetilde c_f^{\ve,j}\|^2_{L^2(\Omega_f^\ve)} + \|\nabla \widetilde c^{\ve,j}_f\|^2_{L^2(\Omega_f^\ve)}
\\
&\leq C_1\left( \|b_e^{\ve,j}\|_{L^\infty(\Omega_e^\ve)}+1\right) \left(\|\widetilde c_e^{\ve,j}\|^2_{L^2(\Omega_e^\ve)} +
 \| \be(\widetilde u_e^{\ve,j-1})\|^2_{L^2(\Omega_e^\ve)}\right)\\
&
 + C_2\left( \| \be(u_e^{\ve,j-2})\|_{L^2(\Omega_e^\ve)}  + \|c_e^{\ve,j-1}\|_{L^2(\Omega_e^\ve)} \right)\left( \|\widetilde c_e^{\ve,j}\|^2_{L^4(\Omega_e^\ve)}  +
  \|\widetilde b_e^{\ve,j}\|^2_{L^4(\Omega_e^\ve)} \right)
  \\
  & + C_3\| \be(\widetilde u_e^{\ve,j-1})\|_{L^2(\Omega_e^\ve)} \|c_e^{\ve,j-1}\|_{L^4(\Omega_e^\ve)}  \|\widetilde c_e^{\ve,j}\|_{L^4(\Omega_e^\ve)}
  +C_4 \|\widetilde c_e^{\ve,j}\|^2_{L^2(\partial \Omega)}
 \\ & + C_5 \Big( \|c_f^{\ve,j-1}\|^2_{L^4(\Omega_f^\ve)} \|\partial_t \widetilde u_f^{\ve,j-1}\|^2_{L^4(\Omega_f^\ve)}
 +   \|\mathcal G(\partial_t  u_f^{\ve,j-1})\|^2_{L^4(\Omega_f^\ve)}\| \widetilde c_f^{\ve,j}  \|^2_{L^4(\Omega_f^\ve)}\Big),
\end{aligned}
\end{equation}
where $\widetilde u_e^{\ve,j-1} =  u_e^{\ve,j-1}-u_e^{\ve,j-2}$  and  $\widetilde u_f^{\ve,j-1} =  u_f^{\ve,j-1}-u_f^{\ve,j-2}$.
 Using  the trace and the  Gagliardo-Nirenberg  inequalities   we  estimate $\|\widetilde b_e^{\ve,j}\|^2_{L^4(\Omega_e^\ve)}$,
 $\|\widetilde c_e^{\ve,j}\|^2_{L^4(\Omega_e^\ve)}$, and $\|\widetilde c_f^{\ve,j}\|^2_{L^4(\Omega_f^\ve)}$, as well as the boundary terms
 $ \ve \|\widetilde b_e^{\ve,j}\|^2_{L^2(\Gamma^\ve)} $, $\|\widetilde b_e^{\ve,j}\|^2_{L^2(\partial \Omega)}$, and $\|\widetilde c_e^{\ve,j}\|^2_{L^2(\partial \Omega)} $
  in the same way as in    \eqref{estim_GN_trace}.
The estimates for $c_e^{\ve, j-1}$ in  $L^\infty(0,T; L^4(\Omega_e^\ve))$ and for $c_f^{\ve, j-1}$ in  $L^\infty(0,T; L^4(\Omega_f^\ve))$  ensure
  \begin{equation*}
\begin{aligned}
\int_0^s\Big[ \| \be(\widetilde u_e^{\ve,j-1})\|_{L^2(\Omega_e^\ve)} \|c_e^{\ve,j-1}\|_{L^4(\Omega_e^\ve)}  \|\widetilde c_e^{\ve,j}\|_{L^4(\Omega_e^\ve)} +  \|c_f^{\ve,j-1}\|^2_{L^4(\Omega_f^\ve)} \|\partial_t \widetilde u_f^{\ve,j-1}\|^2_{L^4(\Omega_f^\ve)} \Big]dt \\
\leq \|c_e^{\ve,j-1}\|_{L^\infty(0,s; L^4(\Omega_e^\ve))} \big[ C_1  \| \be(\widetilde u_e^{\ve,j-1})\|^2_{L^2(\Omega_{e,s}^\ve)} + C_\delta  \|\widetilde c_e^{\ve,j}\|^2_{L^2(\Omega_{e,s}^\ve)} + \delta \|\nabla \widetilde c_e^{\ve,j}\|^2_{L^2(\Omega_{e,s}^\ve)} \big]  \\
 + \|c_f^{\ve,j-1}\|^2_{L^\infty(0,s; L^4(\Omega_f^\ve))}\big[  C_\delta  \|\partial_t \widetilde u_f^{\ve,j-1}\|^2_{L^2(\Omega_{f,s}^\ve)}  + \delta  \|\be(\partial_t \widetilde u_f^{\ve,j-1})\|^2_{L^2(\Omega_{f,s}^\ve)}  \big]
\end{aligned}
\end{equation*}
for any $\delta >0$. Then combining  \eqref{eq_diff_b_weak} and \eqref{eq_diff_c_weak} and applying the Gronwall inequality we obtain
 \begin{equation}\label{eq_diff_bc_weak}
\begin{aligned}
 \| \widetilde b_e^{\ve,j}\|^2_{L^\infty(0,s; L^2(\Omega_e^\ve))} +
\|\nabla \widetilde b^{\ve,j}_e\|^2_{L^2(\Omega_{e,s}^\ve)} +
\| \widetilde c_e^{\ve,j}\|^2_{L^\infty(0,s; L^2(\Omega_e^\ve))}
+ \|\nabla \widetilde c^{\ve,j}_e\|^2_{L^2(\Omega_{e,s}^\ve)}\\
 + \| \widetilde c_f^{\ve,j}\|^2_{L^\infty(0, s; L^2(\Omega_f^\ve))} + \|\nabla \widetilde c^{\ve,j}_f\|^2_{L^2(\Omega_{f,s}^\ve)}
\\
\leq  C_1 \| \be(\widetilde u_e^{\ve,j-1})\|^2_{L^2(\Omega_{e,s}^\ve)}+ C_\delta  \|\partial_t  \widetilde u_f^{\ve,j-1}\|^2_{L^2(\Omega_{f,s}^\ve)} +\delta \|\be(\partial_t  \widetilde u_f^{\ve,j-1})\|^2_{L^2(\Omega_{f,s}^\ve)}.
 \end{aligned}
\end{equation}
Notice that $C_1= C_2 e^{C_3s} \leq C_2 e^{C_3T}$ and $C_\delta= C_4 e^{C_5s} \leq C_4 e^{C_5T}$ for $s\in (0,T]$ and we can consider  $C_1$ and $C_\delta$ to be independent of $s$.

 Considering $|\widetilde b_e^{\ve,j}|^{p-1}$, with $p=2^k$, $k=2,3,\ldots$, as a test function in the weak formulation of \eqref{eq_diff} and \eqref{eq_diff_2},  applying the  Gagliardo-Nirenberg inequality to $|\widetilde b_e^{\ve,j}|^{\frac p2}$,
 and using  the iteration in $p=2^k$ with $k \in \mathbb N$, see \cite[Lemma 3.2]{Alikakos},  we derive  the estimate
$$
\begin{aligned}
\|\widetilde b_{e}^{\ve, j} \|_{L^\infty(0,s; L^\infty(\Omega_e^\ve))}& \leq C_1 \|\be(\widetilde u_e^{\ve, j-1})\|_{L^{1 + \frac 1\sigma}(0,s; L^2(\Omega_e^\ve))} \\
& + C_\delta \|\partial_t \widetilde u_{f}^{\ve, j-1}\|_{L^2(\Omega_{f,s}^\ve)} +  \delta  \|\be(\partial_t \widetilde u_{f}^{\ve, j-1})\|_{L^2(\Omega_{f,s}^\ve)}
\end{aligned}
$$
for  $s \in (0,T]$,  an arbitrary  $0<\delta<1$,  and  for any   $0<\sigma < 1/9$. For more details see the proof of Lemma~\ref{Linf_contract} in Appendix.
Notice that $C_1$ and $C_\delta$ depend on $T$ and are independent of $s$.

Now letting
$\widetilde u^{\ve,j}_e=u^{\ve,j-1}_e-u^{\ve,j}_e$, $\ \widetilde p^{\ve,j}_e=p^{\ve,j-1}_e-p^{\ve,j}_e$, \ $\widetilde u^{\ve,j}_f
= u^{\ve,j-1}_f- u^{\ve,j}_f$,
\ considering  the equations for  $(\widetilde u^{\ve,j}_e, \widetilde p^{\ve,j}_e, \partial_t \widetilde u^{\ve,j}_f)$
and using $(\partial_t \widetilde u^{\ve,j}_e, \widetilde p^{\ve,j}_e, \partial_t \widetilde u^{\ve,j}_f)$ as test functions in the integral formulation of these equations,
we arrive at the following inequality
\begin{equation}
\begin{aligned}
&\quad  \frac 1 2 \rho_e\,  \partial_t \|\partial_t \widetilde u^{\ve,j}_e \|^2_{L^2(\Omega_e^\ve)} +
\frac 12  \partial_t \langle {\bf E}^\ve(b_{e,3}^{\ve,j-1}) \be(\widetilde u_e^{\ve,j}),  \be(\widetilde u_e^{\ve,j})\rangle_{\Omega_e^\ve}
 \\ &  + \frac 12\rho_p\,  \partial_t \|\widetilde p_e^{\ve,j}\|^2_{L^2(\Omega_e^\ve)}  + \|\nabla \widetilde p_e^{\ve,j}\|^2_{L^2(\Omega_e^\ve)}
\\
& +\frac 12 \rho_f\,  \partial_t \|\partial_t \widetilde u^{\ve,j}_f \|^2_{L^2(\Omega_f^\ve)}  + \mu\,  \ve^2\,  \| \be(\partial_t \widetilde u_f^{\ve,j} ) \|^2_{L^2(\Omega_f^\ve)} \\
&   \leq
  \big \langle \partial_t {\bf E}^\ve(b_{e,3}^{\ve,j-1}) \be(\widetilde u_e^{\ve,j}),  \be(\widetilde u_e^{\ve,j})\big\rangle_{\Omega_e^\ve} \\ & +
   \big \langle ({\bf E}^\ve(b_{e,3}^{\ve,j-1})- {\bf E}^\ve(b_{e,3}^{\ve,j})) \, \be( u_e^{\ve,j-1}), \partial_t \be(\widetilde u_e^{\ve, j})\big\rangle_{\Omega_e^\ve}
 \\
 &  \leq C_1  \| \be(\widetilde u_e^{\ve,j})\|^2_{L^2(\Omega_e^\ve)} +
\partial_t \big\langle ({\bf E}^\ve(b_{e,3}^{\ve,j-1})- {\bf E}^\ve(b_{e,3}^{\ve,j})) \be( u_e^{\ve,j-1}),  \be(\widetilde u^{\ve, j}_e )\big\rangle_{\Omega_e^\ve} \\
& - \big\langle \partial_t({\bf E}^\ve(b_{e,3}^{\ve,j-1})- {\bf E}^\ve(b_{e,3}^{\ve,j})) \be(u_e^{\ve, j-1}) ,  \be(\widetilde u_e^{\ve, j} )\big\rangle_{\Omega_e^\ve} \\
& -
 \big\langle ({\bf E}^\ve(b_{e,3}^{\ve,j-1})- {\bf E}^\ve(b_{e,3}^{\ve,j}))\, \partial_t \be(u_e^{\ve, j-1}),  \be(\widetilde u_e^{\ve, j})\big\rangle_{\Omega_e^\ve}.
\end{aligned}
\end{equation}
Thus using a priori estimates for $u_e^{\ve,j}$, $\partial_t u_e^{\ve,j}$, $\partial_t u_f^{\ve,j}$,  $b_e^{\ve,j}$, and $\widetilde b_e^{\ve,j}$ we have that
\begin{equation*}
\begin{aligned}
&\|\widetilde u_e^{\ve,j}\|_{L^\infty(0,s; L^2(\Omega_e^\ve))} +\|\be(\widetilde u_e^{\ve,j})\|_{L^\infty(0,s; L^2(\Omega_e^\ve))}   + \|\widetilde p_e^{\ve,j}\|_{L^\infty(0,s;L^2(\Omega_e^\ve))}
\\ & + \|\nabla \widetilde p_e^{\ve,j}\|_{L^2(\Omega_{e,s}^\ve)} + \|\partial_t \widetilde u_f^{\ve,j} \|_{L^\infty(0,s; L^2(\Omega_f^\ve))} + \|\be(\partial_t \widetilde u_f^{\ve,j}) \|_{L^2(\Omega_{f,s}^\ve)}
\\
&\leq C_1 \|\widetilde b_e^{\ve,j} \|_{L^\infty(0,s; L^\infty( \Omega_e^\ve))} \\
&\leq C_2 \|\be(\widetilde u_e^{\ve, j-1})\|_{L^{\sigma_1}(0,s; L^2(\Omega_e^\ve))}   + C_\delta \|\partial_t \widetilde u_f^{\ve, j-1} \|_{L^2(\Omega_{f,s}^\ve)}  + \delta\|\be(\partial_t \widetilde u_f^{\ve, j-1}) \|_{L^2(\Omega_{f,s}^\ve)}
\end{aligned}
\end{equation*}
for $s\in (0,T]$, $0<\delta<1$,  and $\sigma_1>10$, with the constants $C_1$, $C_2$, and  $C_\delta$ depending on $T$ and the model parameters but independent of the solutions, initial data,
and of $s\in (0,T]$.  Considering $\delta <1$ and  sufficiently small time intervals $s$  in the inequality
\begin{equation*}
\begin{aligned}
&\|\widetilde u_e^{\ve,j}\|_{L^\infty(0,s; L^2(\Omega_e^\ve))} + \|\be(\widetilde u_e^{\ve,j})\|_{L^\infty(0,s; L^2(\Omega_e^\ve))} + \|\partial_t \widetilde u_f^{\ve,j} \|_{L^\infty(0,s; L^2(\Omega_f^\ve))}\\
&  + \|\be(\partial_t \widetilde u_f^{\ve,j}) \|_{L^2(\Omega_{f,s}^\ve)} \leq C_2 s^{1/\sigma_1} \|\be(\widetilde u_e^{\ve, j-1})\|_{L^{\infty}(0,s; L^2(\Omega_e^\ve))}  \\ & + C_\delta  s^{1/2}\|\partial_t \widetilde u_f^{\ve, j-1} \|_{L^\infty(0,s; L^2(\Omega_f^\ve))}  + \delta \|\be(\partial_t \widetilde u_f^{\ve, j-1}) \|_{L^2(\Omega_{f,s}^\ve)}
\end{aligned}
\end{equation*}
we obtain  by the contraction arguments the existence of a fixed point of $\mathcal K$ and hence the existence of a unique weak solution of the microscopic problem  (\ref{eq_codif})--(\ref{exbou_co}) in $(0,s)$.
Since the constants $C_2$ and $C_\delta$ depend only on $T$ and the model parameters and do not depend on $s$,  iterating over  time intervals we obtain the existence and uniqueness result in the whole time interval  $(0,T)$.
\end{proof}

\section{Convergence results}\label{convergence}
The {\it a priori} estimates  proved in Lemma~\ref{Lemma:apriori} imply   convergence results for the components of solutions of the  microscopic problem \eqref{eq_codif}--\eqref{exbou_co}.

\begin{lemma}\label{convergence_u_p}
There exist functions $u_e \in H^1(0,T; H^1(\Omega))\cap H^2(0,T; L^2(\Omega))$, $p_e \in H^1(0,T; H^1(\Omega))$,
$u_e^1, \partial_t u_e^1 \in L^2(\Omega_T; H^1_{\text{per}}(Y_e)/\mathbb R)$,  $p_e^1 \in L^2(\Omega_T; H^1_\text{per}(Y_e)/\mathbb R)$,  $\partial_t u_f$, $\partial_t^2 u_f \in L^2(\Omega_T; H^1_\text{per}(Y_f))$, and $p_f \in L^2(\Omega_T\times Y_f)$ such that, up to a subsequence,
\begin{equation}\label{convergences_1}
\begin{aligned}
& u_e^\ve \rightarrow u_e && \text{ strongly in } H^1(0,T; L^2(\Omega)), \\
& p_e^\ve \rightarrow p_e && \text{ strongly in } L^2( \Omega_T), \\
&  \partial^2_t u_e^\ve  \rightharpoonup \partial^2_t u_e, \;  \partial_t p_e^\ve  \rightharpoonup  \partial_t p_e&& \text{ weakly two-scale} , \\
& \nabla u_e^\ve  \rightharpoonup\nabla u_e + \nabla_y u_e^1 && \text{ weakly two-scale} , \\
& \nabla p_e^\ve \rightharpoonup  \nabla p_e + \nabla_y p_e^1 && \text{ weakly two-scale} ,
\end{aligned}
\end{equation}
and for fluid velocity and pressure we have
\begin{equation}\label{convergences_2}
\begin{aligned}
&\partial_t u_f^\ve  \rightharpoonup  \partial_t u_f, \quad p_f^\ve   \rightharpoonup p_f && \text{ weakly two-scale} , \\
&\ve \nabla \partial_t u_f^\ve  \rightharpoonup  \nabla_y \partial_t u_f && \text{ weakly two-scale} .
\end{aligned}
\end{equation}
Additionally we have weak  two-scale convergence  $\partial_t u^\ve_e  \rightharpoonup  \partial_t u_e$ and $\partial_t u_f^\ve \rightharpoonup  \partial_t u_f$ on $\Gamma^\ve_T$.
\end{lemma}
\begin{proof}
Applying standard extension arguments, see e.g.\   \cite{Acerbi, CiorPaulin99} or Lemma~\ref{extension},  and using the same notation for the original and extended sequences,   from estimates \eqref{estim_u_p_u}  in Lemma~\ref{Lemma:apriori} we obtain  {\it a priori} estimates, uniform in $\ve$,  for $u_e^\ve$, $\nabla u_e^\ve$, $\partial_t u_e^\ve$, $\partial^2_t u_e^\ve$ and $\nabla \partial_t u_e^\ve$, as well as $p_e^\ve$, $\nabla p_e^\ve$ and $\partial_t p_e^\ve$ in  $L^2(\Omega_T)$.
Then the  convergence results for $u_e^\ve$ and $p_e^\ve$ follow directly from the compactness of the embedding of $H^1(0,T; L^2(\Omega))\cap L^2(0,T; H^1(\Omega))$ in $L^2(\Omega_T)$,
the a priori estimates \eqref{estim_u_p_u}, and the compactness theorems for the two-scale convergence, see e.g.\ \cite{allaire, Nguetseng}.
The {\it a priori} estimates \eqref{estim_u_p_u}   and the compactness theorems for the two-scale convergence ensure the convergence results  for $\partial_t u_f^\ve$ and $p_f^\ve$.
Using the trace inequality and a scaling argument  together with  {\it a priori} estimates \eqref{estim_u_p_u} we obtain
$$
\begin{aligned}
\ve \|\partial_t u^\ve_e \|^2_{L^2(\Gamma_T^\ve)} \leq C \big( \|\partial_t u^\ve_e \|^2_{L^2(\Omega_{e,T}^\ve)}  +
\ve^2 \|\nabla \partial_t u^\ve_e \|^2_{L^2(\Omega_{e,T}^\ve)} \big) \leq C, \\
\ve \|\partial_t u^\ve_f \|^2_{L^2(\Gamma_T^\ve)} \leq C \big( \|\partial_t u^\ve_f \|^2_{L^2(\Omega_{f,T}^\ve)}  +
\ve^2 \|\nabla \partial_t u^\ve_f \|^2_{L^2(\Omega_{f,T}^\ve)} \big) \leq C,
\end{aligned}
$$
where the constant $C$ is independent of $\ve$. Then the compactness theorem for the two-scale convergence on oscillating surfaces \cite{Allaire_1996, neuss-radu} ensures the weak  two-scale convergence of
$\partial_t u^\ve_e$ and $\partial_t u^\ve_f$ on $\Gamma^\ve_T$.
\end{proof}

In what  follows we shall use the same notation for $b_e^\ve$, $c^\ve_e$ and their extensions to $\Omega$, whereas the extension of $c^\ve$ from $\widetilde \Omega_{ef}^\ve$ to $\Omega$ will be denoted by $\overline c^\ve$.  Then for $b_e^\ve$ and $c^\ve$ we have the following convergence results.
\begin{lemma}\label{convergence_b_c}
There exist functions
$$b_e, \, c  \in L^2(0,T; H^1(\Omega)), \; \;
b_e^1 \in L^2(\Omega_T; H^1_{\text{per}}(Y_e)/\mathbb R), \; \;  c^1 \in L^2(\Omega_T; H^1_\text{per}(Y\setminus \widetilde \Gamma)/\mathbb R), $$
 such that, up to a subsequence,
\begin{equation}\label{convergences_111}
\begin{aligned}
& b_e^\ve \rightarrow b_e, \;  \; \;  c_e^\ve \rightarrow c, \;  \; \;  \overline c^\ve \to c && \text{ strongly in } L^2(\Omega_T),  \\
& \nabla b_e^\ve \rightharpoonup \nabla b_e + \nabla_y b_e^1 && \text{ weakly two-scale},  \\
& \nabla c^\ve \rightharpoonup \nabla c + \nabla_y c^1 && \text{ weakly two-scale}.
\end{aligned}
\end{equation}
\end{lemma}

\begin{proof}
Using estimates \eqref{estim_b_c} and the extensions of $b_e^\ve$, $c_e^\ve$, and $c^\ve$,  defined in Lemma~\ref{extension},  we obtain
\begin{equation}\label{estim_ext_b}
\begin{aligned}
&\|b_e^\ve\|_{L^2(\Omega_T)}+ \|\nabla b_e^\ve\|_{L^2(\Omega_T)}  + \|c_e^\ve\|_{L^2(\Omega_T)}+ \|\nabla c_e^\ve\|_{L^2(\Omega_T)}  \leq C, \\
& \|\overline c^\ve\|_{L^2(\Omega_T)}+ \|\nabla\overline  c^\ve\|_{L^2(\Omega_T)} \leq C,
\end{aligned}
\end{equation}
where the constant $C$ is independent of $\ve$.
The  estimates \eqref{estim_ext_b},  the compactness of the embedding of $ H^1(\Omega)$ in $L^2(\Omega)$, along with the estimate \eqref{estim_time_h} and the Kolmogorov compactness theorem \cite{Necas}  yield the strong convergence of $b_e^\ve \to b_e$,  $c^\ve_e \to c_e$ and $\overline c^\ve \to c$  in $L^2(\Omega_T)$.  Since  $\Omega_{e,T}^\ve\cap \widetilde \Omega_{ef, T}^\ve \neq \emptyset$ and   $c_e^\ve(t,x) = \overline c^\ve(t,x)$  in  $\Omega_{e,T}^\ve \cap \widetilde \Omega_{ef, T}^\ve$, along with the fact that  $c_e$ and $c$ are independent of the microscopic variables $y$,  we obtain that $c_e(t,x) = c(t,x)$  in $\Omega_T$.

From the estimates for $c^\ve$, applying the compactness theorem for  the two-scale convergence  we obtain that there exists $c^1 \in L^2(\Omega_T; H^1_{\rm per}(Y\setminus \widetilde \Gamma)/\mathbb R)$ such that
$\nabla c^\ve \rightharpoonup \nabla c + \nabla_y c^1$ weakly two-scale~\cite{Zhikov}.
\end{proof}

\section{Derivation of macroscopic equations for  the flow velocity  and elastic deformations}\label{macro_elasticity}

 This section focuses on homogenization of the microscopic problem \eqref{equa_cla}--\eqref{exbou_co}.
First
we define the  effective  tensors  ${\bf E}^{\rm{hom}}$, $K_{p}^{\rm hom}$, and $K_{u}$.

The macroscopic elasticity tensor ${\bf E}^{\rm{hom}}= (E_{ijkl}^{\rm hom})$,  permeability tensor  $K_{p}^{\rm hom}=(K_{p,ij}^{\rm hom})$,  and $K_{u}=(K_{u,ij})$ are defined by
\begin{equation}\label{efective}
\begin{aligned}
E_{ijkl}^{\rm hom} (b_{e,3}) &=\frac 1 {|Y|}\int_{Y_e} \Big(E_{ijkl}(y, b_{e,3}) + E_{ij}(y, b_{e,3}) \be_y(w^{kl}) \Big) dy,\\
K_{p,ij}^{\rm{hom}}(x) & = \frac 1 {|Y|} \int_{Y_e} \Big( K_{p, ij}(x,y) + K_{p,i}(x,y) \nabla_y w_p^j \Big) dy,\\
 K_{u,ij}(x)  & = \frac 1 {|Y|}  \int_{Y_e}  \Big(\delta_{ij} - K_{p,i}(x,y) \nabla_y w_e^j \Big) dy,
\end{aligned}
\end{equation}
where $w^{kl}=w^{kl}(b_{e,3},\cdot)$ for $k,l=1,2,3$,  are $Y$-periodic solutions of the unit cell problems
\begin{equation}\label{unit_cell_ue_w}
\begin{aligned}
&{ \rm div}_y \big({\bf E}(y, b_{e,3}) (\be_y(w^{kl})   + {\bf b}_{kl}) \big) =0 & \text{ in } Y_e, \\
&{\bf E}(y, b_{e,3}) (\be_y(w^{kl})   + {\bf b}_{kl}) \, n = 0 \quad & \text{ on } \Gamma, \\
&\int_{Y_e}  w^{kl} \, dy = 0,
\end{aligned}
\end{equation}
functions $w_p^k=w_p^k(x,\cdot)$,  for $k=1,2, 3$, are $Y$-periodic solutions of  the unit cell problems
\begin{equation}\label{correct_p_Kp}
\begin{aligned}
&{ \rm div}_y \big(K_p(x,y)( \nabla_y w_p^k + e_k)\big) = 0 && \text{ in } Y_e , \\
& K_p(x,y)( \nabla_y w_p^k + e_k) \cdot n = 0  && \text{ on }  \Gamma, \\
&\int_{Y_e}  w_p^k\,  dy = 0,
\end{aligned}
\end{equation}
and  $w_e^k=w_e^k(x,\cdot)$, for $k=1,2,3$,  are $Y$-periodic solutions of the unit cell problems
\begin{equation}\label{correct_p_Ku}
\begin{aligned}
& {\rm div}_y (K_p(x,y) \nabla_y  w_e^k  -   e_k) = 0 && \text{ in } Y_e , \\
& \; \; \; (K_p(x,y) \nabla_y w_e^k  -  e_k) \cdot n = 0  && \text{ on }  \Gamma, \\
&\int_{Y_e}  w_e^k\,  dy = 0.
\end{aligned}
\end{equation}
Here ${\bf b}_{kl} = e_k \otimes e_l$ and $\{ e_j\}_{j=1}^3$ is the canonical  basis of $\mathbb R^3$.

\begin{lemma}
Periodic cell problems  \eqref{unit_cell_ue_w}, \eqref{correct_p_Kp} and \eqref{correct_p_Ku} are well-posed and
have a unique solution. The tensors ${\bf E}^{\rm{hom}}$ and $K_{p}^{\rm hom}$ are positive definite. Moreover,
${\bf E}^{\rm{hom}}$ possesses the symmetries declared in {\bf A1.}
\end{lemma}

\begin{proof}[Proof Sketch]
Assumptions {\bf A1.} on $\bf E$ and the Korn inequality for periodic functions  ensure the existence of a unique solution of the  unit cell problems  \eqref{unit_cell_ue_w} for a given  $b_{e,3}\in L^2(\Omega_T)$, see e.g.\ \cite{OShY}. Assumptions {\bf A2.} on $K_p$  yield the existence of unique solutions of the unit cell problems  \eqref{correct_p_Kp} and \eqref{correct_p_Ku}. The positive definiteness of ${\bf E}$ and $K_p$, the definition of ${\bf E}^{\rm{hom}}$ and  $K_{p}^{\rm hom}$, and the fact that
$w^{kl}$ and $w_p^k$, for $k,l=1,2,3$, are solutions of  \eqref{unit_cell_ue_w}  and \eqref{correct_p_Kp} ensure in the standard way
(see \cite{BLP})  that ${\bf E}^{\rm{hom}}$ and  $K_{p}^{\rm hom}$ are  positive definite.
The definition of ${\bf E}^{\rm{hom}}$ implies that ${\bf E}^{\rm{hom}}$  satisfies  the same symmetry assumptions in {\bf A1.}  as ${\bf E}$.
\end{proof}

Applying the method of the  two-scale convergence  and using the convergence results in Lemmas~\ref{convergence_u_p} and \ref{convergence_b_c} we derive the homogenized equations for displacement gradient, pressure  and flow velocity   for a given $\{b_e^\ve\}$ such that  $b_e^\ve \to b_e$ strongly in $L^2(\Omega_T)^3$ as $\ve \to 0$. It should be emphasized that we have not yet derived the equation for the limit function $b_e$. We only use the strong convergence of $\{b_e^\ve\}$.

In the formations of the macroscopic problem for $(u_e, p_e, \partial_t u_f)$ we shall use the function $Q(x,\partial_t u_f)$ defined as
\begin{equation}\label{q_uf}
\begin{aligned}
&Q(x,\partial_t u_f) = \frac 1{|Y|} \Big( \int_{Y_f}  \partial_t u_f \, dy - \int_{Y_e} K_p(x,y) \nabla_y q(x,y, \partial_t u_f) \, dy\Big),
\end{aligned}
\end{equation}
where for   $(t,x) \in \Omega_T$  the function $q$ is a $Y$-periodic solution of the problem
\begin{equation}\label{two-scale_qf}
\begin{aligned}
&{\rm div}_y (K_p(x,y) \nabla_y q )= 0 && \text{ in } Y_e , \\
& - K_p(x,y) \nabla_y q \cdot n =  \partial_t u_f \cdot n \quad   && \text{ on }  \Gamma, \\
& \int_{Y_e} q(x,y, \partial_t u_f) \, dy =0.&&
\end{aligned}
\end{equation}

\begin{theorem}\label{thm52}
A sequence of solutions $\{u_e^\ve, p_e^\ve, \partial_t u_f^\ve, p_f^\ve \}$ of microscopic problem  \eqref{equa_cla} and \eqref{exbou_co} converges, as $\ve\to0$, to a solution  $(u_e, p_e, \partial_t u_f, \pi_f)$ of the macroscopic equations
\begin{equation}\label{macro_ue}
\begin{aligned}
&\vartheta_e \rho_e\,   \partial_t^2 u_e - {\rm div} ( {\bf E}^{\rm{hom}}(b_{e,3}) \be(u_e)) + \nabla p_e + \vartheta_f  \rho_f\, \dashint_{Y_f} \partial^2_t u_f \, dy = 0 && \text{ in } \Omega_T, \\
&\vartheta_e \rho_p \, \partial_t p_e - {\rm div} \big( K_{p}^{\rm hom} \nabla p_e - K_ {u}\,  \partial_t u_e - Q(x, \partial_t u_f)\big) =0 && \text{ in } \Omega_T,
\end{aligned}
\end{equation}
with boundary and initial conditions
\begin{equation}\label{macro_pe}
\begin{aligned}
&  {\bf E}^{\rm hom}(b_{e,3}) \be(u_e) \, {n} =F_u  && \text{ on } (\partial \Omega)_T, \\
& (K_{p}^{\rm hom} \nabla p_e - K_{u}\,  \partial_t u_e)\cdot {n} = F_p + Q(x, \partial_t u_f) \cdot n
&& \text{ on }   (\partial \Omega)_T, \\
& u_e(0) = u_{e0}, \quad \partial_t u_{e}(0) = u_{e0}^1, \quad p_{e}(0) = p_{e0} && \text{ in } \Omega,
\end{aligned}
\end{equation}
and the  two-scale  problem for  the fluid flow velocity and pressure
\begin{equation}\label{macro_two-scale_uf}
\begin{aligned}
&\rho_f\, \partial^2_t u_f - { \rm div}_y (\mu \, \be_y(\partial_t u_f) - \pi_fI) + \nabla p_e = 0, \quad {\rm  div}_y \partial_t u_f =0 && \text{ in } \; \Omega_T\times Y_f, \\
&\Pi_\tau \partial_t u_f  =  \Pi_\tau\partial_t u_e  && \text{ on } \; \Omega_T\times \Gamma, \\
 & { n} \cdot ( \mu \, \be_y(\partial_t u_f) - \pi_f I) \,  n = - p_e^1 && \text{ on } \; \Omega_T\times \Gamma, \\
 &\partial_t u_ f (0) = u_{f0}^1 && \text{ in } \; \Omega\times Y_f,
 \end{aligned}
\end{equation}
 where $\vartheta_e=|Y_e|/ |Y|$, $\vartheta_f = |Y_f|/ |Y|$, and
\begin{equation}\label{form_pe-1}
p_e^1(t,x,y)= \sum_{k=1}^3 \partial_{x_k} p_e (t,x) \, w^k_p(x,y) + \sum_{k=1}^3 \partial_{t} u_e^k(t,x) \,  w_e^k(x,y) + q(x,y, \partial_t u_f),
\end{equation}
with $w^k_p$, $w_e^k$ and  $q$  being solutions of \eqref{correct_p_Kp},  \eqref{correct_p_Ku}, and \eqref{two-scale_qf},  respectively.

We have $u_e\in H^2(0,T; L^2(\Omega))\cap H^1(0,T; H^1(\Omega))$, $p_e \in H^1(0,T; H^1(\Omega))$, $\partial_t u_f \in L^2(\Omega_T; H^1(Y_f))\cap H^1(0,T; L^2(\Omega\times Y_f))$,
and $\pi_f \in L^2(\Omega_T\times Y_f)$ and the convergence in the following sense
\begin{equation*}
\begin{aligned}
& u_e^\ve \rightarrow u_e \; \;  \text{  in } H^1(0,T; L^2(\Omega)),    &&p_e^\ve \rightarrow p_e \; \;  \text{  in } L^2( \Omega_T), \\
& \nabla u_e^\ve  \rightharpoonup\nabla u_e + \nabla_y u_e^1, &&
  \nabla p_e^\ve \rightharpoonup  \nabla p_e + \nabla_y p_e^1 && \text{ weakly two-scale} , \\
&  \partial_t u_f^\ve  \rightharpoonup  \partial_t u_f, \quad  p_f^\ve   \rightharpoonup p_e, &&  \ve \nabla \partial_t u_f^\ve  \rightharpoonup  \nabla_y \partial_t u_f && \text{ weakly two-scale}  .
\end{aligned}
\end{equation*}
\end{theorem}

\noindent {\bf Remark.} In the original microscopic problem the equations of poroelasticity and the Stokes system  are coupled
through the transmission conditions.
The limit system shows the strong coupling in the whole domain $\Omega_T$. Namely, the equations for macroscopic displacement and pressure defined in the whole domain $\Omega_T$ are coupled with the two-scale equations for the fluid flow defined on $\Omega_T\times Y_f$.   This coupling in the limit problem  can be observed through both  the lower order terms in the equations and the boundary conditions.

\begin{proof}[Proof of Theorem~{\rm \ref{thm52}}]
\hspace{-0.2 cm} Considering $(\ve \phi(t,x, x/\ve), \ve \psi(t,x, x/\ve), \ve \eta(t,x, x/\ve))$  with $\phi \in C^\infty_0(\Omega_T; C^\infty_{\text{per}}(Y_e))^3$,  $\psi \in C^\infty_0(\Omega_T; C^\infty_\text{per}(Y_e))$,  and $\eta \in C^\infty_0(\Omega_T; C^\infty_\text{per}(Y_f))^3$, as test functions in the weak formulation of \eqref{equa_cla}, with the corresponding boundary conditions in \eqref{exbou_co}, we obtain
\begin{equation}
\begin{aligned}
& \langle\rho_e \partial^2_t u^\ve_e, \ve \phi \rangle_{\Omega_{e,T}^\ve} + \langle {\bf E}^\ve(b^\ve_{e,3}) \be( u^\ve_e), \ve \be(\phi) \rangle_{\Omega_{e,T}^\ve}  +
\langle \nabla p_e^\ve, \ve \phi \rangle_{\Omega_{e,T}^\ve}
\\
&  +
\langle \rho_p\partial_t p^\ve_e, \ve \psi \rangle_{\Omega_{e,T}^\ve} + \langle K_p^\ve \nabla p^\ve_e - \partial_t u_e^\ve, \ve \nabla \psi \rangle_{\Omega_{e,T}^\ve} \\
& + \langle \rho_f \partial^2_t u^\ve_f, \ve \eta \rangle_{\Omega_{f,T}^\ve} +  \ve^2\mu\,  \langle \be(\partial_t u^\ve_f), \ve \, \be(\eta) \rangle_{\Omega_{f,T}^\ve}  -
\langle p_f^\ve, \ve\,  \text{div} \, \eta \rangle_{\Omega_{f,T}^\ve}  \\
& +\langle \partial_t u_f^\ve\cdot n , \ve  \psi \rangle_{\Gamma^\ve_T} - \langle p_e^\ve, \ve \eta\cdot n \rangle_{\Gamma^\ve_T}  = \langle F_u, \ve\phi \rangle_{(\partial \Omega)_T} + \langle F_p, \ve\psi \rangle_{(\partial \Omega)_T} .
\end{aligned}
\end{equation}
Letting   $\ve \to 0$ and using the convergence results in Lemmas~\ref{convergence_u_p} and \ref{convergence_b_c}   yield
\begin{equation}\label{macro_11}
\begin{aligned}
&\langle {\bf E}(y, b_{e,3}) (\be( u_e) + \be_y(u_e^1)),  \be_y(\phi) \rangle_{\Omega_{T}\times Y_e}
\\
&  +  \langle K_p( \nabla p_e + \nabla_y p_e^1)- \partial_t u_e,  \nabla_y \psi \rangle_{\Omega_{T}\times Y_e}
+  \langle \partial_t u_f\cdot n , \psi \rangle_{\Omega_T\times \Gamma} \\
& -
 \langle p_f, \text{div}_y \eta \rangle_{\Omega_{T}\times Y_f}
  - \langle p_e,  \eta\cdot n \rangle_{\Omega_T\times \Gamma}  = 0.
\end{aligned}
\end{equation}
Considering  first \\
(i) $\phi \in C^\infty_0(\Omega_T; C^\infty_{0}(Y_e))^3$,  $\psi \in C^\infty_0(\Omega_T; C^\infty_0(Y_e))$,  and $\eta \in C^\infty_0(\Omega_T; C^\infty_0(Y_f))^3$\\
 and then\\
(ii) $\phi \in C^\infty_0(\Omega_T; C^\infty_{\text{per}}( Y_e))^3$,  $\psi \in C^\infty_0(\Omega_T; C^\infty_\text{per}(Y_e))$,  and $\eta \in C^\infty_0(\Omega_T; C^\infty_\text{per}(Y_f))^3$
with $\Pi_\tau \phi  =\Pi_\tau  \eta $  and $\eta\cdot n =0 $ on $\Omega_T \times \Gamma$, we obtain
\begin{equation}\label{preasure_1}
\langle p_f, \text{ div}_y \eta \rangle_{\Omega_T \times Y_f} =0
\end{equation}
and  the equations for correctors
\begin{equation}\label{correctos_ue}
\begin{aligned}
{\rm div}_y \big( {\bf E}(y, b_{e,3}) (\be( u_e) + \be_y(u_e^1))\big) &=0\quad && \text{ in } \; \Omega_T\times Y_e, \quad\\
 {\bf E}(y, b_{e,3}) (\be( u_e) + \be_y(u_e^1))\, n& =0 \quad  && \text{ on  } \Omega_T\times \Gamma,
\end{aligned}
\end{equation}
and
\begin{equation}\label{correctos_pe}
\begin{aligned}
{\rm div}_y (K_p ( \nabla p_e + \nabla_y p_e^1)- \partial_t u_e) &= 0  \quad && \text{ in } \; \Omega_T\times Y_e, \\
(-K_p ( \nabla p_e + \nabla_y p_e^1)+  \partial_t u_e)\cdot n& = \partial_t u_f\cdot n \quad && \text{ on }  \Omega_T\times \Gamma.
\end{aligned}
\end{equation}
Considering $\eta \in C^\infty_0(\Omega_T; C^\infty_\text{per}(Y_f))^3$ with $\Pi_\tau \phi  =\Pi_\tau  \eta $   on $\Omega_T \times \Gamma$, from \eqref{macro_11} and \eqref{preasure_1}  follows that
\begin{equation*}
p_f= p_f(t,x) \qquad \text{ in } \Omega_T\times Y_f \quad \text{ and } \quad  p_f = p_e \quad  \text{ on }  \Omega_T\times \Gamma.
\end{equation*}
Thus we have $p_f = p_e$  in $\Omega_T$.
Taking  $(\phi(t, x), \psi(t, x), \eta(t, x, x/\ve))$,  where
\begin{itemize}
\item  $\phi\in C^\infty(\overline \Omega_T)^3$ and  $\psi\in C^\infty(\overline \Omega_T)$,
\item  $\eta\in C^\infty(\overline \Omega_T; C^\infty_{\text{per}}(Y_f))^3$ with  $\Pi_\tau \eta  =\Pi_\tau  \phi $ on $\Omega_T \times \Gamma$ and  ${\rm div}_y \eta (t, x,y) = 0$ in $\Omega_T \times Y_f$,
\end{itemize}
as  test functions in the weak formulation of \eqref{equa_cla}, with external boundary conditions in \eqref{exbou_co},   yields
\begin{equation}
\begin{aligned}
&\langle\rho_e\, \partial^2_t u^\ve_e,  \phi \rangle_{\Omega_{e,T}^\ve} + \langle {\bf E}^\ve(b^\ve_{e,3}) \be( u^\ve_e),  \be(\phi) \rangle_{\Omega_{e,T}^\ve}  +
\langle \nabla p_e^\ve,  \phi \rangle_{\Omega_{e,T}^\ve}
\\
&  +
\langle \rho_p\, \partial_t p^\ve_e,  \psi \rangle_{\Omega_{e,T}^\ve} + \langle K_p^\ve \nabla p^\ve_e -  \partial_t u_e^\ve,  \nabla \psi \rangle_{\Omega_{e,T}^\ve} \\
& + \langle \rho_f\, \partial^2_t u^\ve_f,  \eta \rangle_{\Omega_{f,T}^\ve} + \mu\,
 \ve^2  \langle \be(\partial_t u^\ve_f),  \be(\eta) +  \ve^{-1}  \be_y(\eta) \rangle_{\Omega_{f,T}^\ve}  -
\langle p_f^\ve,  \text{div}_x \eta \rangle_{\Omega_{f,T}^\ve}  \\
& +\langle \partial_t u_f^\ve\cdot n ,   \psi \rangle_{\Gamma^\ve_T} - \langle p_e^\ve,  \eta\cdot n \rangle_{\Gamma^\ve_T}  = \langle F_u, \phi \rangle_{(\partial \Omega)_T} + \langle F_p, \psi \rangle_{(\partial \Omega)_T}.
\end{aligned}
\end{equation}
Letting $\ve \to 0$ and using the two-scale convergence of $u_e^\ve$,  $p_e^\ve$,  and  $\partial_t u_f^\ve$ we obtain
\begin{equation}
\begin{aligned}
&\langle \rho_e \partial^2_t u_e,  \phi \rangle_{\Omega_{T}\times Y_e} + \langle {\bf E}(y, b_{e,3}) \big(\be( u_e)+ \be_y(u_e^1)\big),  \be(\phi) \rangle_{\Omega_{T}\times Y_e} \\
& +
\langle \nabla p_e+ \nabla_y p_e^1,  \phi \rangle_{\Omega_{T}\times Y_e}  -  \langle \partial_t u_f, \nabla \psi \rangle_{\Omega_T\times Y_f}
\\
& +
\langle\rho_p \, \partial_t p_e,  \psi \rangle_{\Omega_{T}\times Y_e} + \langle K_p( \nabla p_e+ \nabla_y p_e^1)   -  \partial_t u_e,  \nabla \psi \rangle_{\Omega_{T}\times Y_e}   \\
& + \langle \rho_f \, \partial^2_t u_f,  \eta \rangle_{\Omega_{T}\times Y_f} + \mu
 \langle \be_y(\partial_t u_f),   \be_y(\eta) \rangle_{\Omega_{T}\times Y_f}  +
\langle \nabla p_e,  \eta \rangle_{\Omega_{T}\times Y_f} \\
&-  \langle   p_e^1,  \eta \cdot n \rangle_{\Omega_{T}\times \Gamma} =|Y| \langle F_u, \phi \rangle_{(\partial \Omega)_T} +|Y|  \langle F_p, \psi \rangle_{(\partial \Omega)_T}.
\end{aligned}
\end{equation}
Here we used the relation $p_e= p_f$  a.e.\ in  $\Omega_T$, as well as  the fact  that due to the relation $\text{div}\,  \partial_t u_f^\ve=0$ and the two-scale convergence of $\partial_t u_f^\ve$,  we have
\begin{equation}
\begin{aligned}
\lim\limits_{\ve \to 0} \langle \partial_t u_f^\ve\cdot n ,   \psi \rangle_{\Gamma^\ve_T} & =
\lim\limits_{\ve \to 0} \left(- \langle \text{div}\,  \partial_t u_f^\ve ,   \psi \rangle_{\Omega_{f,T}^\ve} - \langle \partial_t u_f^\ve , \nabla   \psi \rangle_{\Omega_{f,T}^\ve}\right)\\
& =- \lim\limits_{\ve \to 0}  \langle \partial_t u_f^\ve , \nabla   \psi \rangle_{\Omega_{f,T}^\ve} = - |Y|^{-1} \langle \partial_t u_f , \nabla   \psi \rangle_{\Omega_T\times Y_f}.
\end{aligned}
\end{equation}
To show the convergence of $\langle p_e^\ve,  \eta\cdot n \rangle_{\Gamma^\ve_T} $ we use  $\text{div}_y \eta =0$ and the fact that $p_e^1$ is well-defined on $\Gamma$:
\begin{equation}
\begin{aligned}
\lim\limits_{\ve \to 0} \langle p_e^\ve, & \eta\cdot n \rangle_{\Gamma^\ve_T}  =\lim\limits_{\ve \to 0}\left( - \langle \nabla p_e^\ve,  \eta \rangle_{\Omega^\ve_{f,T}} -  \langle p_e^\ve,  \text{div}_x \eta \rangle_{\Omega^\ve_{f,T}}\right)
\\& = - |Y|^{-1}\langle \nabla p_e + \nabla_y p_e^1,  \eta \rangle_{\Omega_{T}\times Y_f} - |Y|^{-1} \langle p_e,  \text{div}_x \eta \rangle_{\Omega_{T}\times Y_f}
\\& = \;   |Y|^{-1}\big(\langle   p_e^1,  \eta \cdot n \rangle_{\Omega_{T}\times \Gamma} -   \langle \nabla p_e ,  \eta \rangle_{\Omega_{T}\times Y_f} -  \langle p_e,  \text{div}_x \eta \rangle_{\Omega_{T}\times Y_f}\big).
\end{aligned}
\end{equation}
 Notice that $n$ is the  internal for $Y_f$ normal at the boundary $\Gamma$.

Also, for an arbitrary test function $\eta_1\in C_0^\infty(\Omega_T; C_0^\infty(Y_f))$, from the two-scale convergence of $\partial_t u_f^\ve$  and the fact that   $\partial_t u_f^\ve$ is divergence-free it follows that
\begin{equation*}
\begin{aligned}
0= \lim\limits_{\ve\to 0}\langle \text{div } \partial_t u_f^\ve, \ve \eta_1(t, x, x/\ve) \rangle_{\Omega_{f,T}^\ve} =-
\lim\limits_{\ve\to 0}  \langle  \partial_t u_f^\ve, \ve \nabla_x\eta_1 + \nabla_y \eta_1  \rangle_{\Omega_{f,T}^\ve}\\  =-
|Y|^{-1}\langle  \partial_t u_f,  \nabla_y \eta_1\rangle_{\Omega_T\times Y_f}  =|Y|^{-1} \langle \text{div}_y  \partial_t u_f, \eta_1 \rangle_{\Omega_T\times Y_f}.
\end{aligned}
\end{equation*}
Thus $\text{div}_y\, \partial_t u_f =0$  in $\Omega_T\times Y_f$.

Considering $\phi\equiv 0$ and $\psi\equiv 0$, and taking  first  $\eta \in C^\infty_0(\Omega_T; C^\infty_0(Y_f))^3$,  with ${\rm div}_y \, \eta =0$
and then  $\eta \in C^\infty_0(\Omega_T; C^\infty_{\rm per}(Y_f))^3$ with   $\Pi_\tau \eta  =0$ on $\Omega_T\times \Gamma$
we obtain  the two-scale problem \eqref{macro_two-scale_uf} for $\partial_t u_f$. From   the boundary conditions $\Pi_\tau \partial_t u^\ve_e =  \Pi_\tau \partial_t u^\ve_f $ on $\Gamma^\ve_T$ and the two-scale convergence of  $\partial_t u_e^\ve$ and $\partial_t u_f^\ve$ on $\Gamma^\ve_T$, see Lemma~\ref{convergence_u_p},
we obtain
\begin{equation*}
\begin{aligned}
&\frac 1{|Y|} \int_{\Omega_T} \int_\Gamma \Pi_\tau \partial_t u_e(t,x) \,  \psi(t,x,y) \, d\gamma_y dx dt  =  \lim\limits_{\ve \to 0 }  \ve \int_{\Gamma^\ve_T} \Pi_\tau \partial_t u_e^\ve(t,x)\,  \psi(t,x, x/\ve) d\gamma dt \\
& =
 \lim\limits_{\ve \to 0 }  \ve \int_{\Gamma^\ve_T} \Pi_\tau \partial_t u_f^\ve(t,x) \, \psi(t,x, x/\ve)  d\gamma dt
\\&  =\frac 1{|Y|} \int_{\Omega_T} \int_\Gamma \Pi_\tau \partial_t u_f(t,x,y)\,  \psi(t,x,y) \, d\gamma_y dx dt
\end{aligned}
\end{equation*}
for all $\psi \in C_0(\Omega_T; C_{\rm per} (Y))$. Thus $\Pi_\tau \partial_t u_e  = \Pi_\tau \partial_t u_f$   on $\Omega_T \times \Gamma$.

Considering first $\phi\in C^\infty_0(\Omega_T)^3$,  $\psi\in C^\infty_0(\Omega_T)$,  and then
$\phi\in C^\infty(\overline\Omega_T)^3$,  $\psi\in C^\infty(\overline \Omega_T)$,  together with $\eta \in C_0^\infty(\Omega_T; C^\infty_{\rm per} (Y_f))^3$ and  $\Pi_\tau \eta = \Pi_\tau \phi$ on $\Gamma$,  and using the  equations \eqref{macro_two-scale_uf} for $\partial_t u_f$,  we obtain   the limit equations for $u_e$ and $p_e$:
\begin{equation}\label{ue_macro_11}
\begin{aligned}
\vartheta_e  \rho_e \,  \partial_t^2 u_e - {\rm div} \big( {\bf E}^{\rm hom} (b_{e,3}) \, \be(u_e)\big) + \vartheta_e  \nabla p_e +\frac 1{|Y|} \int_{Y_e} \nabla_y p_e^1\,  dy
\\   -
 \frac1{|Y|}\big\langle \mu \,\Pi_\tau (\be(\partial_t u_f)\, n), 1 \big\rangle_{H^{-1/2}, H^{1/2}(\Gamma)}  = 0 & \quad   \text{in } \Omega_T, \\
 {\bf E}^{\rm hom} (b_{e,3}) \, \be(u_e)\, n   = F_u  & \quad   \text{on } (\partial\Omega)_T,
\end{aligned}
\end{equation}
where $\vartheta_e = |Y_e|/ |Y|$ and    the  effective elasticity tensor ${\bf E}^{\rm hom}$ is defined by \eqref{efective},  and
\begin{equation}\label{pe_macro_11}
\begin{aligned}
& \vartheta_e  \rho_p\,  \partial_t p_e\\&  - \frac1 {|Y|}  {\rm div}\Big[ \int_{Y_e} [K_p(\nabla p_e + \nabla_y p_e^1)   - \partial_t u_e ] dy -  \int_{Y_f} \partial_t u_f \, dy \Big] = 0 && \text{ in }  \Omega_T, \\
& \frac 1 {|Y|} \Big[ \int_{Y_e} \left[K_p(\nabla p_e + \nabla_y p_e^1)   - \partial_t u_e\right] dy -   \int_{Y_f}\partial_t u_f\,dy \Big]\cdot n
 =  F_p &&
 \text{ on } (\partial\Omega)_T,
\end{aligned}
\end{equation}
with $p_e^1$ defined by the two-scale problem \eqref{correctos_pe}.  Considering the weak formulation of    \eqref{macro_two-scale_uf}  with the test function $\eta =1$ yields
\begin{equation*}
\begin{aligned}
\rho_f \int_{Y_f}   \partial^2_t u_f \, dy + |Y_f| \nabla p_e =   -  \langle\mu \,( \be_y(\partial_t u_f) - \pi_f I)  n, 1 \rangle_{H^{-1/2}, H^{1/2}(\Gamma)}    = \int_\Gamma p_e^1 n \, d\gamma_y \\
-  \langle \mu\,  \Pi_\tau ( \be_y(\partial_t u_f)  n),1 \rangle_{H^{-1/2}, H^{1/2}(\Gamma)}.
\end{aligned}
\end{equation*}
Using the  $Y$-periodicity of  $p_e^1$ we obtain
\begin{equation*}\label{u_f_int_Yf}
\begin{aligned}
- \langle  \mu \, \Pi_\tau ( \be_y(\partial_t u_f) \, n ), 1 \rangle_{H^{-1/2}, H^{1/2}(\Gamma)}
=\rho_f \int_{Y_f} \partial^2_t u_f dy + |Y_f| \nabla p_e - \int_{Y_e} \nabla_y p_e^1 dy.
\end{aligned}
\end{equation*}
Thus we can  rewrite the equation for $u_e$ as
\begin{equation}\label{eq_ue_macro_1}
\begin{aligned}
\vartheta_e \rho_e \,  \partial_t^2 u_e - {\rm div} \big( {\bf E}^{\rm hom} (b_{e,3})\,  \be(u_e)\big) +  \nabla p_e  + \vartheta_f  \rho_f\,  \dashint_{Y_f}  \partial^2_t u_f dy & = 0 &&  \text{ in } \Omega_T,
\end{aligned}
\end{equation}
where $\vartheta_f = |Y_f|/|Y|$. Considering   the structure of  problem \eqref{correctos_pe} we represent  $p_e^1$ in the form
\begin{equation}\label{form_p1}
\begin{aligned}
p_e^1(t,x,y)= \sum_{k=1}^3 \partial_{x_k} p_e (t,x) \, w^k_p(x, y) + \sum_{k=1}^3 \partial_{t} u_e^k(t,x) \,  w_e^k(x, y) + q(x,y, \partial_t u_f),
\end{aligned}
\end{equation}
where
$w^k_p$ and $w_e^k$ are solutions of unit cell problems \eqref{correct_p_Kp} and  \eqref{correct_p_Ku}, and $q$ is a solution of the two-scale problem \eqref{two-scale_qf}. Incorporating the expression  \eqref{form_p1} for $p_e^1$ into equations \eqref{pe_macro_11} and  considering  \eqref{eq_ue_macro_1} we obtain that  $p_e$ and  $u_e$ satisfy the macroscopic problem  \eqref{macro_ue}--\eqref{macro_pe},
where   ${\bf E}^{\rm hom}$, $K_p^{\rm hom}$, and  $K_u$ are  defined by \eqref{efective}.  The coupling with the flow velocity   $\partial_t u_f$ is reflected in the interaction  function $Q$,  defined by \eqref{q_uf}.  Notice that since ${\rm div}\,  \partial_t u_f = 0 $ in $Y_f$, we have that $\int_\Gamma \partial_t u_f \cdot n  \, d\gamma = 0$ and the problem \eqref{two-scale_qf} is well posed, i.e.\  the compatibility condition is satisfied.
\end{proof}

\section{Strong two-scale convergence of $\be(u^\ve_e)$, $\nabla p_e^\ve$,  and $\partial_t u_f^\ve$}\label{strong_convergence}

\begin{lemma} For a   subsequences of  solutions of microscopic problem  \eqref{equa_cla}--\eqref{exbou_co}  $\{u^\ve_e\}$, $\{ p_e^\ve \}$  and $\{\partial_t u_f^\ve\}$  (denoted again by $\{u^\ve_e\}$, $\{ p_e^\ve \}$, and $\{\partial_t u_f^\ve\}$)  and the  limit functions $u_e$, $u_e^1$, $p_e$, $p_e^1$, and $\partial_t u_f$  as in Lemma~\ref{convergence_u_p} we have
\begin{equation}
\begin{aligned}
& \nabla u_e^\ve \to  \nabla u_e + \nabla_y u_e^1 \qquad && \text{ strongly two-scale} , \\
& \nabla p_e^\ve \to  \nabla p_e + \nabla_y p_e^1 && \text{ strongly two-scale} , \\
&\partial_t u_f^\ve \to  \partial_t u_f && \text{ strongly two-scale}, \\
&\ve \,  \be(\partial_t u_f^\ve) \to  \be_y(\partial_t u_f) && \text{ strongly two-scale}.
\end{aligned}
\end{equation}
\end{lemma}
\begin{proof} To show the  strong two-scale convergence we prove  the convergence of  the energy  functional related to equations \eqref{equa_cla} for $u_e^\ve$, $p_e^\ve$, and $\partial_t u_f^\ve$.
Because of the dependence of ${\bf E}$ on the temporal variable we have to consider a modified  form of the energy  functional.
We consider a monotone decreasing function  $\zeta:\mathbb R_+\to \mathbb R_+$, e.g.\ $\zeta(t)= e^{-\gamma t}$ for $t \in \mathbb R_{+}$,  and define    the energy functional for the microscopic problem \eqref{equa_cla}--\eqref{exbou_co} as
\begin{equation}\label{function_def_1}
\begin{aligned}
\mathcal E^\ve(u_e^\ve, p_e^\ve, \partial_t u^\ve_f) = \frac 12 \rho_e\|\partial_t u_e^\ve(s)\zeta(s) \|^2_{L^2(\Omega_e^\ve)} -  \langle \zeta^\prime \zeta,  \rho_e|\partial_t u_e^\ve|^2 \rangle_{\Omega_{e,s}^\ve} \qquad \qquad \\
 + \frac 12  \langle {\bf E}^\ve(b^\ve_{e,3}) \,\be(u^\ve_e(s))\, \zeta(s), \be(u_e^\ve(s))\, \zeta(s) \rangle_{\Omega_{e}^\ve}\qquad \qquad   \\
  -  \frac 12 \left\langle \big[ 2 \zeta^\prime  \zeta\,  {\bf E}^\ve(b^\ve_{e,3})  +  \zeta^2\,  \partial_t {\bf E}^\ve(b^\ve_{e,3})\big]  \be(u^\ve_e),  \be(u^\ve_e)\right\rangle_{\Omega_{e, s}^\ve}
\\  + \frac 12\rho_p \|p_e^\ve(s)\zeta(s)\|^2_{L^2(\Omega_e^\ve)} -  \langle \zeta^\prime \zeta, \rho_p |p_e^\ve|^2 \rangle_{\Omega_{e,s}^\ve}
+ \langle K_p^\ve \nabla p_e^\ve \, \zeta, \nabla p_e^\ve\,  \zeta\rangle_{\Omega_{e,s}^\ve} \qquad
\\
 + \frac 12\rho_f \|\partial_t u_f^\ve(s)\zeta(s)\|^2_{L^2(\Omega_f^\ve)}  -  \langle \zeta^\prime  \zeta, \rho_f |\partial_t u_f^\ve|^2 \rangle_{\Omega_{f, s}^\ve}
+ \mu \| \ve  \zeta \, \be(\partial_t u_f^\ve)\|^2_{L^2(\Omega_{f,s}^\ve)}
\end{aligned}
\end{equation}
for  $s \in (0,T]$. Considering  $(\partial_t u_e^\ve\, \zeta^2, p_e^\ve\, \zeta^2, \partial_t u_f^\ve\, \zeta^2)$ as a test function in  \eqref{weak_u_ef} we  obtain  the   equality
\begin{equation}\label{strong2scale_1}
\begin{aligned}
&\mathcal E^\ve(u_e^\ve, p_e^\ve, \partial_t u^\ve_f) = \frac 12\rho_e \|\partial_t u_e^\ve(0)\|^2_{L^2(\Omega_e^\ve)} \\&\hskip 2cm
+ \frac 12\left\langle {\bf E}^\ve(b^\ve_{e,3}) \, \be(u^\ve_e(0)) , \be(u_e^\ve(0)) \right\rangle_{\Omega_{e}^\ve}
+ \frac 12\rho_p \| p_e^\ve(0)\|^2_{L^2(\Omega_e^\ve)}
\\& \hskip 2cm+\frac 12\rho_f \|\partial_t u_f^\ve(0)\|^2_{L^2(\Omega_f^\ve)}
   + \langle F_u, \partial_t u_e^\ve \rangle_{(\partial \Omega)_s} + \langle F_p,  p_e^\ve \rangle_{(\partial \Omega)_s}.
\end{aligned}
\end{equation}
Due to the assumptions on  ${\bf E}$ and $\partial_t {\bf E}$ there exists a positive constant  $\gamma$ such that
 $$
 \begin{array}{c}
  \left(2 \gamma  {\bf E}(y,  \xi) -  \partial_t {\bf E}(y, \xi) \right) A \cdot A  \geq 0  \; \text{ for all symmetric matrices } A,\\[2mm]  \text{ all continuous bounded functions }  \xi, \text{ and } y \in Y.
  \end{array}
 $$

Since $\{b^\ve_e\}$ converges strongly in $L^2(\Omega_T)$,  $\be( u_e^\ve)$ converges weakly two-scale,  and ${\bf E}^\ve(b^\ve_{e,3})$ is uniformly bounded,  we have  the  weak two-scale convergence of the sequence $({\bf E}^\ve(b^\ve_{e,3}))^{\frac 1 2} \be( u_e^\ve)$  to $ ({\bf E}(y,b_{e,3}))^{\frac 1 2} (\be( u_e)+\be_y(u_e^1)) $  and  of $(2\gamma {\bf E}^\ve(b^\ve_{e,3}) -  \partial_t {\bf E}^\ve(b_{e,3}^\ve))^{\frac 1 2} \be( u_e^\ve)$ to $ (2\gamma {\bf E}(y,b_{e,3}) -  \partial_t {\bf E}(y, b_{e,3}))^{\frac 1 2} (\be( u_e)+\be_y(u_e^1))$,  as $\ve \to 0$.
Using in \eqref{function_def_1} and  \eqref{strong2scale_1} the lower-semicontinuity of the corresponding norms, the initial conditions for $u_e^\ve$, $p_e^\ve$, and $\partial_t u_f^\ve$, and the convergence of $\partial_t u_e^\ve$, $p_e^\ve$, and $\partial_t u_f^\ve$   implies
\begin{equation}\label{lim_functional_1}
\begin{aligned}
&\phantom{+\, } \rho_e \|\partial_t u_e(s)\zeta(s) \|^2_{L^2(\Omega\times Y_e)} + 2\gamma  \rho_e \|\partial_t u_e\, \zeta\|^2_{L^2(\Omega_s\times Y_e)} \\
 & +   \left\langle {\bf E}(y, b_{e,3}) \zeta^2(s) (\be(u_e(s))+ \be_y(u_e^1(s)) , \be(u_e(s)) + \be_y(u_e^1(s))\right\rangle_{\Omega \times Y_e}
 \\
 & +   \langle \zeta^2 \big(2 \gamma\,   {\bf E}(y, b_{e,3}) -   \partial_t{\bf E}(y, b_{e,3})\big) ( \be(u_e)+ \be_y(u_e^1)),  \be(u_e)+  \be_y(u_e^1)\rangle_{\Omega_{s}\times Y_e}
\\
& + \rho_p \|p_e(s)\zeta(s)\|^2_{L^2(\Omega\times Y_e)}
+ 2 \gamma \rho_p\|p_e \, \zeta \|^2_{L^2(\Omega_{s}\times Y_e)}
\\ & + 2 \langle  \zeta^2 K_p (\nabla p_e+ \nabla_y p_e^1) , \nabla p_e + \nabla_y p_e^1\rangle_{\Omega_{s}\times Y_e}
 + \rho_f \|\partial_t u_f(s)\zeta(s)\|^2_{L^2(\Omega\times Y_f)}  \\
 &+ 2 \gamma \rho_f \|\partial_t u_f \, \zeta \|^2_{L^2(\Omega_{s}\times Y_f)}
+2 \mu
\|  \be_y(\partial_t u_f)\, \zeta\|^2_{L^2(\Omega_s \times Y_f)}
 \\ &\leq 2|Y|\liminf \limits_{\ve \to 0} \mathcal E^\ve(u_e^\ve, p_e^\ve, \partial_t u^\ve_f)
 \leq 2|Y|\limsup\limits_{\ve \to 0} \mathcal E^\ve(u_e^\ve, p_e^\ve, \partial_t u^\ve_f)  \\
 &
   =  \langle {\bf E}(y, b_{e,3}) (\be(u_{e0}) +  \be_y(\hat u_{e0})), \be(u_{e0}) +  \be_y(\hat u_{e0}) \rangle_{\Omega \times Y_{e}}
\\
&+\rho_e \|u^1_{e0}\|^2_{L^2(\Omega\times Y_e)}  + \rho_p \|p_{e0}\|^2_{L^2(\Omega\times Y_e)}     +
\rho_f \|u^1_{f0}\|^2_{L^2(\Omega\times Y_f)}  \\
& + 2|Y| \langle F_u,  \partial_t u_e \zeta^2 \rangle_{(\partial \Omega)_s}+ 2|Y| \langle F_p,  p_e  \zeta^2  \rangle_{(\partial \Omega)_s}
\end{aligned}
\end{equation}
for $s \in (0,T]$. Here we used the  weak and the weak two-scale  convergences of $\partial_t u_e^\ve$, $\be(u_e^\ve)$,
$\be(\partial_t u_e^\ve)$,  $p_e^\ve$, and $\nabla p_e^\ve$, and the weak two-scale convergence of $\partial_t u_f^\ve$ and $\ve\,  \be(\partial_t u_f^\ve)$.
Considering the limit equations  \eqref{macro_ue}--\eqref{macro_two-scale_uf} for  $u_e$,  $p_e$, and $\partial_t u_f$ and taking $(\partial_t u_e \, \zeta^2, p_e\, \zeta^2, \partial_t u_f\, \zeta^2)$ as a test function  yields
\begin{equation}\label{macro_test_1}
\begin{aligned}
&\;  \frac 12\rho_e \| \partial_t u_e(s) \zeta(s)\|^2_{L^2(\Omega \times Y_e)}  - \frac 12\rho_e \| \partial_t u_e(0) \|^2_{L^2(\Omega \times Y_e)}+ \gamma  \rho_e\| \partial_t u_e \zeta\|^2_{L^2(\Omega_{s}\times Y_e)}
\\
&+ \langle {\bf E}(y, b_{e,3}) \big(\be( u_e)+ \be_y(u_e^1)\big),  \be(\partial_t u_e)\,  \zeta^2 \rangle_{\Omega_{s}\times Y_e}  +
\langle  \nabla p_e+ \nabla_y p_e^1,  \partial_t u_e \, \zeta^2 \rangle_{\Omega_{s}\times Y_e}
\\
& +
\frac 12\rho_p \|p_e(s) \zeta^2(s)\|^2_{L^2(\Omega\times Y_e)}
- \frac 12\rho_p \|  p_e(0) \|^2_{L^2(\Omega \times Y_e)}+ \gamma\rho_p  \| p_e \zeta\|^2_{L^2(\Omega_{s}\times Y_e)} \\
& + \langle K_p(x,y)( \nabla p_e+ \nabla_y p_e^1)  -  \partial_t u_e,  \nabla p_e \zeta^2 \rangle_{\Omega_{s}\times Y_e}
+\frac 12 \rho_f\|\partial_t u_f(s)\zeta(s)\|^2_{L^2(\Omega\times Y_f)}
 \\ & - \frac 12 \rho_f\|\partial_t u_f(0)\|^2_{L^2(\Omega\times Y_f)} +  \gamma\rho_f  \| \partial_t u_f \zeta\|^2_{L^2(\Omega_{s}\times Y_f)} -  \langle   p_e^1,  \partial_t u_f \cdot n  \, \zeta^2\rangle_{\Omega_{s}\times \Gamma}
 \\ & + \mu
\langle \be_y(\partial_t u_f),   \be_y(\partial_t u_f)\zeta^2 \rangle_{\Omega_{s}\times Y_f}    =|Y| \langle F_u, \partial_t u_e \, \zeta^2  \rangle_{(\partial \Omega)_s} +  |Y|\langle F_p, p_e \, \zeta^2 \rangle_{(\partial \Omega)_s}
\end{aligned}
\end{equation}
for $s \in (0,T]$. From  equation \eqref{correctos_pe} for the corrector $p_e^1$ we obtain
\begin{equation}\label{macro_corrector_pe1}
\begin{aligned}
 -  \langle   p_e^1,  \partial_t u_f \cdot n \, \zeta^2\rangle_{\Omega_{s}\times \Gamma} =  \langle K_p(x,y)( \nabla p_e+ \nabla_y p_e^1)  -  \partial_t u_e,  \nabla_y p_e^1 \, \zeta^2 \rangle_{\Omega_{s}\times Y_e}.
\end{aligned}
\end{equation}
Considering    equation \eqref{correctos_ue} for the corrector $u_e^1$ and taking   $\partial_t u_e^1 \,  \zeta^2$ as a test function yields
\begin{equation}\label{macro_corrector_ue1}
\begin{aligned}
&\; \;  \; \; \langle {\bf E}(y, b_{e,3}) \big(\be( u_e)+ \be_y(u_e^1)\big),  \be(\partial_t u_e) \zeta^2 \rangle_{\Omega_{s}\times Y_e} \\
& = \langle {\bf E}(y, b_{e,3}) \big(\be( u_e)+ \be_y(u_e^1)\big),  (\be(\partial_t u_e) +  \be_y(\partial_t u_e^1))\, \zeta^2 \rangle_{\Omega_{s}\times Y_e}  \\
& =
\frac 12 \left \langle {\bf E}(y, b_{e,3}) \big(\be(u_e(s))+ \be_y(u_e^1(s))\big)\zeta^2(s) ,  \be(u_e(s)) +  \be_y(u_e^1(s)) \right\rangle_{\Omega\times Y_e} \\
&  -
\frac 12 \left \langle {\bf E}(y, b_{e,3}) \big(\be( u_e(0))+ \be_y(u_e^1(0))\big),   \be(u_e(0)) +  \be_y(u_e^1(0))\right \rangle_{\Omega\times Y_e}\\
& + \frac 12 \left\langle \big(2 \gamma {\bf E}(y, b_{e,3}) -  \partial_t {\bf E}(y, b_{e,3}) \big) \zeta^2\big (\be( u_e)+ \be_y(u_e^1)\big),  \be(u_e) +  \be_y(u_e^1) \right\rangle_{\Omega_s\times Y_e}.
\end{aligned}
\end{equation}
Combining \eqref{macro_test_1}--\eqref{macro_corrector_ue1}  with  \eqref{lim_functional_1}  and using that $\be_y(u_e^1(0))= \be_y(\hat u_{e0})$  in $\Omega\times Y_e$ we obtain
\begin{equation*}
\begin{aligned}
\mathcal E(u_e, p_e, \partial_t u_f) \leq \liminf\limits_{\ve\to 0} \mathcal E^\ve(u_e^\ve, p_e^\ve, \partial_t u^\ve_f)
\leq \limsup\limits_{\ve\to0} \mathcal E^\ve(u_e^\ve, p_e^\ve, \partial_t u^\ve_f) =  \mathcal E(u_e, p_e, \partial_t u_f)
\end{aligned}
\end{equation*}
 and thus conclude  that  $ \lim\limits_{\ve\to 0} \mathcal E^\ve(u_e^\ve, p_e^\ve, \partial_t u^\ve_f)=\mathcal E(u_e, p_e, \partial_t u_f)$. Then
 the strong two-scale convergence relations stated in  Lemma follow by lower semicontinuity arguments.
\end{proof}

\section{Derivation of macroscopic equations for  reaction-diffusion-convec\-tion problem}\label{macro_diffusion}
The homogenized coefficients in the reaction-diffusion-convection equations,  which will be obtained in the derivation of the macroscopic problem,  are defined by
\begin{equation}\label{dhom_uhom}
\begin{aligned}
& D_{b,ij}^{\rm hom} =\!\frac 1 {|Y|} \int_{Y_e}\! \big[ D^{ij}_b + \big(D_b\nabla_y \omega^j_b (y)\big)_i \big] dy, \\ &
D_{ij}^{\rm hom} = \dashint_{Y} \big[ D^{ij}(y) + \big(D(y)\nabla_y \omega^j (y)\big)_i \big] dy, \\
& {\rm v}_f(t,x) = \frac 1 {|Y|}\int_{Y_f} \big[\mathcal G(\partial_t u_f(t,x,y)) -  D_f \nabla_y z(t,x,y) \big]dy,
\end{aligned}
\end{equation}
with $\omega_b$ and $\omega$ being   $Y$-periodic solutions of the unit cell problems
\begin{equation}\label{unit_cell_prob_b}
\begin{aligned}
{\rm div} \big( D_b (\nabla_y \omega^j_b(y) + e_j)\big) = 0 \quad \text{ in } Y_e,\\
D_b (\nabla_y \omega^j_b(y) + e_j) \cdot n = 0 \quad \text{ on } \Gamma,
\end{aligned}
\end{equation}
and
\begin{equation}\label{unit_cell_prob_c}
\begin{aligned}
&{\rm div}_y ( D(y) (\nabla_y \omega^j + e_j)) =0 \quad &&  \text{ in } Y\setminus \tilde \Gamma, \\
&D_e (\nabla_y \omega^j_e + e_j)\cdot n =0,  \quad  D_f (\nabla_y \omega^j_f + e_j)\cdot n =0  && \text{ on } \widetilde \Gamma,
\end{aligned}
\end{equation}
where $\omega_e^j(y) = \omega^j(y)$ for $y\in Y_e$ and $\omega_f^j(y) = \omega^j(y)$ for $y\in Y_f$,  and  $z$   is a $Y$-periodic solution of
\begin{equation}\label{unit_cell_prob_cf}
\begin{aligned}
&{\rm div}_y ( D_f \nabla_y z - \mathcal G(\partial_t u_f)) =0 \quad && \text{ in } Y_f, \\
& ( D_f \nabla_y z - \mathcal G(\partial_t u_f))\cdot n = 0 \quad && \text{ on } \Gamma.
\end{aligned}
\end{equation}
Here  $D(y) = \begin{cases} D_e \quad \text{ in } Y_e\\
D_f \quad \text{ in } Y_f
\end{cases}. $
  Notice that the definition of $D_b^{\rm hom}$ and $D^{\rm hom}$, and the fact that $D_b^{jj} >0$, with $j=1,2,3$,  $D_e>0$, $D_f>0$, and $\omega_b^{j}$, $\omega^{j}$ are solutions of the unit cell problems
\eqref{unit_cell_prob_b} and \eqref{unit_cell_prob_c} ensure that  $D_b^{\rm hom}$ and $D^{\rm hom}$ are positive definite.

Next we derive macroscopic equations for  the limit functions  $b_e$ and $c$ defined in \eqref{convergences_111}.
The main difficulty  in  the proof is to show the convergence of the non-linear functions depending on  the displacement gradient.
\begin{theorem}
Solutions of the microscopic problem  \eqref{eq_codif},~\eqref{exbou_co}  converge to solutions  $b_e, c \in L^2(0,T; H^1(\Omega))$
of the macroscopic equations
\begin{equation}\label{macro_concentration}
\begin{aligned}
&\vartheta_e \, \partial_t b_e - {\rm div} ( D_{b }^{\rm hom} \nabla b_e )
 \\ &  \hspace{1.6 cm } =  \vartheta_e  \dashint_{Y_e} g_b(c, b_e, W(b_{e,3},y) \, \be(u_e)) dy + \vartheta_\Gamma \,  P(b_e)  \quad && \text{\rm in } \Omega_T, \\
& \partial_t c - {\rm div} ( D^{\rm hom} \nabla c  -{\rm v}_f \,  c)
 \\ & \hspace{1.6 cm }  =\vartheta_f \, g_{f}(c)+ \vartheta_e \, \dashint_{Y_e} g_e(c, b_e,   W(b_{e,3}, y) \, \be(u_e)) dy  && \text{\rm in } \Omega_T,\\
& D_{b}^{\rm hom} \nabla b_e   \cdot n = F_b(b_e)  && \text{\rm on }  (\partial \Omega)_T, \\
& ( D^{\rm hom} \nabla c  - {\rm v}_f \, c)  \cdot n = F_c(c)  \quad && \text{\rm on }  (\partial \Omega)_T, \\
& b(0,x) = b_0(x), \quad  c(0,x)= c_0(x) \quad && \text{\rm in }  \Omega,
\end{aligned}
\end{equation}
where
\begin{equation}\label{defi_w}
W(b_{e,3},y)=\big\{W_{klij}(b_{e,3}, y)\big\}_{k,l,i,j=1}^3 =\big\{{b}_{kl}^{ij} + (\be_y(w^{ij}(b_{e,3},y)))_{kl}\big\}_{k,l,i,j=1}^3
\end{equation}
with $w^{ij}$ being  solutions of the  unit cell problems  \eqref{unit_cell_ue_w}, and  ${\bf b}_{kl}  = e_k \otimes e_l$,  $\{e_k\}_{k=1}^3$ is the  canonical basis of $\mathbb R^3$.    \\
Here  $\vartheta_e=|Y_e|/|Y|$, $\vartheta_f=|Y_f|/|Y|$,  and $\vartheta_\Gamma = |\Gamma|/|Y|$.
We have the convergence in the following sense
\begin{equation*}
\begin{aligned}
& b_e^\ve \rightarrow b_e, \;  \; \;  \qquad &&\overline c^\ve \to c && \text{ strongly in } L^2(\Omega_T),  \\
& \nabla b_e^\ve \rightharpoonup \nabla b_e + \nabla_y b_e^1,  &&\nabla c^\ve \rightharpoonup \nabla c + \nabla_y c^1 && \text{ weakly two-scale}.
\end{aligned}
\end{equation*}
 \end{theorem}
\begin{proof}
We can rewrite  the  microscopic equation for $b_e^\ve$ as
\begin{equation}\label{cd_two_scale_b}
\begin{aligned}
& -\langle   b^\ve_e  \, \chi\big._{\Omega_e^\ve} ,\partial_t \varphi_1\rangle_{\Omega_T}+\langle D^\ve_b \nabla b^\ve_e, \nabla\varphi_1\,  \chi\big._{\Omega_e^\ve}  \rangle_{\Omega_{T}} -
\langle b_{e0}, \varphi_1\chi\big._{\Omega_e^\ve} \rangle_{\Omega_T}
\\ &
\qquad \qquad = \langle g_b(c^\ve_e,b_e^\ve, \be(u_e^\ve)), \varphi_1\,  \chi\big._{\Omega_e^\ve}  \rangle_{\Omega_{T}}
 + \ve \langle  P(b^\ve_e), \varphi_1 \rangle_{ \Gamma^\ve_T} +  \langle  F_b(b_e^\ve),  \varphi_1 \rangle_{(\partial \Omega)_T}
\end{aligned}
\end{equation}
with  $\varphi_1 = \phi_1(t,x) + \ve \phi_2(t,x, x/\ve)$, where $\phi_1 \in C^\infty(\overline\Omega_T)$ is such that  $\phi_1(T,x)=0$ for $x\in \overline\Omega$,  and $\phi_2 \in C^\infty_0(\Omega_T; C^\infty_{\rm per} (Y))$, and $\chi\big._{\Omega_e^\ve}$ is the characteristic function of $\Omega_e^\ve$.
Taking into account  the strong convergence of $b_e^\ve$ and $c_e^\ve$ and the two-scale convergence of $\nabla b_e^\ve$ and $\nabla c_e^\ve$, see Lemma~\ref{convergence_b_c},   together with  the strong two-scale convergence of $\be(u_e^\ve)$ we obtain
\begin{equation}\label{corrector_b_11}
\begin{aligned}
 -\langle |Y_e|   b_e,  \partial_t\phi_1 \rangle_{\Omega_T}+\langle D_b (\nabla b_e + \nabla_y b_e^1), \nabla\phi_1 + \nabla_y \phi_2 \rangle_{\Omega_{T}\times Y_e} \\
 -  \langle |Y_e|  b_{e0}, \phi_1 \rangle_{\Omega_T}
= \langle g_b(c,b_e, \be(u_e)+ \be_y(u_e^1)), \phi_1 \rangle_{\Omega_{T}\times Y_e}\\
 + \langle  P(b_e), \phi_1 \rangle_{\Omega_T \times \Gamma} +  |Y| \langle  F_b(b_e),  \phi_1 \rangle_{ (\partial \Omega)_T}.
\end{aligned}
\end{equation}
Here we used the fact that due to the strong two-scale convergence of $\be(u_e^\ve)$ we have
$$
\lim\limits_{\ve \to 0} \|\mathcal T^\ast_\ve({\bf e}(u_e^\ve)) -{\bf e}(u_e) -{\bf  e}_y(u_e^1)\|_{L^2(\Omega_{T}\times Y_e)} =0,
$$
where $\mathcal T^\ast_\ve$ is the periodic unfolding operator for the perforated domain $\Omega_e^\ve$, see e.g.\ \cite{CDDGZ}.
Assumptions on $g_b$ in {\bf A4.} and the \textit{a priori} estimates for  $c^\ve$, $b_e^\ve$, and $u_e^\ve$  ensure
\begin{equation}\label{estim_converg_g_b}
\begin{aligned}
& \|g_b(  \mathcal T_\ve^\ast(c^\ve_e),  \mathcal T_\ve^\ast(b^\ve_e),  \mathcal T_\ve^\ast (\be(u_e^\ve )))  - g_b(c, b_e, \be(u_e) + \be_y(u_e^1)) \|_{L^1(\Omega_T\times Y_e)}\\
&\qquad  \qquad \leq
C_1  \Big(  \| \mathcal T_\ve^\ast(c^\ve_e) -  c\|_{L^2(\Omega_T\times Y_e)}  +  \| \mathcal T_\ve^\ast(b^\ve_e) -  b_e\|_{L^2(\Omega_T\times Y_e)}
  \\
  & \hspace{3.4 cm }  +  \| \mathcal T_\ve^\ast(\be(u^\ve_e)) -  \be(u_e)  - \be_y(u_e^1)  \|_{L^2(\Omega_T\times Y_e)}  \Big),\\
&  \|g_b( c^\ve_e,  b^\ve_e,  \be(u_e^\ve ))  \|_{L^2(\Omega^\ve_{e,T})} \leq C_2,
 \end{aligned}
\end{equation}
where  $C_1= C_1\big(  \| \mathcal T_\ve^\ast (\be(u^\ve_e)) \|_{L^2(\Omega_T\times Y_e)}, \|\be(u_e)  +\be_y(u_e^1) \|_{L^2(\Omega_T\times Y_e)}, \| \mathcal T_\ve^\ast(c^\ve_e)\|_{L^2(\Omega_T\times Y_e)},\\   \| \mathcal T_\ve^\ast(b^\ve_e)\|_{L^2(\Omega_T\times Y_e)}, \|c\|_{L^2(\Omega_T)}, \,  \|b_e\|_{L^2(\Omega_T)}\big)$ and the constants $C_1$ and $C_2$ are independent of $\ve$.
Combining the estimates in  \eqref{estim_converg_g_b},   the definition of $\Omega_e^\ve$, and    the  strong convergence of $c^\ve_e$ and $b^\ve_e$ in $L^2(\Omega_T)$ and of $\mathcal T_\ve^\ast (\be(u^\ve_e))$  in $L^2(\Omega_T\times Y_e)$, along with  the properties of the unfolding operator,  we obtain
\begin{eqnarray}\label{converg_gb_two-scale}
 \lim\limits_{\ve \to 0} \!\int_{\Omega_{e,T}^\ve} \! g_b(  c^\ve_e,  b^\ve_e,  \be(u_e^\ve ))  \psi(t,x,  x/ \ve)  dx dt  =
\frac 1 {|Y|}\!\int\limits_{\Omega_{T}}\! \int\limits_{Y_e} g_b( c,  b_e,  \be(u_e ) + \be_y(u_e^1)) \psi  dy dx dt  \nonumber \\
 + \!\frac 1 {|Y|}\! \lim\limits_{\ve \to 0} \!  \int\limits_{\Omega_{T}}\!\!\int\limits_{Y_e}\! \big[g_b(  \mathcal T_\ve^\ast(c^\ve_e),  \mathcal T_\ve^\ast(b^\ve_e),  \mathcal T_\ve^\ast (\be(u_e^\ve ))) - g_b(c, b_e, \be(u_e) + \be_y(u_e^1)) \big]  \mathcal T_\ve^\ast(\psi)   dy dx dt  \nonumber \\
 = \frac 1 {|Y|}\int_{\Omega_{T}}\int_{Y_e} g_b( c,  b_e,  \be(u_e ) + \be_y(u_e^1))\,  \psi\,   dy dx dt
\end{eqnarray}
for all $\psi \in C^\infty_0(\Omega_T; C_{\rm per} (Y))$.  Thus using the estimate for $\|g_b( c^\ve_e, b^\ve_e,  \be(u_e^\ve ))  \|_{L^2(\Omega^\ve_{e,T})}$  in   \eqref{estim_converg_g_b} we conclude
$$
g_b(c_e^\ve,b_e^\ve, {\bf e}(u_e^\ve)) \to g_b(c_e,b_e,{\bf  e}(u_e)+ {\bf e}_y(u_e^1)) \quad \text{ two-scale} .
$$
To show the convergence of  the  boundary integral over $\Gamma^\ve$ we used the Lipschitz continuity of $P$ and the trace estimate
\begin{equation}
\begin{aligned}
\ve \| b_e^\ve - b_e \|^2_{L^2(\Gamma^\ve_T)}  & \leq C_1 \Big(  \|b_e^\ve - b_e \|^2_{L^2(\Omega^\ve_{e,T})}  + \ve^2  \|\nabla (b_e^\ve - b_e) \|^2_{L^2(\Omega^\ve_{e,T})} \Big) \\ &  \leq C_2  \Big(  \|b_e^\ve - b_e \|^2_{L^2(\Omega^\ve_{e,T})}  + \ve^2  \big[  \|\nabla b_e^\ve\|_{L^2(\Omega^\ve_{e,T})} +\| \nabla b_e \|^2_{L^2(\Omega^\ve_{e,T})} \big] \Big) .
\end{aligned}
\end{equation}
Then  due to the strong convergence of $b^\ve_e$  in $L^2(\Omega_T)$, the regularity of $b_e$, i.e.\ $b_e \in L^2(0, T; H^1(\Omega))$,  and  the boundedness of $\nabla b^\ve_e$ in $L^2(\Omega_{e,T}^\ve)$, uniformly in $\ve$, we obtain
$$
\lim\limits_{\ve \to 0} \ve\| P(b_e^\ve) - P(b_e) \|^2_{L^2(\Gamma^\ve_T)}  \leq  C \lim\limits_{\ve \to 0} \ve\,  \| b_e^\ve - b_e \|^2_{L^2(\Gamma^\ve_T)}   =0.
$$
Taking in \eqref{corrector_b_11} first $\phi_1\equiv 0$ and then $\phi_2\equiv 0$ and considering such $\phi_1$ that $\phi_1(0) =0$ we obtain macroscopic equations  for $b_e$ in~\eqref{macro_concentration}.
The standard arguments for parabolic equations imply that $\partial_t b_e \in L^2(0, T; H^1(\Omega)^\prime)$.
 Combining this with  the fact that $b_e\in L^2(0, T; H^1(\Omega))$, see Lemma~\ref{convergence_b_c},  we conclude that $b_e \in C([0,T]; L^2(\Omega))$.
Then from \eqref{corrector_b_11}  we obtain that  $b_e$  satisfies the  initial condition.

The properties of $\Omega_f^\ve$ and of  the unfolding operator $\mathcal T^\ast_{\ve, f}$ for the domain $\Omega_f^\ve$ yield
\begin{equation*}
\begin{aligned}
& \lim\limits_{\ve \to 0 } \int_{\Omega_{f,T}^\ve} \mathcal G(\partial_t u_f^\ve )\, \psi(t,x, x/\ve) \, dx dt =\lim\limits_{\ve \to 0 } \frac 1{|Y|} \int_{\Omega_T}\int_{Y_f} \mathcal G( \mathcal T_{\ve,f}^\ast (\partial_t u_f^\ve )) \mathcal T^\ast_{\ve,f} ( \psi ) dy dx dt
\\& \quad =  \frac 1 {|Y|}\int_{\Omega_T}\int_{Y_f} \mathcal G(\partial_t u_f) \psi \, dy dx dt
\\ &\quad +\frac 1 {|Y|} \lim\limits_{\ve \to 0 }   \int_{\Omega_T}\int_{Y_f} \left[\mathcal G( \mathcal T_{\ve,f}^\ast (\partial_t u_f^\ve ))  - \mathcal G(\partial_t u_f) \right]\mathcal T^\ast_{\ve,f} (\psi) \, dy dx dt
\end{aligned}
\end{equation*}
for all $\psi \in C_0^\infty(\Omega_T; C_{\rm per}(Y))$.  Using the Lipschitz continuity of $\mathcal G$ and the strong convergence of $\mathcal T_{\ve,f}^\ast (\partial_t u_f^\ve )$, ensured by the strong two-scale convergence of $\partial_t u_f^\ve$, we obtain
\begin{equation*}
\begin{aligned}
& \lim\limits_{\ve \to 0 }   \int_{\Omega_T\times Y_f} \left[\mathcal G( \mathcal T_{\ve,f}^\ast (\partial_t u_f^\ve ))  - \mathcal G(\partial_t u_f) \right] \mathcal T_{\ve,f}^\ast (\psi) dy dx dt
\\
& \leq  C \lim\limits_{\ve \to 0 }  \|  \mathcal T_{\ve,f}^\ast (\partial_t u_f^\ve ) - \partial_t u_f \|_{L^2(\Omega_T\times Y_f)} \|\psi \|_{L^2(\Omega_T\times Y_f)}  =0.
 \end{aligned}
\end{equation*}
Thus taking into account  the boundedness of  $ \mathcal G(\partial_t u_f^\ve )$ we conclude
\begin{equation*}
\begin{aligned}
& \mathcal G(\partial_t u_f^\ve )\to \mathcal G(\partial_t u_f) && \text{ two-scale}.
\end{aligned}
\end{equation*}
In the same way as for $g_b$, the assumptions in {\bf A4.} ensure
\begin{equation}\label{estim_converg_g_e}
\begin{aligned}
& \|g_e(  \mathcal T_\ve^\ast(c^\ve_e),  \mathcal T_{\ve}^\ast(b^\ve_e),  \mathcal T_\ve^\ast (\be(u_e^\ve )))  - g_e(c, b_e, \be(u_e) + \be_y(u_e^1)) \|_{L^1(\Omega_T\times Y_e)}\\ &  \leq
C  \Big(  \| \mathcal T_{\ve}^\ast(c^\ve_e) -  c\|_{L^2(\Omega_T\times Y_e)}\\
&+  \| \mathcal T_{\ve}^\ast(b^\ve_e) -  b_e\|_{L^2(\Omega_T\times Y_e)}
  +  \| \mathcal T_{\ve}^\ast(\be(u^\ve_e)) -  \be(u_e)  - \be_y(u_e^1)  \|_{L^2(\Omega_T\times Y_e)}  \Big),
 \end{aligned}
\end{equation}
where $C = C\big(  \| \mathcal T_\ve^\ast (\be(u^\ve_e)) \|_{L^2(\Omega_T\times Y_e)},   \|\be(u_e)  +\be_y(u_e^1) \|_{L^2(\Omega_T\times Y_e)}, \| \mathcal T_\ve^\ast(c^\ve_e)\|_{L^2(\Omega_T\times Y_e)}, \\    \| \mathcal T_\ve^\ast(b^\ve_e)\|_{L^2(\Omega_T\times Y_e)},  \|c\|_{L^2(\Omega_T)},  \|b_e\|_{L^2(\Omega_T)} \big)$.  The {\it a priori} estimates for $c^\ve$, $b_e^\ve$, and $u_e^\ve$  and assumptions on $g$ in {\bf A4.}  imply
$$
\|g_e(  c^\ve_e,  b^\ve_e,  \be(u_e^\ve )) \|_{L^2(\Omega_{e,T}^\ve)} \leq C,
$$
with a constant $C$ independent of $\ve$.
Then  estimate \eqref{estim_converg_g_e} and  the  strong convergence of $c^\ve_e$ and $b^\ve_e$ in $L^2(\Omega_T)$ and of $\mathcal T_\ve^\ast (\be(u^\ve_e))$  in $L^2(\Omega_T\times Y_e)$,
together with calculations similar to \eqref{converg_gb_two-scale}, yield
\begin{equation*}
g_e(c_e^\ve, b_e^\ve, \be(u_e^\ve)) \to g_e(c, b_e, \be(u_e) + \be_y(u_e^1)) \quad  \text{ two-scale}.
\end{equation*}
\noindent Considering  $\varphi_2(t,x)= \psi_1(t,x) + \ve \psi_2\big(t,x, \frac x \ve\big)$,  with  
$\psi_1 \in C_0^\infty(0,T; C^\infty(\overline\Omega))$,  $\psi_2 \in C^\infty_0(\Omega_T; C^\infty_{\rm per} (Y\setminus \widetilde \Gamma))$, as a test function in
\eqref{cd_two} we obtain
\begin{equation*}\label{cd_two_scale_c}
\begin{aligned}
  -\langle    c^\ve_e  \chi_{\Omega_e^\ve},  \partial_t\varphi_2\rangle_{\Omega_T}+\langle D_c \nabla c^\ve_e, \nabla\varphi_2 \chi_{\Omega_e^\ve}  \rangle_{\Omega_{T}}
-\langle g_c(c^\ve_e,b_e^\ve, \be(u_e^\ve)), \varphi_2 \chi_{\Omega_e^\ve}  \rangle_{\Omega_{T}} \;  \qquad
\\  -\langle   c^\ve_f  \chi_{\Omega_f^\ve}, \partial_t \varphi_2  \rangle_{\Omega_T}
  +\langle D_f \nabla c^\ve_f - \mathcal{G}(\partial_t u^\ve_f) c_f^\ve, \nabla\varphi_2 \chi_{\Omega_f^\ve} \rangle_{\Omega_{T}}
-\langle g_f(c^\ve_f), \varphi_2 \chi_{\Omega_f^\ve} \rangle_{\Omega_{T}} \\
= \langle F_c(c^\ve_e),  \varphi_2 \rangle_{(\partial\Omega)_T}.
\end{aligned}
\end{equation*}
The two-scale  and  the strong convergences of $c_e^\ve$ and $c_f^\ve$ together with strong two-scale convergence of
$\be(u_e^\ve)$ and $\partial_t u_f^\ve$  ensure  that 
\begin{equation*}
\begin{aligned}
& - \langle |Y_e|  c,  \partial_t  \psi_1  \rangle_{\Omega_T}+\langle D_c (\nabla c+ \nabla_y c^1), \nabla\psi_1  + \nabla_y \psi_2  \rangle_{\Omega_{T}\times Y_e}
\\ & -\langle |Y_f|   c,  \partial_t \psi_1   \rangle_{\Omega_T}+
\langle D_f (\nabla c + \nabla_y c^1) - \mathcal G(\partial_t u_f) c, \nabla\psi_1 + \nabla_y \psi_2 \rangle_{\Omega_{T}\times Y_f}\\ &-\langle g_c(c, b_e, \be(u_e)+ \be(u_e^1)), \psi_1   \rangle_{\Omega_{T}\times Y_e} -\langle g_f(c), \psi_1  \rangle_{\Omega_{T}\times Y_f}=|Y| \langle F_c(c),  \psi_1 \rangle_{(\partial\Omega)_T}.
\end{aligned}
\end{equation*}
Letting $\psi_1= 0$ yields
\begin{equation}\label{corrector_c1}
\begin{aligned}
\langle D_c (\nabla c+ \nabla_y c_e^1),  \nabla_y \psi_2  \rangle_{\Omega_{T}\times Y_e}+
\langle D_f (\nabla c + \nabla_y c_f^1),  \nabla_y \psi_2 \rangle_{\Omega_{T}\times Y_f}
\\   -\langle\mathcal G(\partial_t u_f) c,  \nabla_y \psi_2  \rangle_{\Omega_{T}\times Y_f}=0,
\end{aligned}
\end{equation}
where $c_l^1(t,x,y)= c^1(t,x,y)$ for $y\in Y_l$ and  $(t,x) \in \Omega_T$, with $l=e,f$. Taking into account  the structure of  equation  \eqref{corrector_c1} we represent $c^1$ in the form
\begin{equation*}
\begin{aligned}
& c^1_e(t,x,y)  = \sum_{j=1}^3 \partial_{x_j} c(t,x)\,  \omega^j(y)  && \text{for } \; \;   (t,x) \in \Omega_T, \; y\in Y_e,\\
& c^1_f(t,x,y)  = \sum_{j=1}^3 \partial_{x_j} c(t,x)\,  \omega^j(y) + c(t,x) \, z(t,x,y) \quad  && \text{for } \; \;   (t,x) \in \Omega_T, \; y\in Y_f,
\end{aligned}
\end{equation*}
where  $\omega^j$, with $j=1, 2, 3$, and $z$ are solutions of the unit cell problems
\eqref{unit_cell_prob_c} and \eqref{unit_cell_prob_cf}, respectively.
Then choosing $\psi_2 =0$ we obtain  the macroscopic equations for $c$ in \eqref{macro_concentration}.
\end{proof}

\section{Well-posedness of the limit problem. Uniqueness of a weak solution}\label{uniqueness_macro}
To ensure that the whole sequence of solutions of microscopic problem converges we shall prove the uniqueness of a solution of the limit  problem \eqref{macro_ue}--\eqref{macro_two-scale_uf}, \eqref{macro_concentration}.
In fact we are going to prove, using the contraction arguments, that the limit problem is well-posed and in particular has a unique solution.

\medskip\noindent
We consider an operator $\mathcal K$ on $L^\infty(0,T; H^1(\Omega))\times L^\infty(0,T; L^2(\Omega))$  given by  $(u^{j}_{e}, \partial_t u^{j}_{f})$  $= \mathcal K (u_{e}^{j-1},  \partial_t u^{j-1}_{f})$, where for given $(u_{e}^{j-1}, \partial_t u_{f}^{j-1})$ we first define $b^{j}_{e}, c^{j}$ as a solution of  \eqref{macro_concentration} with $(u_{e}^{j-1}, \partial_t u_{f}^{j-1})$ in place of
$(u_{e},\partial_t u_{f})$ and then  $(u^{j}_{e}, p_e^{j}, u_{f}^{j}, \pi_f^j)$ are  solutions of  \eqref{macro_ue}--\eqref{macro_two-scale_uf}, with $b_{e}^{j}$ in place of
$b_{e}$.
We denote $\widetilde c^j= c^{j} - c^{j-1}$, $\widetilde b^j_e= b^{j}_e - b^{j-1}_e$, $\widetilde u_{e}^{j-1}= u_e^{j-1} - u_e^{j-2}$, $\widetilde p_{e}^{j-1}= p_e^{j-1} - p_e^{j-2}$, and  $\widetilde u_{f}^{j-1}= u_f^{j-1} - u_f^{j-2}$.
To prove the existence of a unique solution of problem \eqref{macro_ue}--\eqref{macro_two-scale_uf}, \eqref{macro_concentration} we derive a contraction inequality  and show that the operator $\mathcal K$ has a fixed point.

First we obtain  estimates for solutions of the reaction-diffusion-convection system \eqref{macro_concentration}.
\begin{lemma}\label{estim_uniq_macro_bc}
Any two consecutive  iterations
$$
(u_{e}^{j-1},    \partial_t u_{f}^{j-1}), \ (b_{e}^j, c^j) \ \   \hbox{and}\ \   (u_{e}^{j-2},   \partial_t u_{f}^{j-2}), \ (b_{e}^{j-1}, c^{j-1})
 $$
 for the  limit  problem \eqref{macro_ue}--\eqref{macro_two-scale_uf},~\eqref{macro_concentration}   satisfy the following estimates
\begin{equation}
\begin{aligned}
&\| b_{e}^j \|_{L^\infty(0,T; L^\infty(\Omega))}  +\| c^j \|_{L^\infty(0,T; L^\infty(\Omega))}  + \| b_{e}^{j-1} \|_{L^\infty(0,T; L^\infty(\Omega))}
  \\ & \hskip 3cm+\| c^{j-1} \|_{L^\infty(0,T; L^\infty(\Omega))}  \leq C,  \\ &
  \|\widetilde  b_{e}^j\|_{L^\infty(0,s; L^\infty(\Omega))} + \| \widetilde c^j\|_{L^\infty(0, s; L^2(\Omega))}
  \\ & \hskip 3cm \leq C\big[ \|{\bf e}(\widetilde u_{e}^{j-1})\|_{L^{1 + \frac 1 \sigma}(0,s; L^2(\Omega))} +\| \partial_t \widetilde u_{f}^{j-1} \|_{L^2(\Omega_s\times Y_f)}  \big]
\end{aligned}
\end{equation}
with an arbitrary  $s\in (0,T]$  and   any  $0 < \sigma< 1/9$, the constant $C$ being independent of $s$.
\end{lemma}
\begin{proof}
The boundedness of $b_e^j$ and $b_e^{j-1}$ can be obtained in the same way as the corresponding estimate for $b^\ve_e$  in  \eqref{estim_b_c}.
To show the boundedness of $c^j$ we consider  $(c^j- M)^{+}$, where $M \geq \max\{ \|c_0\|_{L^\infty(\Omega)}, 1\}$,  as a test function in  the equation for $c^j$  in \eqref{macro_concentration}. Using assumptions {\bf A4.}  on $g_e$, $g_f$, and $F_c$ we obtain
\begin{equation*}
\begin{aligned}
&\|(c^j(s) - M)^+\|^2_{L^2(\Omega)} + \|\nabla (c^j-M)^+\|^2_{L^2(\Omega_s)}
\\ &\qquad \leq
 M\big[\| b_e^j\|_{L^\infty(\Omega_s)} +1\big]\|(c^j- M)^+\|_{L^1(\Omega_s)}\\ & \qquad
+M\big[\| {\rm v}^{j-1}_f\|_{L^\infty(\Omega_s)} +1\big]\|\nabla (c^j- M)^+\|_{L^1(\Omega_s)}\\
& \qquad
+ C\big[\| b_e^j\|_{L^\infty(\Omega_s)} + C_\delta \|{\rm v}_f^{j-1} \|_{L^\infty(\Omega_s)} + 1 \big] \|(c^j- M)^+\|^2_{L^2(\Omega_s)}  \\ & \qquad
+  \| {\bf e}(u_e^{j-1})\|_{L^\infty(0,s; L^2(\Omega))} \| W^{j}\|_{L^\infty(\Omega_s; L^2(Y_e))} \|(c^j- M)^+\|^2_{L^2(0,s; L^4(\Omega))}\\ & \qquad
+   \| {\bf e}(u_e^{j-1})\|^2_{L^2(\Omega_s)} \| W^{j}\|^2_{L^\infty(\Omega_s; L^2(Y_e))} (1+ M^2) \int_0^s |\Omega_M(t)|^{\frac  12} dt,
\end{aligned}
\end{equation*}
for   $s\in (0,T)$, where $\Omega_M(t) =\{ x \in \Omega:  c^j(t,x) > M \}$ for $t \in (0,T)$.
 Here  ${\rm v}^{j-1}_f$ is defined in  the following way:
first we replace $\partial_t u_f$ in the unit cell problem \eqref{unit_cell_prob_cf}  with  $\partial_t u_f^{j-1}$ to obtain $z^{j-1}$ and then we  use  the third line  of \eqref{dhom_uhom}  with $z^{j-1}$ instead of $z$ to obtain ${\rm v}^{j-1}_f$.
The definition of ${\rm v}^{j-1}_f$    and of  $W^j=W(b_{e,3}^j,y)$ in \eqref{defi_w} together with  assumptions {\bf A1.} and   {\bf A3.} on  ${\bf E}$ and   $\mathcal G$ ensure that $\| {\rm v}^{j-1}_f\|_{L^\infty(\Omega_s)} \leq C$ and $\| W^{j}\|_{L^\infty(\Omega_s;  L^2(Y_e))} \leq C_1\|b_{e,3}^{j} \|_{L^\infty(\Omega_s)} \leq C_2$.
Using the embedding  $H^1(\Omega) \subset L^4(\Omega)$  we obtain
\begin{equation*}
\begin{aligned}
\|(c^j - M)^+\|^2_{L^\infty(0, s; L^2(\Omega))} + \|\nabla (c^j-M)^+\|^2_{L^2(\Omega_s)} \leq  C  M^2  \int_0^s \Big[ |\Omega_M(t)| +   |\Omega_M(t)|^{\frac 12} \Big]dt
\end{aligned}
\end{equation*}
for some $s\in (0,T]$.

Then  applying Theorem II.6.1 in \cite{LSU}   with $q=4(1+\gamma)$, $r=5(1+\gamma)/2$  and iterating over time intervals yield the boundedness of $c^j$ in $L^\infty(0,T; L^\infty(\Omega))$. The same calculations ensure also  the boundedness of $c^{j-1}$.

Considering the equations  for   $\widetilde b^j_e$ and $\widetilde c^j$   and using  $\widetilde b^j_e$ and $\widetilde c^j$ as test functions in these equations, we obtain
\begin{equation}\label{uniq_b_1}
\begin{aligned}
&\|\widetilde b_{e}^j(s)  \|^2_{L^2(\Omega)} +  \|\nabla  \widetilde b_{e}^j \|^2_{L^2(\Omega_s)} \leq
C_1  \|c^{j-1}\|_{L^\infty(0, s; L^2(\Omega))} \|\widetilde b_{e}^j \|^2_{L^2(0, s; L^4(\Omega))}
\\ &+  \| \widetilde b_{e}^j \|^2_{L^2(\Omega_s)}
+
C_2 \|b_{e}^j\|_{L^\infty(\Omega_s)}\big[ \|\widetilde b_{e}^j\|^2_{L^2(\Omega_s)} + \|\widetilde c^j\|^2_{L^2(\Omega_s)} \big]
\\ &
+ C_3 \|b_{e}^{j-1}\|_{L^\infty(\Omega_s)} \Big[\|{\bf e} (\widetilde u_{e}^{j-1}) \|^2_{L^2(\Omega_s)}
 +  \| \widetilde W^{j}\|^2_{L^2(0,s; L^4(\Omega;  L^2(Y_e)))} +  \|\widetilde b_{e}^j \|^2_{ L^2(\Omega_s)} \\
 &+  \|\widetilde b_{e}^j \|^2_{L^2(0,s; L^4(\Omega))}\Big]
+ C_4\|{\bf e}(u_{e}^{j-1})\|_{L^\infty(0, s; L^2(\Omega))}  \|\widetilde b_{e}^j \|^2_{L^2(0, s; L^4(\Omega))}
\end{aligned}
\end{equation}
where $\widetilde W^{j}=W(b_{e,3}^j,y) - W(b_{e,3}^{j-1},y)$, and
\begin{equation}\label{uniq_c_1}
\begin{aligned}
& \|\widetilde c^j(s)  \|^2_{L^2(\Omega)} +  \|\nabla \widetilde c^j\|^2_{L^2(\Omega_s)}
 \\ & \leq
 C_1\left[1+ \|b_{e}^{j-1}\|_{L^\infty(\Omega_s)} + \|c^{j}\|_{L^\infty(\Omega_s)} \right]\big[\| \widetilde c^j \|^2_{L^2(\Omega_s)}+ \| \widetilde b_{e}^j \|^2_{L^2(\Omega_s)} \big]
 \\ &
   +
C_2 \|{\bf e}(u_{e}^{j-1})\|_{L^\infty(0,s; L^2(\Omega))}  \big[ \|\widetilde b_{e}^j \|^2_{L^2(0, s; L^4(\Omega))} + \|\widetilde c^j \|^2_{L^2(0, s; L^4(\Omega))} \big]
\\ &+ C_3\left( \|b_{e}^{j-1}\|_{L^\infty(\Omega_s)}
+  \|c^{j-1}\|_{L^\infty(\Omega_s)}  \right)   \left[\|{\bf e} (\widetilde u_{e}^{j-1}) \|^2_{L^2(\Omega_s)} \right. \\
& \left. \qquad \; \; +  \| \widetilde W^{j}\|^2_{L^2(0,s; L^4(\Omega;  L^2(Y_e)))} +  \|\widetilde c^j \|^2_{ L^2(\Omega_s)} + \|\widetilde c^j \|^2_{ L^2(0,s; L^4(\Omega))}\right]
\\ &
+ C_\mu \left[ \|{\rm v}^{j-1}_f \|^2_{L^\infty(\Omega_s)} \| \widetilde c^j \|^2_{L^2(\Omega_s)} + \|c^{j-1}\|^2_{L^\infty(\Omega_s)}
 \| \widetilde {\rm v}^{j-1}_f \|^2_{L^2(\Omega_s)}\right]
\end{aligned}
\end{equation}
for  $s \in (0,T]$. Here we used  assumptions {\bf A4.} on the non-linear functions  $g_b$,  $g_e$,  $g_f$, $P$, $F_b$, and  $F_c$.
From the definition of ${\rm v}_f^{j-1}$ and $W^{j-1}$, the Lipschitz continuity of $\mathcal G$ and assumptions on ${\bf E}$, it follows that
$$
 \| \widetilde {\rm v}^{j-1}_f \|_{L^2(\Omega_s)}\leq C \| \partial_t \widetilde u_{f}^{j-1}\|_{L^2(\Omega_s \times Y_f)}, \quad
 \| \widetilde W^{j}\|_{L^2(0,s; L^4(\Omega;  L^2(Y_e)))}  \leq C \|\widetilde b_e^{j} \|_{L^2(0,s; L^4(\Omega))}.
$$
Adding the inequalities \eqref{uniq_b_1} and \eqref{uniq_c_1},  considering  the compactness of  embedding $H^1(\Omega)\subset L^4(\Omega)$,  and using the H\"older and Gronwall inequalities  yield
 \begin{equation}\label{estim_bc_2}
\begin{aligned}
\|\widetilde b_{e}^j\|_{L^\infty(0,s; L^2(\Omega))} +  \|\nabla   \widetilde b_{e}^j \|_{L^2(\Omega_s)}
+  \|\widetilde c^j \|_{L^\infty(0, s; L^2(\Omega))} +  \|\nabla  \widetilde c^j \|_{L^2(\Omega_s)} \qquad  \qquad \\
 \leq C\left[ \|{\bf e} (\widetilde u_{e}^{j-1}) \|_{L^2(\Omega_s)} +  \| \partial_t \widetilde u_{f}^{j-1}\|_{L^2(\Omega_s \times Y_f)}  \right].
\end{aligned}
\end{equation}
To derive  the estimate for the  $L^\infty$-norm of $\widetilde b^j_e$ we
use $(\widetilde b_e^j)^{p-1}$ as a test function in  \eqref{macro_concentration}:
\begin{equation}\label{estim_contr_22}
\begin{aligned}
& \frac 1 p \|\widetilde b_{e}^j(s)  \|^p_{L^p(\Omega)}  + \frac {4(p-1)}{p^2} \|\nabla |\widetilde b_{e}^j|^{\frac p 2} \|^2_{L^2(\Omega_s)} \leq
C_1\Big[ \|c^j\|_{L^\infty(0,s; L^2(\Omega))} \\
&+ \|{\bf e}(u_{e}^{j-1})\|_{L^\infty(0,s; L^2(\Omega))}\|W^{j-1}\|_{L^\infty(\Omega_s; L^2(Y_e))}\Big]  \| |\widetilde b_{e}^j|^{\frac p2}\|^2_{L^2(0,s; L^{4}(\Omega))}
\\
& + C_2 \| \widetilde b_{e}^j \|^p_{L^p(\Omega_s)} +  C_3 \|b_{e}^{j-1}\|^2_{L^\infty(\Omega_s)}  \big \langle |{\bf e}(\widetilde u_{e}^{j-1})|, |\widetilde b_{e}^j|^{p-1}\big \rangle_{\Omega_s}
\\ & +
C_4 \|b_{e}^{j-1}\|_{L^\infty(\Omega_s)}  \big \langle |{\bf e}(u_{e}^{j-1})|\|\widetilde W^{j}\|_{L^2(Y_e)}, |\widetilde b_{e}^j|^{p-1}\big \rangle_{\Omega_s}
\\ &
 +  C_5 \|b_{e}^{j-1}\|_{L^\infty(\Omega_s)} \Big[\frac 1 p  \|\widetilde c^j\|^p_{L^\infty(0, s; L^2(\Omega))} +
  \frac { p-1}p\||\widetilde b_{e}^j|^{\frac p 2}\|^2_{L^2(0,s; L^4(\Omega))} \Big]
\end{aligned}
\end{equation}
for  $s\in (0,T]$.
Using the Gagliardo-Nirenberg inequality
$$
\|w\|_{L^4(\Omega)} \leq C \|\nabla w \|^{\alpha}_{L^2(\Omega)} \| w \|^{1-\alpha}_{L^1(\Omega)},
$$
with $\alpha = 9/10$,    and making calculations   similar to  those in \eqref{contract_b_e_1} in Appendix   we obtain  the following estimate
\begin{equation}\label{estim_conts_33}
\begin{aligned}
 & \langle |{\bf e}(\widetilde u_{e}^{j-1})| , |\widetilde b_{e}^j|^{p-1} \rangle_{\Omega_s}\leq
\delta \frac { p-1}{p^2} \|\nabla |\widetilde b_{e}^j|^{\frac p 2} \|^2_{L^2(\Omega_s)}  \\
& \qquad \qquad + C_\delta \frac{(p-1) p^\beta}{p} \| |\widetilde b_{e}^j|^{\frac p 2} \|_{L^\infty(0, s; L^1(\Omega))}^2
+
C\frac 1 p \|{\bf e}(\widetilde u_{e}^{j-1})\|^p_{L^{1+\frac 1\sigma}(0,s; L^{2}(\Omega))},
\end{aligned}
\end{equation}
where $\beta = \frac \alpha{1 - \alpha}$,  $0<\sigma< 1/9$, and  $\delta>0$ can be chosen  arbitrarily.  The definition of $\widetilde W^j$ implies
\begin{eqnarray}\label{estim_conts_34}
\quad  \; \;  \langle |{\bf e}(u_{e}^{j-1})|\, \|\widetilde W^{j}\|_{L^2(Y_e)}, |\widetilde b_{e}^j|^{p-1} \rangle_{\Omega_s}\!
 \leq\! C \| {\bf e}(u_{e}^{j-1}) \|_{L^\infty(0,s; L^2(\Omega))} \||\widetilde b_{e}^j|^{\frac p 2}\|^2_{L^2(0,s; L^4(\Omega))}.
\end{eqnarray}
Incorporating  the estimate    \eqref{estim_bc_2} for $\|\tilde c^j\|_{L^\infty(0,\tau; L^2(\Omega))}$  together with   \eqref{estim_conts_33}  and \eqref{estim_conts_34} into  \eqref{estim_contr_22},  using  the  Gagliardo-Nirenberg inequality  to estimate  $ \| \widetilde b_{e}^j \|^p_{L^p(\Omega_s)}$ by  $
\|| \widetilde b_{e}^j|^{\frac p 2} \|^2_{L^1(\Omega_s)}$,  and  iterating over $p=2^k$, with $k=2,3,\ldots$, as in the  Alikakos  Lemma  \cite[Lemma~3.2]{Alikakos}  we finally get
$$
\|\widetilde b_{e}^{j} \|_{L^\infty(0,s; L^\infty(\Omega))} \leq C\big[ \|{\bf e}(\widetilde u_{e}^{j-1})\|_{L^{1+\frac 1 \sigma}(0,s; L^2(\Omega))} + \|\partial_t \widetilde u_{f}^{j-1}\|_{L^2(\Omega_s\times Y_f)}  \big]
$$
for  $s \in (0,T]$ and any  $0<\sigma < 1/9$.
\end{proof}

The macroscopic equations for elastic deformation and pressure are coupled with the  two-scale problem for  fluid flow velocity. Thus the derivation of the estimates for  $u_{e}$ and $\partial_t u_{f}$ is non-standard and is shown in the following lemma.
\begin{lemma}\label{estim_u_ef_macro}
For two iterations
$$
(u_{e}^{j-1},   p_{e}^{j-1}, \partial_t u_{f}^{j-1}, \pi_f^{j-1}),\  (b_e^{j-1}, c^{j-1})\ \  \hbox{and} \ \  (u_{e}^j,   p_{e}^j, \partial_t u_{f}^j, \pi_f^j),\  (b_e^j, c^j)
$$
for  limit problem \eqref{macro_ue}--\eqref{macro_two-scale_uf},~\eqref{macro_concentration} we have the following estimates
\begin{equation}
\begin{aligned}
  \| \partial_t \widetilde u_{e}^j\|_{L^\infty(0,s; L^2(\Omega))} + \| {\bf e}(\widetilde u_{e}^{j})\|_{L^\infty(0,s; L^2(\Omega))}\\  +
    \| \widetilde p_{e}^j\|_{L^\infty(0,s; L^2(\Omega))}  + \| \nabla \widetilde p_{e}^j\|_{L^2(\Omega_s)}  & \leq C \| \widetilde b_{e}^j\|_{L^\infty(0,s; L^\infty(\Omega))} , \\
  \| \partial_t \widetilde  u_{f}^j\|_{L^\infty(0,s; L^2(\Omega\times Y_f))} + \| {\bf e}_y(\partial_t \widetilde u_{f}^j)\|_{L^2(\Omega_s\times Y_f)}
   & \leq C \|\widetilde  b_{e}^j\|_{L^\infty(0,s; L^\infty(\Omega))},
\end{aligned}
\end{equation}
for  $s\in(0,T]$, where  $\widetilde u_e^j = u_e^j - u_e^{j-1}$, $\widetilde p_e^j = p_e^j - p_e^{j-1}$, $\partial_t\widetilde u_f^j = \partial_t u_f^j - \partial_t  u_f^{j-1}$,    $\widetilde b_e^j = b_e^j - b_e^{j-1}$, and  the constant $C$ is independent of $s$ and solutions of the macroscopic problem.
\end{lemma}
\begin{proof} We begin with the two-scale model for fluid flow velocity.
Taking $\partial_t \widetilde u_{f}^j  -  \partial_t \widetilde u_{e}^j$ as a test function in the equation for the difference  $\partial_t \widetilde u_{f}^j$ we obtain
\begin{equation}\label{eq_uniq_u}
  \begin{aligned}
&   \big\langle {\bf E}^{\rm hom}(b_{e,3}^{j-1}) {\bf e}(\widetilde u_{e}^j(s)), {\bf e}(\widetilde u_{e}^j(s))\big\rangle_\Omega
 -  \big\langle \partial_t {\bf E}^{\rm hom}(b_{e,3}^{j-1}) {\bf e}(\widetilde u_{e}^j),{\bf e}(\widetilde u_{e}^j)\big\rangle_{\Omega_s}
\\ & +2 \big \langle ({\bf E}^{\rm hom}(b_{e,3}^{j}) - {\bf  E}^{\rm hom}(b_{e,3}^{j-1})) \, {\bf e}(u_{e}^j(s)),   \be(\widetilde u_{e}^j(s))\big\rangle_\Omega \\
& - 2 \big\langle \partial_t  ({\bf E}^{\rm hom}(b_{e,3}^{j}) - {\bf E}^{\rm hom}(b_{e,3}^{j-1}))  {\bf e}(u_{e}^j),
 {\bf e}(\widetilde u_{e}^j)\big \rangle_{\Omega_s} \\
 & - 2 \big\langle({\bf E}^{\rm hom}(b_{e,3}^{j}) - {\bf E}^{\rm hom}(b_{e,3}^{j-1}))  \partial_t {\bf e}(u_{e}^j),
 {\bf e}(\widetilde u_{e}^j)\big \rangle_{\Omega_s}
\\ &
 + \rho_e \|\partial_t  \widetilde u_{e}^j(s)\|^2_{L^2(\Omega)}
+ \rho_f \|\partial_t \widetilde u_{f}^j(s) \|^2_{L^2(\Omega\times Y_f)}  + 2 \mu \|{\bf e}_y(\partial_t \widetilde u_{f}^j)\|^2_{L^2(\Omega_s\times Y_f)}
\\
& +2  \langle \nabla \widetilde p_{e}^j, \partial_t \widetilde u_e^j + \int_{Y_f} \partial_t \widetilde u_{f}^j \, dy \rangle_{\Omega_s}  = 2 \langle \widetilde p_e^{1, j}, \partial_t \widetilde u_{f}^j \cdot n \rangle_{\Omega_s\times \Gamma}
+ \rho_e  \|\partial_t  \widetilde u_{e}^j(0)\|^2_{L^2(\Omega)}\\
&
+ \rho_f \|\partial_t \widetilde u_{f}^j(0) \|^2_{L^2(\Omega\times Y_f)}
+  \big\langle {\bf E}^{\rm hom}(b_{e,3}^{j-1})\, {\bf e}(\widetilde u_{e}^j(0)), {\bf e}(\widetilde u_{e}^j(0))\big\rangle_\Omega
\\ &+ 2 \big \langle ({\bf E}^{\rm hom}(b_{e,3}^{j}) - {\bf  E}^{\rm hom}(b_{e,3}^{j-1})) \, {\bf e}(u_{e}^j(0)),   \be(\widetilde u_{e}^j(0))\big\rangle_\Omega ,
\end{aligned}
\end{equation}
where  $\widetilde p_e^{1,j} = p^{1,j}_{e} - p^{1,j-1}_{e}$.   Equation \eqref{macro_ue} for $p_e^j$ and $p_e^{j-1}$ yields
\begin{equation} \label{eq_uniq_p}
\begin{aligned}
&\rho_p \|\widetilde p_{e}^j(s)\|^2_{L^2(\Omega)} + 2 \langle K_{p}^{\rm hom} \nabla \widetilde p_{e}^j, \nabla \widetilde  p_{e}^j  \rangle_{\Omega_s}
\\ &
=  2 \langle K_{u}\,  \partial_t \widetilde  u_e^j  + Q(x, \partial_t u_f^j)- Q(x, \partial_t u_f^{j-1}), \nabla \widetilde   p_{e}^j \rangle_{\Omega_s}   + \rho_p  \|\widetilde p_{e}^j(0)\|^2_{L^2(\Omega)} .
\end{aligned}
\end{equation}
Due to the  assumptions in {\bf A1.} on ${\bf E}$ we have
$$
\begin{aligned}
& \|{\bf E}^{\rm hom}(b_{e,3}^j) - {\bf E}^{\rm hom}(b_{e,3}^{j-1})\|_{L^\infty(\Omega_s)}+ \|\partial_t  ({\bf E}^{\rm hom}(b_{e,3}^j) - {\bf E}^{\rm hom}(b_{e,3}^{j-1})) \|_{L^\infty(\Omega_s)}
 \\ & \qquad\leq  C \|\widetilde b_{e}^j \|_{L^\infty(0,s; L^\infty(\Omega))}
\end{aligned}
$$
for  $s \in (0,T]$.  The  expression \eqref{form_pe-1} for  $p_{e}^{1,j}$ and $p_e^{1,j-1}$    and the estimates for the  $H^1$-norm  of the  solutions of the unit cell problems  \eqref{correct_p_Kp}, \eqref{correct_p_Ku}, and \eqref{two-scale_qf} yield
\begin{equation*}
\begin{aligned}
\| \widetilde p^{1,j}_{e} \|_{L^2(\Omega_s\times \Gamma)}  \leq  C \left(\|\nabla \widetilde p_{e}^j  \|_{L^2(\Omega_s)} +  \|\partial_t  \widetilde u_{e}^j \|_{L^2(\Omega_s)}
+ \|\partial_t  \widetilde u_{f}^j \|_{L^2(\Omega_s\times \Gamma)} \right).
\end{aligned}
\end{equation*}
From the compactness of  the embedding  $H^1(Y_f) \subset L^2(\Gamma)$  we obtain
\begin{equation*}
\begin{aligned}
 \|\partial_t  \widetilde u_{f}^j \|_{L^2(\Omega_s\times \Gamma)} \leq  C_\delta \|\partial_t  \widetilde u_{f}^j \|_{L^2(\Omega_s\times Y_f)} + \delta  \|\nabla_y \partial_t  \widetilde u_{f}^j \|_{L^2(\Omega_s\times Y_f)}
\end{aligned}
\end{equation*}
for any $\delta >0$.
Adding  \eqref{eq_uniq_u} and \eqref{eq_uniq_p}, and applying the  H\"older and Gronwall inequalities yield
\begin{equation*}
\begin{aligned}
&  \|\partial_t  \widetilde u_{f}^j \|_{L^\infty(0,s; L^2(\Omega\times Y_f))}  +  \|{\bf e}_y(\partial_t  \widetilde u_{f}^j)\|_{L^2(\Omega_s\times Y_f)}
+ \|\partial_t  \widetilde u_{e}^j\|_{L^\infty(0,s; L^2(\Omega))}
\\ & + \| {\bf e}(\widetilde u_{e}^j) \|_{L^\infty(0,s; L^2(\Omega))}
 +  \|\widetilde p_{e}^j\|_{L^\infty(0,s; L^2(\Omega))}  + \|\nabla \widetilde p_{e}^j\|_{ L^2(\Omega_s)}  \leq C  \|\widetilde b_{e}^j \|_{L^\infty(0, s; L^\infty(\Omega))}
\end{aligned}
\end{equation*}
for all $s\in (0,T]$.
\end{proof}

The estimates in Lemmas~\ref{estim_uniq_macro_bc}~and~\ref{estim_u_ef_macro}  together with a fixed point argument  imply  the existence of a unique solution of the  strongly coupled limit problem \eqref{macro_ue}--\eqref{macro_two-scale_uf},~\eqref{macro_concentration}.
\begin{lemma}\label{uniq_macro}
There exists a unique weak solution of the limit  problem \eqref{macro_ue}--\eqref{macro_two-scale_uf} and \eqref{macro_concentration}.
\end{lemma}
\begin{proof}
Considering the equations for the difference of two iterations for  \eqref{macro_ue}--\eqref{macro_two-scale_uf},~\eqref{macro_concentration},  and using estimates in Lemmas~\ref{estim_uniq_macro_bc} and \ref{estim_u_ef_macro}
 yield
\begin{equation}
\begin{aligned}
&\|\partial_t (u_{e}^j -u_{e}^{j-1}) \|_{L^\infty(0,s; L^2(\Omega))} + \|{\bf e} (u_{e}^j -u_{e}^{j-1}) \|_{L^\infty(0,s; L^2(\Omega))}  \\
 & +
\|\partial_t (u_{f}^j -u_{f}^{j-1}) \|_{L^\infty(0,s; L^2(\Omega\times Y_f))} + \|{\bf e}_y (u_{f}^j -u_{f}^{j-1}) \|_{L^2(\Omega_s \times Y_f)}
\\ &  \leq C_1 \|b_e^j- b_e^{j-1}\|_{L^\infty(0,s; L^\infty(\Omega))} \\
&   \leq  C \big[\|{\bf e} (u_{e}^{j-1} -u_{e}^{j-2}) \|_{L^{1+ \frac 1 \sigma}(0,s; L^2(\Omega))} + \|\partial_t (u_{f}^{j-1} -u_{f}^{j-2}) \|_{L^2(\Omega_s\times Y_f)} \big]
\end{aligned}
\end{equation}
for  $s\in (0,T)$ and  any $0<\sigma < 1/9$, where $C$ is independent of $s$ and  iterative solutions of the limit problem. Considering a time interval $(0,\tilde T)$, such that $C \tilde T^{\frac \sigma{1+\sigma}} < 1$  and $C \tilde T^{1/2} < 1$
and applying a fixed point argument we obtain the existence of a unique solution of the coupled system \eqref{macro_ue}--\eqref{macro_two-scale_uf},~\eqref{macro_concentration}  on the time interval $[0,\tilde T]$.
Iterating this step over time-intervals of length $\tilde T$ yields  the existence and uniqueness of a solution of the macroscopic problem \eqref{macro_ue}--\eqref{macro_two-scale_uf},~\eqref{macro_concentration} on an arbitrary time interval $[0,T]$.
\end{proof}

\section{Incompressible case. Quasi-stationary   poroelastic equations in $\Omega_e^\ve$}\label{incompresible}

Problem \eqref{eq_codif}--\eqref{exbou_co} was derived under a number of assumptions on plant tissue. In some cases
these assumptions should be changed and system \eqref{eq_codif}--\eqref{exbou_co} should be modified accordingly.

In this section we consider two possible modifications of problem \eqref{eq_codif}--\eqref{exbou_co}:\\
 (i) the incompressible case, when the intercellular space is completely saturated
with water  and we have the elliptic equation for $p_e^\ve$; \\
(ii) the quasi-stationary case for the displacement  $u_e^\ve$.
 In this case we can consider both compressible and incompressible fluid phase  in the elastic part $\Omega_e^\ve$.

 In the first case the equation for  $p_e^\ve$ in \eqref{equa_cla} is  replaced  with the following  elliptic equation:
\begin{equation}\label{mod_pe}
- {\rm div} ( K^\ve_p \nabla p^\ve_e -  \partial_t u_e^\ve)  = 0  \qquad   \text{ in } \Omega_{e,T}^\ve.
\end{equation}
In the second situation we consider in \eqref{equa_cla}  the quasi-stationary equations for  $u_e^\ve$
\begin{equation}\label{mod_ue}
 - {\rm div} ( {\bf E}^\ve(b^\ve_{e,3}) \be( u^\ve_e)) + \nabla p_e^\ve  = 0  \qquad  \text{ in } \Omega_{e,T}^\ve.
 \end{equation}
In the incompressible case, i.e.\  $p_e^\ve$ satisfies  equation  \eqref{mod_pe},  Definition~\ref{def_weak_micro} of a weak solution of  microscopic problem  \eqref{eq_codif}--\eqref{exbou_co} should be modified. Namely,  we assume that
\begin{equation}\label{def_constr}
\begin{aligned}
& p_e^\ve \in L^2(0,T; H^1(\Omega_e^\ve))\; \; \;  \text{ with } \; \;  \int_{\Omega_e^\ve} p_e^\ve(t,x) \, dx   =0 \;\; \;  \text{for } \; \; \; t \in (0,T)
\end{aligned}
\end{equation}
 and no initial conditions for $p_e^\ve$ are required.  Additionally we  assume that
 $$\int_{\partial \Omega} F_p(t,x) \,  dx  = 0 \quad  \text{ for } \; \; t \in (0,T). $$

The analysis of  the quasi-stationary problems considered in this section is very similar to the analysis of  \eqref{eq_codif}--\eqref{exbou_co} presented in the previous sections. The only part that should be slightly modified is the derivation of \textit{a priori} estimates.

For the incompressible case,  in the same  way as in the proof of Lemma~\ref{Lemma:apriori}, but now with  equation \eqref{mod_pe} for $p^\ve_e$,    we obtain
\begin{equation}\label{estim_uef_p_12}
\begin{aligned}
& \|\partial_t u_e^\ve(s)\|^2_{L^2(\Omega_e^\ve)}  +  \| \be(u_e^\ve(s) )\|^2_{L^2(\Omega_e^\ve)} + \|\nabla p_e^\ve\|^2_{L^2(\Omega_{e, s}^\ve)}
 + \ve^2 \| \be(\partial_t u_f^\ve)\|^2_{L^2(\Omega_{e,s}^\ve)}\\
& +  \|\partial_t u_f^\ve(s)\|^2_{L^2(\Omega_f^\ve)} \leq    \delta \big[\|u^\ve_e(s) \|^2_{L^2(\partial \Omega)} +  \|p^\ve_e \|^2_{L^2((0,s)\times\partial \Omega)} \big]
  + C_1
 \|\be(u_e^\ve)\|^2_{L^2(\Omega_{e,s}^\ve)}  \\
& + C_\delta\big[\| F_u \|^2_{L^\infty(0,s; L^2(\partial \Omega))}
  + \| \partial_t F_u \|^2_{L^2((0,s)\times\partial \Omega)}
  + \| F_p \|^2_{L^2((0,s)\times  \partial \Omega)} \big] +C_2
\end{aligned}
\end{equation}
for $s \in (0,T]$  and arbitrary $\delta>0$. Then, as in the proof of Lemma~\ref{Lemma:apriori},  applying the trace and Korn inequalities \cite{OShY}  and using extension properties of $u_e^\ve$ and assumptions {\bf A5.} on initial data $u^\ve_{e0}$, $u_{e0}^1$, and $u_{f0}^1$,  we obtain  estimates \eqref{bou_u},  \eqref{bou_init_u}, and \eqref{bou_uef_11}.
The trace and Poincare  inequalities together with the constraints in \eqref{def_constr} and properties of an extension of $p^\ve_e$ from $\Omega_e^\ve$ to $\Omega$, see Lemma~\ref{extension},  ensure that
\begin{equation}\label{bou_p}
\begin{aligned}
\|p^\ve_e \|^2_{L^2((0,s)\times\partial \Omega)}  \leq C \|\nabla p_e^\ve\|^2_{L^2(\Omega_{e, s}^\ve)}
\end{aligned}
\end{equation}
for $s\in (0,T]$. Then  applying the Gronwall inequality we obtain from  \eqref{estim_uef_p_12}  the estimates for $u_e^\ve$, $\partial_t u_e^\ve$, $p_e^\ve$ and $\partial_t u_f^\ve$ in \eqref{estim_upu_22}.

Differentiating the  equations in \eqref{equa_cla} and  \eqref{mod_pe}, with respect to time $t$ and  taking  $(\partial^2_t u_e^\ve,\,  \partial_t p_e^\ve,\, \partial^2_t u_f^\ve)$ as test functions in the weak formulation of the  resulting equations,  we obtain
\begin{equation}
\begin{aligned}
& \rho_e \| \partial^2_t u^\ve_e(s) \|^2_{L^2(\Omega_{e}^\ve)} +
\langle {\bf E}^\ve(b_{e,3}^\ve) \be( \partial_t u^\ve_e(s)), \be(\partial_t u^\ve_e(s)) \rangle_{\Omega_{e}^\ve}  \\
& +
2 \langle K_p^\ve \nabla \partial_t p^\ve_e, \nabla \partial_t p^\ve_e \rangle_{\Omega_{e,s}^\ve}
 +\rho_f  \| \partial^2_t u^\ve_f(s) \|^2_{L^2(\Omega_{f}^\ve)}\! + 2 \mu \,
\ve^2  \| \be(\partial^2_t u^\ve_f)\|^2_{L^2(\Omega_{f, s}^\ve)}\\
&
\! =\! 2 \langle \partial_t F_u,   \partial^2_t u^\ve_e \rangle_{(\partial \Omega)_s} +2 \langle \partial_t F_p,   \partial_t p^\ve_e \rangle_{(\partial \Omega)_s}   +\rho_e \| \partial^2_t u^\ve_e(0) \|^2_{L^2(\Omega_{e}^\ve)}
\\
&+\rho_f \| \partial^2_t u^\ve_f(0) \|^2_{L^2(\Omega_{f}^\ve)}+2 \big\langle \partial_t {\bf E}^\ve(b^\ve_{e,3}) \be( u^\ve_e(s)), \be(\partial_t u^\ve_e(s)) \big\rangle_{\Omega_{e}^\ve}\\
& +  \big\langle {\bf E}^\ve(b_{e,3}^\ve) \be(\partial_t u^\ve_e(0)) - 2\partial_t {\bf E}^\ve(b^\ve_{e,3}(0)) \be( u^\ve_e(0)), \be(\partial_t u^\ve_e(0)) \big\rangle_{\Omega_{e}^\ve} \\
& -\big\langle 2 \partial^2_t {\bf E}^\ve(b_{e,3}^\ve) \be( u^\ve_e)+ \partial_t {\bf E}^\ve(b_{e,3}^\ve) \be(\partial_t u^\ve_e), \be(\partial_t u^\ve_e)\big \rangle_{\Omega_{e, s}^\ve}
\end{aligned}
\end{equation}
for  $s\in (0,T]$.  As before, applying  the Korn inequality  and the Poincare inequality together with  the constraint in \eqref{def_constr},   we obtain  the estimates for $\partial^2_t u_e^\ve$,  $\partial_t p^\ve_e$, and $\partial^2_t u_f^\ve$ stated in \eqref{estim_u_p_u}.
The equations for $\partial_t u_f^\ve$ and  $u_e^\ve$ and calculations similar to those in the proof of Lemma~\ref{Lemma:apriori} ensure the estimate for $p_f^\ve$.

To derive the {\it a priori} estimates  in  the second case, when $u_e^\ve$ satisfies the quasi-stationary equations \eqref{mod_ue},  we have to check that the Korn  inequality holds for  $u_e^\ve$.
\begin{lemma}\label{Korn_quasi} For $u_e^\ve(s) \in H^1(\Omega_e^\ve)$, with $s \in (0,T]$, we have the following estimate
\begin{equation}
\begin{aligned}
 \|u^\ve_e(s)\|_{H^1(\Omega_e^\ve)} \leq &\, C \big[ \|\be(u_e^\ve(s)) \|_{L^2(\Omega_e^\ve)}  + \ve^{\frac 12 }  \|\Pi_\tau \partial_t  u_f^\ve\|_{L^2(\Gamma^\ve_T)}  + \| u^\ve_{e}(0) \|_{H^1(\Omega)}\big], \\
 \|\partial_t u^\ve_e(s)\|_{H^1(\Omega_e^\ve)} \leq &\, C \big[ \|\partial_t \be(u_e^\ve(s)) \|_{L^2(\Omega_e^\ve)}  + \ve^{\frac 12 }  \|\Pi_\tau \partial_t  u_f^\ve(s)\|_{L^2(\Gamma^\ve)} \big].
\end{aligned}
\end{equation}
\end{lemma}
\begin{proof}
Consider first $Y_e$ and $\mathcal V=\{v\in H^1(Y_e)^3 : \Pi_\tau v =0 \text{ on } \Gamma\}$. Then since $\mathcal V\cap \mathcal R(Y_e)=\{0\}$, where $\mathcal R(Y_e)$ is the space of all rigid displacements,   we have
\begin{equation}
\|v\|^2_{H^1(Y_e)} \leq C\big[ \|\be(v)\|^2_{L^2(Y_e)} + \|\Pi_{\tau} v \|^2_{L^2(\Gamma)}\big].
\end{equation}
Considering scaling $x = \ve y$ and summing over $\xi\in \Xi^\ve$ we obtain
\begin{equation}\label{estim_korn_v}
\|v\|^2_{L^2(\hat \Omega^\ve_e)}  +\ve^2  \|\nabla v\|^2_{L^2(\hat \Omega_e^\ve)} \leq C\big[ \ve^2 \|\be(v)\|^2_{L^2(\hat \Omega^\ve_e)} +
\ve  \|\Pi_{\tau} v \|^2_{L^2(\Gamma^\ve)}\big],
\end{equation}
where $\hat \Omega^\ve_e = {\rm Int} \big(\bigcup_{\xi \in \Xi^\ve} \ve(\overline Y_e + \xi) \big)$.
Using the fact that $\Pi_\tau \partial_t u_e^\ve =  \Pi_\tau \partial_t u_f^\ve $ on $\Gamma^\ve$ and   estimating $u_e^\ve$ by $\partial_t u_e^\ve$ and the initial value $u^\ve_{e}(0)$ we obtain
$$
\|\Pi_\tau u_e^\ve(s)\|_{L^2(\Gamma^\ve)} \leq C \big[ \|\Pi_\tau \partial_t u_f^\ve \|_{L^2(\Gamma^\ve_T)}+ \|u^\ve_{e}(0)\|_{L^2(\Gamma^\ve)}\big].
$$
Hence applying \eqref{estim_korn_v} to $u_e^\ve$  and using the fact that $\ve\|u^\ve_{e}(0)\|^2_{L^2(\Gamma^\ve)}\leq C \|u^\ve_{e}(0)\|^2_{H^1(\Omega)}$ we have
\begin{equation*}
\|u_e^\ve(s)\|^2_{L^2(\hat \Omega^\ve_e)}  \leq C\big[ \ve^2 \|\be(u_e^\ve)(s)\|^2_{L^2(\hat \Omega^\ve_e)} +
\ve  \|\Pi_{\tau} \partial_t u_f^\ve \|^2_{L^2(\Gamma^\ve_T)} +  \| u^\ve_{e}(0) \|^2_{H^1(\Omega)}\big].
\end{equation*}
Then considering the  extension of $u_e^\ve$ from $\Omega_e^\ve$ to $\Omega$, see e.g.\ \cite{OShY}, and applying the Korn inequality in $\Omega$ yield  the estimate stated in the Lemma.
\end{proof}

Then,  in the same way as in the proof of Lemma~\ref{Lemma:apriori},  applying the Korn inequalities proved in
Lemma~\ref{Korn_quasi},  using extension properties of $u_e^\ve$ and the regularity of the initial data  $u_{f0}^1 \in H^2(\Omega)^3$  we obtain the following {\it a priori} estimates  for solutions of the quasi-stationary  problem
\begin{equation}\label{estim_u_p_u_quasi}
\begin{aligned}
&\| u_e^\ve\|_{L^\infty(0,T; H^1(\Omega_e^\ve))} + \|\partial_t u_e^\ve\|_{L^\infty(0,T; H^1(\Omega_e^\ve))}  \leq C , \\
&\| p_e^\ve\|_{L^2(0,T; H^1(\Omega_e^\ve))}+   \| \partial_t p_e^\ve \|_{L^2(0,T; H^1(\Omega_{e}^\ve))} \leq C, \\
& \| \partial_t u_f^\ve\|_{L^\infty(0,T; L^2(\Omega_f^\ve))} + \| \partial^2_t u_f^\ve\|_{L^\infty(0,T; L^2(\Omega_{f}^\ve))}
\\ &\hskip 3cm + \ve \|\nabla \partial_t u_f^\ve\|_{H^1(0,T; L^2(\Omega_{f}^\ve))}   + \| p_f^\ve \|_{L^2(\Omega_{f,T}^\ve)} \leq C,
\end{aligned}
\end{equation}
where the constant $C$ is independent of $\ve$. Notice that in the incompressible and quasi-stationary  case, i.e.\ in the case of equations  \eqref{mod_pe} and \eqref{mod_ue} for $p_e^\ve$ and $u_e^\ve$, respectively,  problem  \eqref{equa_cla},   \eqref{exbou_co}, \eqref{mod_pe}, and \eqref{mod_ue}  is well-posed without the initial conditions for $u_e^\ve$ and $p_e^\ve$. In this case  $u_e^\ve(0, \cdot)$ and $\partial_t u_e^\ve(0, \cdot)$ are determined from the corresponding elliptic equations and the initial values for the fluid flow $u_{f0}^1$.

In contrast with the limit equations given by \eqref{macro_ue}, in the quasi-stationary and incompressible  case the macroscopic equations
for  effective displacement and pressure do not contain time derivatives and take the form
\begin{equation}\label{macro_ue_quasi}
\begin{aligned}
& - {\rm div} ( {\bf E}^{\rm{hom}}(b_{e,3}) \be(u_e)) + \nabla p_e + \vartheta_f  \rho_f \dashint_{Y_f} \partial^2_t u_f \, dy = 0 && \text{ in } \Omega_T, \\
& - {\rm div} \big( K_{p}^{\rm hom} \nabla p_e - K_ {u}\,  \partial_t u_e - Q(x, \partial_t u_f)\big) =0 && \text{ in } \Omega_T,\\
& \; \;  {\bf E}^{\rm hom}(b_{e,3}) \be(u_e) \, {n} =F_u  && \text{ on } (\partial \Omega)_T, \\
&\; \;  (K_{p}^{\rm hom} \nabla p_e - K_{u}\,  \partial_t u_e)\cdot {n} = F_p + Q(x, \partial_t u_f) \cdot n
&& \text{ on }   (\partial \Omega)_T,
\end{aligned}
\end{equation}
together with the two-scale equations \eqref{macro_two-scale_uf} for  $u_f$ and $\pi_f$.

\section{Appendix} Here we provide  proofs of the estimates for  $\| b^\ve_e \|_{L^\infty(0,T; L^\infty(\Omega_e^\ve))}$, $\| c^\ve \|_{L^\infty(0,T; L^4(\Omega^\ve))}$ and for the difference $\| \widetilde b^{\ve,j}_e \|_{L^\infty(0,T; L^\infty(\Omega_e^\ve))}$ of two iterations for system  \eqref{eq_codif}--\eqref{exbou_co}.
\begin{lemma}\label{Linf_L4}
Under assumptions {\bf A1.}--{\bf A5.}  solutions of the microscopic problem \eqref{eq_codif}--\eqref{exbou_co} satisfy the following estimates
\begin{equation}\label{estim_Linfty_L4}
\begin{aligned}
&\| b^\ve_e \|_{L^\infty(0,T; L^\infty(\Omega_e^\ve))} \leq C, \\
&\| c^\ve_e \|_{L^\infty(0,T; L^4(\Omega_e^\ve))} + \| c^\ve_f \|_{L^\infty(0,T; L^4(\Omega_f^\ve))}   \leq C,
\end{aligned}
\end{equation}
where the constant $C$ is independent of $\ve$.
\end{lemma}
\begin{proof}
  To show that $|b_e^\ve|^p$, for $p\geq 2$,  is an admissible test function for the equation \eqref{cd_one},  we  set  $b_{e,N}^\ve(t,x)= \min\{ b_e^\ve(t,x), N\}$ for  $(t,x) \in \Omega_{e,T}^\ve$, where  $N>\|b_{e0}\|_{L^\infty(\Omega)}$, and  derive  estimates for $|b_{e,N}^\ve|^p$ independent of $N$.  Then letting $N \to \infty$ we obtain the desired  estimates for  $b_e^\ve$. Taking $(b_{e,N}^\ve)^{p-1}$  as a test function in \eqref{cd_one} and applying simple calculations we obtain
\begin{equation}\label{estim_Linfty_be_1}
\begin{aligned}
& \|b^\ve_{e,N}(s)\|^p_{L^p(\Omega_e^\ve)} + \|\nabla |b^\ve_{e,N}|^{\frac p2} \|^2_{L^2(\Omega_{e,s}^\ve)}\leq  C_1 \Big[\| \be(u_e^\ve)\|_{L^\infty(0, s; L^2(\Omega_e^\ve))} \\
&\qquad \qquad  + \|c^\ve_e\|_{L^\infty(0, s; L^2(\Omega_e^\ve))} +1\Big]
\int_0^s \|b_{e}^\ve\|_{L^{2p}(\Omega_e^\ve)} \|b_{e,N}^\ve\|^{p-1}_{L^{2p}(\Omega_e^\ve)} dt
\\
& + C_2 \|b_{e0}\|^p_{L^p(\Omega_e^\ve)}+ C_3 \|b^\ve_e\|_{L^\infty(0,s; L^2(\Omega_e^\ve)) }\||b^\ve_{e,N}|^{\frac p2}\|^{2\frac{p-1} p}_{L^2(0,s; L^4(\Omega_{e}^\ve))} \\
&+ C_4 \Big[\||b_{e,N}^\ve|^{\frac p 2}\|^2_{L^2(0,s; L^4(\Omega_e^\ve))} + \| c^\ve_e\|^p_{L^p(0, s; L^2(\Omega_e^\ve))}+  \| \be(u_e^\ve)\|^p_{L^p(0, s; L^2(\Omega_e^\ve))} \Big]\\
& \qquad \qquad  + \ve \langle |P ( b_e^\ve)|, |b_{e,N}^\ve|^{p-1} \rangle_{\Gamma_s^\ve}  + \langle |F_b ( b_e^\ve)|, |b_{e,N}^\ve|^{p-1} \rangle_{(\partial \Omega)_s}
\end{aligned}
\end{equation}
for $s \in (0,T]$. Here we used the fact that the definition of $b_{e,N}^\ve$ implies
$$
\langle \nabla b_e^\ve, \nabla (b_{e,N}^\ve)^{p-1} \rangle_{\Omega_{e,s}^\ve}= \langle \nabla b_{e,N}^\ve, \nabla (b_{e,N}^\ve)^{p-1} \rangle_{\Omega_{e,s}^\ve}
$$
and that due to the inequality  $b_e^\ve \geq 0$ in $\Omega_{e,T}^\ve$ we have
\begin{equation*}
\begin{aligned}
& \langle \partial_t b_e^\ve, |b_{e,N}^\ve|^{p-1} \rangle_{\Omega_{e,s}^\ve}
 \geq\frac 1p  \|b^\ve_{e,N}(s)\|^p_{L^p(\Omega_e^\ve)}  - \frac 1 p \|b_{e0}\|^p_{L^p(\Omega_e^\ve)}
 -  \|b_{e0}\|^p_{L^p(\Omega_e^\ve)}\\
 & + \langle  b_{e}^\ve(s), |b_{e,N}^\ve(s)|^{p-1} \rangle_{\Omega_{e}^\ve \setminus \Omega_{e}^{\ve,N}\!(s)}
 \geq \frac 1p  \|b^\ve_{e,N}(s)\|^p_{L^p(\Omega_e^\ve)} - (1 + 1/p) \|b_{e0}\|^p_{L^p(\Omega_e^\ve)}.
\end{aligned}
\end{equation*}
Here $\Omega_{e}^{\ve,N}(t)=\{ x \in \Omega_{e,s}^\ve\, : \, b_e^\ve(t,x) \leq N \}$ for $t \in (0,T)$.
Applying  the Gagliardo-Nirenberg inequality   we can estimate
\begin{equation}\label{GN_p_Linfty}
\begin{aligned}
\||b^\ve_{e,N}|^p\|_{L^2(\Omega_{e}^\ve)}  =  \||b^\ve_{e,N}|^{\frac p2}\|^2_{L^4(\Omega_{e}^\ve)}  \leq C\|\nabla |b^\ve_{e,N}|^{\frac p2}\|^{2\alpha}_{L^2(\Omega^\ve_{e})}  \||b^\ve_{e,N}|^{\frac p2}\|^{1-\alpha}_{L^1(\Omega^\ve_{e})}
\end{aligned}
\end{equation}
with $\alpha = 9/10$. Using the embedding $L^2(0, s; H^1(\Omega_e^\ve)) \subset L^2(0, s; L^6(\Omega_e^\ve))$, in space-dimension two and  three, and applying the Gagliardo-Nirenberg inequality to \\
$\||b_{e,N}^\ve|^{\frac p 2}\|_{L^4(\Omega_e^\ve)}$ yield
\begin{equation*}
\begin{aligned}
&\int_0^s \|b_{e}^\ve\|_{L^{2p}(\Omega_e^\ve)} \|b_{e,N}^\ve\|^{p-1}_{L^{2p}(\Omega_e^\ve)} dt
\\ & \leq
\int_0^s \big(\|b_e^\ve\|_{L^{\frac{2p}3}(\Omega_e^\ve)} +  \|\nabla |b_e^\ve|^{\frac p3}\|^{\frac 3p}_{L^{2}(\Omega_e^\ve)} \big)
\|b_{e,N}^\ve\|^{\frac{p-1}4}_{L^{p}(\Omega_e^\ve)} \|\nabla|b_{e,N}^\ve|^{\frac p 2}\|^{\frac{3 (p-1)}{2 p}}_{L^{2}(\Omega_e^\ve)} dt.
\end{aligned}
\end{equation*}
 Then using  the H\"older inequality on the right-hand side of the last estimate, we obtain
\begin{equation}\label{etim_compl_1}
\begin{aligned}
&\int_0^s \|b_{e}^\ve\|_{L^{2p}(\Omega_e^\ve)} \|b_{e,N}^\ve\|^{p-1}_{L^{2p}(\Omega_e^\ve)} dt  \\
&\leq C\Big[ \int_0^s \big(\|b_e^\ve\|^{\frac{2p}3}_{L^{\frac{2p}3}(\Omega_e^\ve)} +  \|\nabla |b_e^\ve|^{\frac p3}\|^2_{L^{2}(\Omega_e^\ve)} \big)
dt\Big]^{\frac 3{2p}}
\\ & \qquad \qquad  \times \sup_{(0,s)} \|b_{e,N}^\ve\|^{ \frac{p-1} 4}_{L^{p}(\Omega_e^\ve)}  \Big[\int_0^s   \|\nabla|b_{e,N}^\ve|^{\frac p 2}\|^{\frac 32\frac{(2p-2)} {2p-3}}_{L^{2}(\Omega_e^\ve)} dt\Big]^{\frac{2p-3}{2p}}.
\end{aligned}
\end{equation}
For $p \geq 3$ we can estimate
\begin{equation}\label{spec_nabla}
\begin{aligned}
\|\nabla |b_e^\ve|^{\frac p3}\|^2_{L^{2}(\Omega_{e,s}^\ve)}
 \leq \|\nabla b_e^\ve \|^2_{L^{2}(\Omega_{e,s}^{\ve,1 })} +  \|\nabla |b_e^\ve|^{\frac {p-1} 2}\|^2_{L^{2}(\Omega_{e,s}^\ve\setminus \Omega_{e,s}^{\ve, 1})}
\\  \leq   \|\nabla b_e^\ve \|^2_{L^{2}(\Omega_{e,s}^{\ve})} +  \|\nabla |b_e^\ve|^{\frac {p-1} 2}\|^2_{L^{2}(\Omega_{e,s}^{\ve})},
\end{aligned}
\end{equation}
where $\Omega_{e,s}^{\ve,1} =\{ (t,x) \in \Omega_{e,s}^\ve:  b_e^\ve(t,x) \leq 1 \}. $ Also notice that for  $p\geq 3$  we have  $\frac 3 4\frac{(2p-2)} {2p-3}\leq 1$ and $\frac{2p}3 \leq p-1$.
Thus  applying the Young   inequality  in \eqref{etim_compl_1}   yields
\begin{equation}\label{estim_int_cut_b}
\begin{aligned}
\int_0^s\! \|b_{e}^\ve\|_{L^{2p}(\Omega_e^\ve)} \|b_{e,N}^\ve\|^{p-1}_{L^{2p}(\Omega_e^\ve)} dt
\leq \delta_1\sup_{(0,s)} \|b_{e,N}^\ve\|^{p}_{L^{p}(\Omega_e^\ve)}  + \delta_2   \|\nabla|b_{e,N}^\ve|^{\frac p 2}\|^2 _{L^{2}(\Omega_{e,s}^\ve)}
\\  +  C_\delta \Big(1+ \|b_e^\ve\|^{p-1}_{L^{p-1}(\Omega_{e,s}^{\ve})} + \|\nabla b_e^\ve \|^2_{L^{2}(\Omega_{e,s}^{\ve})} +  \|\nabla |b_e^\ve|^{\frac {p-1} 2}\|^2_{L^{2}(\Omega_{e,s}^{\ve})} \Big)^{\frac 32}
\end{aligned}
\end{equation}
for any $\delta_1>0$ and $\delta_2>0$.
 Using the trace inequality,  we  estimate  the integral over $\Gamma^\ve$ as
\begin{equation}\label{estim-bound_int_b}
\begin{aligned}
&\ve \langle |P ( b_e^\ve)|, |b_{e,N}^\ve|^{p-1} \rangle_{\Gamma_s^\ve} \leq   C_1 \ve  \langle 1 + |b_e^\ve|,  |b_{e,N}^\ve|^{p-1}\rangle_{\Gamma_s^\ve}\\
& \leq
\!C_2(\ve)\!\int\limits_0^s\!\Big[1+  \||b_e^\ve|^{\frac p3}\|^{\frac 14}_{L^2(\Omega_{e}^\ve)}  \|\nabla  |b_e^\ve|^{\frac p 3}\|^{\frac 34}_{L^2(\Omega_{e}^\ve)}
+  \||b_e^\ve|^{\frac p3}\|^{\frac 16}_{L^2(\Omega_{e}^\ve)}  \|\nabla  |b_e^\ve|^{\frac p 3}\|^{\frac 56}_{L^2(\Omega_{e}^\ve)} \Big]^{\frac 3 p} \\
&\hspace{4.5 cm } \times \Big[ \||b_{e,N}^\ve|^{\frac p 2}\|_{L^2(\Omega_{e}^\ve)}\|\nabla  |b_{e,N}^\ve|^{\frac p 2}\|_{L^2(\Omega_{e}^\ve)}  \Big]^{\frac {p-1} p} dt\\
&\leq C_3(\ve)\left[1+ \||b_e^\ve|^{\frac p3}\|^{\frac 1 {2p}}_{L^\infty(0,s;L^2(\Omega_{e}^\ve))} \|\nabla  |b_e^\ve|^{\frac p 3}\|^{\frac 5{2(p+1)}}_{L^2(\Omega_{e,s}^\ve)}\right]
\\ & \hspace{4.5 cm}   \times \sup\limits_{(0,s)}\||b_{e,N}^\ve|^{\frac p 2}\|^{\frac{p-1} p}_{L^2(\Omega_{e}^\ve)}  \|\nabla  |b_{e,N}^\ve|^{\frac p 2}\|^{\frac {p-1}{p}}_{L^2(\Omega_{e,s}^\ve)}.
\end{aligned}
\end{equation}
Applying  the Young inequality on the right-hand side of \eqref{estim-bound_int_b} and using  \eqref{spec_nabla},  together with the uniform estimate of $\|\nabla b_e^\ve\|_{L^2(\Omega_{e,s}^\ve)}$, obtained in Lemma~\ref{Lemma:apriori}, yield
\begin{equation*}
\begin{aligned}
\ve \langle |P ( b_e^\ve)|, |b_{e,N}^\ve|^{p-1} \rangle_{\Gamma_s^\ve} \leq C(\ve)\Big[1+ \||b_e^\ve|^{\frac p3}\|^{\frac 12 }_{L^\infty(0,s;L^2(\Omega_{e}^\ve))}\big(1+ \|\nabla  |b_e^\ve|^{\frac {p-1} 2}\|^{\frac {5p}{2(p+1)}}_{L^2(\Omega_{e,s}^\ve)}\big)\Big] \\
+ \delta_1 \sup_{(0,s)}\||b_{e,N}^\ve|^{\frac p 2}\|^2_{ L^2(\Omega_{e}^\ve)}
+ \delta_2  \|\nabla  |b_{e,N}^\ve|^{\frac p 2}\|^2_{L^2(\Omega_{e,s}^\ve)}.
\end{aligned}
\end{equation*}
The same calculations together with \eqref{spec_nabla} ensure that
\begin{equation*}
\begin{aligned}
\langle |F_b ( b_e^\ve)|, |b_{e,N}^\ve|^{p-1} \rangle_{(\partial \Omega)_s} \leq
C(\ve)\Big[1+ \||b_e^\ve|^{\frac {p-1}2}\|^3_{L^\infty(0,s;L^2(\Omega_{e}^\ve))} +  \|\nabla  |b_e^\ve|^{\frac {p-1} 2}\|^{3}_{L^2(\Omega_{e,s}^\ve)}\Big]
\\
 + \delta_1\sup_{(0,s)} \| b_{e,N}^\ve\|^p_{L^p(\Omega_{e}^\ve)}
+ \delta_2  \|\nabla  |b_{e,N}^\ve|^{\frac p 2}\|^2_{L^2(\Omega_{e,s}^\ve)}.
\end{aligned}
\end{equation*}
Considering $p=3$  and using the standard {\it a priori} estimates \eqref{bou_b_c} for $b_e^\ve$ yield
\begin{equation}\label{estim_p3_1}
\begin{aligned}
\int_0^s \! \|b_{e}^\ve\|_{L^{6}(\Omega_e^\ve)} \|b_{e,N}^\ve\|^{2}_{L^{6}(\Omega_e^\ve)} dt \leq C\Big[ \int_0^s\! \Big(\|b_e^\ve\|^{2}_{L^{2}(\Omega_e^\ve)} +  \|\nabla  b_e^\ve \|^2_{L^{2}(\Omega_e^\ve)} \Big)
dt\Big]^{\frac 1{2}}
\\  \times \sup_{(0,s)} \|b_{e,N}^\ve\|^{ \frac{1} 2}_{L^{3}(\Omega_e^\ve)}  \Big[\int_0^s   \|\nabla|b_{e,N}^\ve|^{\frac 3 2}\|^{2}_{L^{2}(\Omega_e^\ve)} dt\Big]^{\frac{1}{2}}\\
 \leq  C_\delta + \delta_1  \sup_{(0,s)} \|b_{e,N}^\ve(s)\|^{3}_{L^{3}(\Omega_e^\ve)}  + \delta_2   \|\nabla|b_{e,N}^\ve|^{\frac 3 2}\|^2 _{L^{2}(\Omega_{e,s}^\ve)}.
\end{aligned}
\end{equation}
For the boundary integrals, for $p=3$, we have
\begin{equation}\label{estim_p3_2}
\begin{aligned}
& \ve \langle |P ( b_e^\ve)|, |b_{e,N}^\ve|^{2} \rangle_{\Gamma_s^\ve}
 + \langle |F_b ( b_e^\ve)|, |b_{e,N}^\ve|^{2} \rangle_{(\partial \Omega)_s}  \leq C_1(\ve)\Big[1\\
 &+ \| b_e^\ve\|^{\frac 1 {6}}_{L^\infty(0,s;L^2(\Omega_{e}^\ve))} \|\nabla b_e^\ve\|^{\frac 5{8}}_{L^2(\Omega_{e,s}^\ve)}\Big]
 \sup\limits_{(0,s)}\||b_{e,N}^\ve|^{\frac 3 2}\|^{\frac{2} 3}_{L^2(\Omega_{e}^\ve)}  \|\nabla  |b_{e,N}^\ve|^{\frac 3 2}\|^{\frac {2}{3}}_{L^2(\Omega_{e,s}^\ve)}\\
& \leq C_2(\ve) \big[1+ \| b_e^\ve\|^{\frac 1 {2}}_{L^\infty(0,s;L^2(\Omega_{e}^\ve))} \|\nabla b_e^\ve\|^{\frac {15}{8}}_{L^2(\Omega_{e,s}^\ve)}\big]
\\ & +
\delta_1  \sup\limits_{(0,s)}\|b_{e,N}^\ve(s)\|^{3}_{L^3(\Omega_{e}^\ve)}  + \delta_2  \|\nabla  |b_{e,N}^\ve|^{\frac 3 2}\|^{2}_{L^2(\Omega_{e,s}^\ve)} .
\end{aligned}
\end{equation}
Considering  \eqref{estim_Linfty_be_1} for $p=3$ and using   the estimates  \eqref{GN_p_Linfty}, \eqref{estim_p3_1} and \eqref{estim_p3_2}  together with the standard {\it a priori} estimates for $b_e^\ve$, $c_e^\ve$ and $u_e^\ve$, shown in Lemma~\ref{Lemma:apriori},  we obtain
\begin{equation*}
\begin{aligned}
& \|b^\ve_{e,N}(s)\|^3_{L^3(\Omega_e^\ve)} + \|\nabla |b^\ve_{e,N}|^{\frac 32} \|^2_{L^2(\Omega_{e,s}^\ve)}
\\ &
 \leq  C(\ve) + \delta_1  \sup\limits_{(0,s)}\|b_{e,N}^\ve(s)\|^{3}_{L^3(\Omega_{e}^\ve)}
 + \delta_2  \|\nabla  |b_{e,N}^\ve|^{\frac 3 2}\|^{2}_{L^2(\Omega_{e,s}^\ve)}
\end{aligned}
\end{equation*}
with  $s\in (0,T]$,  a constant $C(\ve)$ independent of $N$, and arbitrary $0<\delta_1\leq \frac 12 $ and $0<\delta_2 \leq \frac 12$.  Considering the supremum over $(0,s)$  and taking the  limit $N \to \infty$ yield that $b_e^\ve \in L^\infty(0,T; L^3(\Omega_e^\ve))$ and $\nabla |b^\ve_{e}|^{\frac 32} \in L^2(\Omega_{e,T}^\ve)$.
Taking iteratively   $p=4,5,\ldots$ and choosing $\delta_1>0$ and $\delta_2>0$ sufficiently small for each fixed $p$ and  for fixed $\ve$ we obtain estimates for $\|b_{e,N}^\ve\|_{L^\infty(0,T;L^p(\Omega_e^\ve))}$ and $\| \nabla  |b^\ve_{e,N}|^{\frac p2} \|^2_{L^2(\Omega_{e,T}^\ve)}$
independent of $N$. Letting  $N\to \infty$ yields that $ |b^\ve_{e}|^{\frac p 2} \in L^2(0,T;  H^1(\Omega_{e}^\ve))$  and $b^\ve_{e} \in L^\infty(0,T; L^p(\Omega_e^\ve))$  for every fixed $p \geq 2$.

 Now we  consider  $(b_e^\ve)^{p-1}$  as a test function in \eqref{cd_one} and obtain
\begin{equation}\label{estim_Lp_b}
\begin{aligned}
\frac 1 p \|b^\ve_{e}(s)\|^p_{L^p(\Omega_e^\ve)} +\frac {4(p-1)} {p^2} \|\nabla |b^\ve_{e}|^{\frac p2} \|^2_{L^2(\Omega_{e,s}^\ve)}\leq
\frac 1 p \|b_{e0}\|^p_{L^p(\Omega_e^\ve)}
+ \|b^\ve_{e}\|^p_{L^p(\Omega_{e,s}^\ve)} \\
+ C_1 \big(\| c^\ve_e\|_{L^\infty(0, s; L^2(\Omega_e^\ve))}+  \| \be(u_e^\ve)\|_{L^\infty(0, s; L^2(\Omega_e^\ve))} +1\big) \||b_{e}^\ve|^{\frac p 2}\|^2_{L^2(0,s; L^{4}(\Omega_e^\ve))}
\\
+ \frac {C_2} p \big(\| c^\ve_e\|^p_{L^p(0, s; L^2(\Omega_e^\ve))}+  \| \be(u_e^\ve)\|^p_{L^p(0, s; L^2(\Omega_e^\ve))} \big)
 \\ + \ve \langle |P ( b_e^\ve)|, |b_e^\ve|^{p-1}\rangle_{\Gamma_s^\ve} +
\langle F_b(b_e^\ve), |b_e^\ve|^{p-1}\rangle_{(\partial\Omega)_s}
\end{aligned}
\end{equation}
for  $s \in (0,T]$.
The integral over $\Gamma^\ve$ is estimated as
\begin{equation*}
\begin{aligned}
\ve \langle |P ( b_e^\ve)|, |b_e^\ve|^{p-1}\rangle_{\Gamma_s^\ve} & \leq   C_1 \ve  \langle 1 + |b_e^\ve|,  |b_e^\ve|^{p-1} \rangle _{\Gamma_s^\ve}
\\ & \leq
C_2\Big( 1+ \||b_e^\ve|^{\frac p 2}\|^2_{L^2(\Omega_{e,s}^\ve)} + \ve^2 \|\nabla  |b_e^\ve|^{\frac p 2}\|^2_{L^2(\Omega_{e,s}^\ve)} \Big).
\end{aligned}
\end{equation*}
Using the properties of extension of $b_e^\ve$ from $\Omega_e^\ve$ to $\Omega$
and applying  the Gagliardo-Nirenberg inequality
$$
\|w\|_{L^2(\Omega)} \leq C \|\nabla w \|^{\alpha_1}_{L^2(\Omega)} \| w \|^{1-\alpha_1}_{L^1(\Omega)},
\qquad
 \|w\|_{L^4(\Omega)} \leq C \|\nabla w \|^{\alpha_2}_{L^2(\Omega)} \| w \|^{1-\alpha_2}_{L^1(\Omega)},
$$
with $\alpha_1=\frac 3 5$ and $\alpha_2 = \frac {9}{10}$,   we obtain
\begin{equation*}
\begin{aligned}
\ve \langle |P ( b_e^\ve)|, |b_e^\ve|^{p-1}\rangle_{\Gamma_s^\ve} & \leq C_\delta\big( 1+\||b_e^\ve|^{\frac p 2}\|^2_{L^1(\Omega_{e,s}^\ve)} \big)+ (\ve^2+\delta) \|\nabla  |b_e^\ve|^{\frac p 2}\|^2_{L^2(\Omega_{e,s}^\ve)}, \\
\langle |F_b(b_e^\ve)|, |b_e^\ve|^{p-1}\rangle_{(\partial\Omega)_s} &\leq C\big(1 + \||b_e^\ve|^{\frac p 2}\|^2_{L^2((0,s)\times\partial \Omega)}\big)
\\ &
\leq C_\delta\big[1+ \||b_e^\ve|^{\frac p 2}\|^2_{L^1(\Omega_{e,s}^\ve)}\big] + \delta \|\nabla  |b_e^\ve|^{\frac p 2}\|^2_{L^2(\Omega_{e,s}^\ve)}.
\end{aligned}
\end{equation*}
Then  applying  the Gagliardo-Nirenberg inequality and  the extension lemma, Lemma~\ref{extension}, to  $\||b_e^\ve|^{\frac p2}\|^2_{L^2(\Omega_s)}$ and
$\||b_{e}^\ve|^{\frac p 2}\|^2_{L^2(0,s; L^{4}(\Omega_e^\ve))}$  in \eqref{estim_Lp_b} and using the  estimates \eqref{bou_b_c}  yield
\begin{equation*}
\begin{aligned}
& \|b^\ve_{e}(s)\|^p_{L^p(\Omega_e^\ve)} +\|\nabla |b^\ve_{e}|^{\frac p2} \|^2_{L^2(\Omega_{e,s}^\ve)}\leq
C^p_1
+ C_2 (1+ p^{10})\int_0^s \||b_{e}^\ve|^{\frac p 2}\|^2_{L^1(\Omega_e^\ve)}  dt,
\end{aligned}
\end{equation*}
where the constants $C_1$ and $C_2$ are independent of $\ve$. Then the  Alikakos iteration lemma implies the boundedness of $b_e^\ve$, uniformly in  $\ve$.

We turn to $c^\ve$.
Considering first  $(c_{e,N}^{\ve})^{p-1}$ and  $(c_{f,N}^{\ve})^{p-1}$, where $c_{j,N}^\ve(t,x) = \min\{ c_j^\ve(t,x), N\}$ for  $(t,x) \in \Omega_{j,T}^\ve$  with $j=e,f$ and $N >0$, as test functions in  \eqref{cd_two} and performing calculations similar to those in the derivation of  \eqref{estim_Linfty_be_1}, we obtain
 \begin{equation*}
\begin{aligned}
 \| c_{e,N}^\ve(s)\|^p_{L^p(\Omega_e^\ve)} + \| c_{f,N}^\ve(s)\|^p_{L^p(\Omega_f^\ve)}   + \| \nabla |c_{e,N}^\ve|^{\frac p2}\|^2_{L^2(\Omega_{e,s}^\ve)}+  \| \nabla |c_{f,N}^\ve|^{\frac p 2}\|^2_{L^2(\Omega_{f, s}^\ve)}
\\
 \leq \| c_e^\ve(0)\|^p_{L^p(\Omega_e^\ve)}
 +   \| c_f^\ve(0)\|^p_{L^p(\Omega_f^\ve)}
+ C_1 \big[1+  \|\mathcal G(\partial_t u_f^\ve)\|^2_{L^\infty(\Omega_{f,T}^\ve)}\big]  \|c_{f,N}^\ve\|^p_{L^p(\Omega_{f,s}^\ve)} \\
+ C_2\Big[\|b_e^\ve\|^p_{L^p(\Omega_{e, s}^\ve)} +\|c_{e,N}^\ve\|^p_{L^p(\Omega_{e, s}^\ve)}\Big]
+ C_3\int_0^s \! \big(1+ \|c_e^\ve\|_{L^{p}(\partial\Omega)}\big) \| c_{e,N}^\ve\|^{p-1}_{L^p(\partial \Omega)} dt
\\
+ C_4\|\be(u_e^\ve)\|_{L^\infty(0,s; L^2(\Omega_{e}^\ve))} \int_0^s \! \Big[ \||b_e^\ve|^p\|_{L^2(\Omega_{e}^\ve)}   +
 \|c_{e}^\ve\|_{L^{2p}(\Omega_{e}^\ve)} \|c_{e,N}^\ve\|^{p-1}_{L^{2p}(\Omega_e^\ve)} \\
  + \||c_{e,N}^\ve|^p\|_{L^2(\Omega_{e}^\ve)}\Big] dt
+C_5 \|c_{e}^\ve\|_{L^\infty(0,s; L^2(\Omega_{e}^\ve))}\||c_{e,N}^\ve|^{\frac p2}\|^{2\frac {p-1} p}_{L^2(0,s; L^4(\Omega_{e}^\ve))} \\
 + C_6
 \|c_{f}^\ve\|_{L^\infty(0,s; L^2(\Omega_{f}^\ve))}\ \||c_{f,N}^\ve|^{\frac p 2}\|^{2\frac{p-1} p}_{L^2(0,s; L^4(\Omega_{f}^\ve))} .
\end{aligned}
\end{equation*}
Similar to \eqref{estim_int_cut_b}   we estimate
 \begin{equation*}
\begin{aligned}
\int_0^s \! \|c_{e}^\ve\|_{L^{2p}(\Omega_{e}^\ve)} \|c_{e,N}^\ve\|^{p-1}_{L^{2p}(\Omega_e^\ve)} dt
\leq  \delta_1 \! \sup_{(0,s)} \|c_{e,N}^\ve(s)\|^p_{L^p(\Omega_e^\ve)}
+ \delta_2  \|\nabla |c_{e,N}^\ve|^{\frac p 2}\|^2_{L^2(\Omega_{e,s}^\ve)}
\\  +
C_\delta \Big( 1+\|c_e^\ve\|^{p-1}_{L^{p-1}(\Omega_{e,s}^{\ve})} + \|\nabla c_e^\ve \|^2_{L^{2}(\Omega_{e,s}^{\ve})} +  \|\nabla |c_e^\ve|^{\frac {p-1} 2}\|^2_{L^{2}(\Omega_{e,s}^{\ve})} \Big)^{\frac 32}.
\end{aligned}
\end{equation*}
The boundary integral can be estimated in the same way as in \eqref{estim-bound_int_b}:
\begin{equation*}
\begin{aligned}
\int_0^s \! (1+\|c_e^\ve\|_{L^{p}(\partial\Omega)}) \| c_{e,N}^\ve\|^{p-1}_{L^p(\partial \Omega)} dt \leq  \delta_1 \! \sup_{(0,s)} \| c_{e,N}^\ve(s)\|^p_{L^p(\Omega_{e}^\ve)}
+ \delta_2  \|\nabla  |c_{e,N}^\ve|^{\frac p 2}\|^2_{L^2(\Omega_{e,s}^\ve)}
\\
  +
C(\ve)\Big[1+ \||c_e^\ve|^{\frac {p-1}2}\|^3_{L^\infty(0,s;L^2(\Omega_{e}^\ve))} +  \|\nabla  |c_e^\ve|^{\frac {p-1} 2}\|^{3}_{L^2(\Omega_{e,s}^\ve)}\Big].
\end{aligned}
\end{equation*}
Considering $p=3,\ldots, 6$ iteratively, using estimates \eqref{bou_b_c}, and making the calculations similar to those for $b_{e,N}^\ve$   yield
$$
\begin{aligned}
& \| c_{e,N}^\ve\|_{L^\infty(0,T; L^6(\Omega_e^\ve))} + \| c_{f,N}^\ve(s)\|_{L^\infty(0,T; L^6(\Omega_f^\ve))}
\\ & + \| \nabla |c_{e,N}^\ve|^{3}\|_{L^2(\Omega_{e,s}^\ve)}+  \| \nabla |c_{f,N}^\ve|^3\|_{L^2(\Omega_{f, s}^\ve)} \leq  C,
\end{aligned}
$$
where the constant $C$ depends on $p$ and $\ve$ and  is independent of $N$.  Letting $N \to \infty$ we obtain that $(c_{j}^\ve)^{p-1} \in L^2(0,T; H^1(\Omega_j^\ve))$  with $j=e,f$ and $p=3,\ldots, 6$.
Thus we can consider $(c_e^\ve)^{p-1}$ and $(c_f^\ve)^{p-1}$, with $p=3,4$, as test functions  in  \eqref{cd_two}:
\begin{equation*}
\begin{aligned}
 &\| c_{e}^\ve(s)\|^p_{L^p(\Omega_e^\ve)} + \| c_{f}^\ve(s)\|^p_{L^p(\Omega_f^\ve)}   + \| \nabla |c_{e}^\ve|^{\frac p2}\|^2_{L^2(\Omega_{e,s}^\ve)}+  \| \nabla |c_{f}^\ve|^{\frac p 2}\|^2_{L^2(\Omega_{f, s}^\ve)}
\\ &\leq \| c_e^\ve(0)\|^p_{L^p(\Omega_e^\ve)}  +   \| c_f^\ve(0)\|^p_{L^p(\Omega_f^\ve)} + C_1 \big[1+  \|\mathcal G(\partial_t u_f^\ve)\|^2_{L^\infty(\Omega_{f,T}^\ve)}\big] \|c_{f}^\ve\|^p_{L^{p}(\Omega_{f,s}^\ve)}
\\ & + C_2\|\be(u_e^\ve)\|_{L^\infty(0,s; L^2(\Omega_{e}^\ve))}  \big[ \||b_e^\ve|^{\frac p 2}\|^2_{L^2(0,s; L^4(\Omega_{e}^\ve))}+  \||c_{e}^\ve|^{\frac p 2}\|^2_{L^{2}(0,s ; L^4(\Omega_{e}^\ve))} \big]
\\ & +C_3\big[1+\|b_e^\ve\|^p_{L^p(\Omega_{e, s}^\ve)} + \|c_{e}^\ve\|^p_{L^p(\Omega_{e, s}^\ve)} +  \|c_e^\ve\|^p_{L^{p}((0,s)\times\partial\Omega)}\big] .
\end{aligned}
\end{equation*}

In the same way as  in \eqref{estim_Lp_b},  applying the Gagliardo-Nirenberg inequality to $|c_j^\ve|^{\frac p2}$ in $L^2(\Omega_j^\ve)$ and  $L^4(\Omega_j^\ve)$ and using properties of the extension of $c_e^\ve$ from $\Omega_e^\ve$ to $\Omega$ and of $c^\ve$ from $\widetilde \Omega_{ef}$ to $\Omega$ we obtain
$$
\| c_{e}^\ve\|_{L^\infty(0,T; L^4(\Omega_e^\ve))} + \| c_{f}^\ve \|_{L^\infty(0,T; L^4(\Omega_f^\ve))}
 + \| \nabla |c_{e}^\ve|^{2}\|_{L^2(\Omega_{e,T}^\ve)}+  \| \nabla |c_{f}^\ve|^2\|_{L^2(\Omega_{f, T}^\ve)} \leq  C,
$$
where the constant $C$ is independent of $\ve$.
\end{proof}
Next we present the proof of the estimate for  $\|\widetilde b_{e}^{\ve, j} \|_{L^\infty(0,s; L^\infty(\Omega_e^\ve))}$.
\begin{lemma}\label{Linf_contract}
For the difference of two iterations $\widetilde b^{\ve,j}_e=b^{\ve, j-1}_e - b_e^{\ve, j}$, $\widetilde u^{\ve,j-1}_e=u^{\ve, j-2}_e - u_e^{\ve, j-1}$,  and $\partial_t \widetilde u^{\ve,j-1}_f=\partial_t u^{\ve, j-2}_f - \partial_t u_f^{\ve, j-1}$ for the microscopic system \eqref{eq_codif}--\eqref{exbou_co}, defined in Theorem~\ref{th:exist_uniq_micro}, we have
$$
\begin{aligned}
\|\widetilde b_{e}^{\ve, j} \|_{L^\infty(0,s; L^\infty(\Omega_e^\ve))} & \leq C\|\be(\widetilde u_e^{\ve, j-1})\|_{L^{1+ \frac 1\sigma}(0,s; L^2(\Omega_e^\ve))}
\\ &  +C_\delta \|\partial_t \widetilde u_{f}^{\ve, j-1}\|_{L^2(\Omega_{f,s}^\ve)} +
\delta  \|\be(\partial_t \widetilde u_{f}^{\ve, j-1})\|_{L^2(\Omega_{f,s}^\ve)}
\end{aligned}
$$
for  $s \in (0,T)$,  any $\delta>0$ and   $0<\sigma < 1/9$, where the constants $C$ and $C_\delta$ are independent of $s$ and $j$.
\end{lemma}
\begin{proof} Considering  $(\widetilde b_{e}^{\ve, j})^{p-1}$ as a test function in the weak formulation of \eqref{eq_diff} yields
 \begin{equation}\label{estim_contr_222}
\begin{aligned}
& \frac 1 p \|\widetilde b_{e}^{\ve, j}(s)  \|^p_{L^p(\Omega_e^\ve)}  + \frac {2(p-1)}{p^2} \|\nabla |\widetilde b_{e}^{\ve,j}|^{\frac p 2} \|^2_{L^2(\Omega_{e,s}^\ve)}  \leq
C_1 \| \widetilde b^{\ve, j}_{e} \|^p_{L^p(\Omega_{e,s}^\ve)}
\\
&+C_2\left[ \|c^{\ve, j}_e\|_{L^\infty(0,s; L^2(\Omega_e^\ve))} + \|{\bf e}(u_{e}^{\ve, j-1})\|_{L^\infty(0,s; L^2(\Omega_e^\ve))}\right] \||\widetilde b_{e}^{\ve, j}|^{\frac p2}\|^2_{L^2(0,s; L^{4}(\Omega_e^\ve))}
\\
& +  C_3 \|b_{e}^{\ve, j-1}\|_{L^\infty(\Omega_{e,s}^\ve)} \left[\frac 1 p  \|\widetilde c_e^{\ve, j}\|^p_{L^\infty(0,s; L^2(\Omega_e^\ve))} +
  \frac { p-1}p\||\widetilde b_{e}^{\ve,j}|^{\frac p 2}\|^2_{L^2(0, s; L^4(\Omega_e^\ve))} \right]\\
  &+
 C_4 \|b_{e}^{\ve, j-1}\|_{L^\infty(\Omega_{e,s}^\ve)}   \langle |{\bf e}(\widetilde u_{e}^{\ve, j-1})|, |\widetilde b_{e}^{\ve, j}|^{p-1} \rangle_{\Omega_{e,s}^\ve}
\end{aligned}
\end{equation}
for $s\in (0,T)$.  Applying the H\"older inequality we estimate
\begin{equation}
\begin{aligned}
\langle |{\bf e}(\widetilde u_{e}^{\ve, j-1})| , |\widetilde b_{e}^{\ve, j}|^{p-1} \rangle_{\Omega^\ve_{e,s}} \leq \int_0^s \|{\bf e}(\widetilde u_{e}^{\ve, j-1})\|_{L^{\frac{2p}{p+1}}(\Omega_e^\ve)} \|\widetilde b_{e}^{\ve, j} \|^{p-1}_{L^{2p}(\Omega_{e}^\ve)}  dt \\
 \leq  C_1 \int_0^s \|{\bf e}(\widetilde u_{e}^{\ve, j-1})\|_{L^{2}(\Omega_e^\ve)} \|\widetilde b_e^{\ve, j} \|^{p-1}_{L^{2p}(\Omega_{e}^\ve)} dt\\
\leq C_2 \Big(\int_0^s \|\be(\widetilde u_{e}^{\ve,j-1})\|^{\frac{p(1+\sigma)}{(p\sigma +1)}}_{L^{2}(\Omega_e^\ve)} dt\Big)^{\frac{(p\sigma +1)}{p(1+\sigma)}}  \Big(\int_0^s \||\widetilde b_{e}^{\ve, j}|^{\frac p 2} \|^{2(1+\sigma)}_{L^{4}(\Omega_e^\ve)} dt\Big)^{\frac{(p-1)}{p(1+\sigma)}}
\end{aligned}
\end{equation}
for some $\sigma >0$. Applying the  Gagliardo-Nirenberg inequality
$$
\|w\|_{L^4(\Omega)} \leq C \|\nabla w \|^{\alpha}_{L^2(\Omega)} \| w \|^{1-\alpha}_{L^1(\Omega)}
$$
with $\alpha = \frac {9}{10}$,   we obtain for  $0<\sigma < 1/9$
\begin{eqnarray}\label{contract_b_e_1}
\Big(\!\int\limits_0^s \||\widetilde b_{e}^{\ve, j}|^{\frac p 2} \|^{2(1+\sigma)}_{L^{4}(\Omega_e^\ve)} dt \Big)^{\frac 1{1+\sigma}}\! \leq  C \Big(\! \int\limits_0^s \! \|\nabla |\widetilde b_{e}^{\ve, j}|^{\frac p 2} \|^{2(1+\sigma)\alpha}_{L^{2}(\Omega_e^\ve)} \||\widetilde b_{e}^{\ve, j} |^{\frac p 2} \|^{2(1+\sigma)(1-\alpha)}_{L^{1}(\Omega_e^\ve)} dt \Big)^{\frac 1{(1+\sigma)}} \nonumber \\
 \leq  C_2 \|\nabla |\widetilde b_{e}^{\ve, j}|^{\frac p 2} \|^{2\alpha}_{L^2(\Omega_{e,s}^\ve)} \Big( \int_0^s \| |\widetilde b_{e}^{\ve, j}|^{\frac p 2} \|^{\frac{2(1+\sigma)(1-\alpha)}{ 1- \alpha(1+\sigma)}}_{L^{1}(\Omega_e^\ve)} dt \Big)^{\frac{1-\alpha(1+\sigma)}{1 + \sigma}} \nonumber \\
\leq\frac  \delta p  \| \nabla |\widetilde b_{e}^{\ve, j}|^{\frac p 2} \|^2_{L^2(\Omega_{e,s}^\ve)}  +
C_\delta p^{\frac \alpha{1 - \alpha}}\Big( \int_0^s \||\widetilde b^{\ve, j}_{e}|^{\frac p 2} \|^{\frac{2(1+\sigma)(1-\alpha)}{ 1- \alpha(1+\sigma)}}_{L^{1}(\Omega_e^\ve)} dt \Big)^{\frac{1-\alpha(1+\sigma)}{(1 + \sigma)(1-\alpha)}}\\
 \leq \frac \delta p   \|\nabla |\widetilde b_{e}^{\ve, j}|^{\frac p 2} \|^2_{L^2(\Omega_{e,s}^\ve)}  +  C_\delta  p^{\frac \alpha{1 - \alpha}}\| |\widetilde b_{e}^{\ve, j}|^{\frac p 2} \|_{L^\infty(0,s; L^1(\Omega_e^\ve))}^2 \nonumber
\end{eqnarray}
for any $\delta >0$. Hence we have the following estimate:
\begin{equation}\label{estim_conts_332}
\begin{aligned}
& \langle |{\bf e}(\widetilde u_{e}^{\ve,j-1})|, |\widetilde b_{e}^{\ve, j}|^{p-1} \rangle_{\Omega_{e,s}^\ve} \leq
\delta \frac { p-1}{p^2} \|\nabla |\widetilde b_{e}^{\ve,j}|^{\frac p 2} \|^2_{L^2(\Omega_{e,s}^\ve)}
\\ & \quad + C_\delta \frac{(p-1) p^\beta}{p} \| |\widetilde b_{e}^{\ve,j}|^{\frac p 2} \|_{L^\infty(0, s; L^1(\Omega_e^\ve))}^2   +
C\frac 1 p \|{\bf e}(\widetilde u_{e}^{\ve, j-1})\|^p_{L^{1+ \frac{1}\sigma}(0,s; L^{2}(\Omega_e^\ve))},
\end{aligned}
\end{equation}
with $\beta = \frac \alpha{1 - \alpha}$.
Incorporating   \eqref{estim_conts_332}   in  \eqref{estim_contr_222},
estimating   $ \| \widetilde b_{e}^{\ve, j} \|^p_{L^p(\Omega_{e,s}^\ve)}$ and \\ $\||\widetilde b_{e}^{\ve,j}|^{\frac p 2}\|^2_{L^2(0, s; L^4(\Omega_e^\ve))}$  in terms of  $
\|| \widetilde b_{e}^{\ve, j}|^{\frac p 2} \|^2_{L^2(0,s; L^1(\Omega_e^\ve))}$ and $\|\nabla |\widetilde b_{e}^{\ve,j}|^{\frac p 2} \|^2_{L^2(\Omega_{e,s}^\ve)}$ by applying   the  Gagliardo-Nirenberg inequality,   and using  the estimate \eqref{eq_diff_bc_weak} for \\ $\|\widetilde c_e^{\ve, j}\|^p_{L^\infty(0,s; L^2(\Omega_e^\ve))}$  and the boundedness of  $b_e^{\ve, j-1}$, which can be shown in the same way as the $L^\infty$-estimates in  \eqref{estim_Linfty_L4},  yield
\begin{equation*}
\begin{aligned}
 &\|\widetilde b_{e}^{\ve, j}(s)  \|^p_{L^p(\Omega^\ve_e)}  + \|\nabla |\widetilde b_{e}^{\ve,j}|^{\frac p 2} \|^2_{L^2(\Omega^\ve_{e,s})} \\
& \qquad \qquad  \leq C_1 p^{10} \sup\limits_{(0,s)}\| |\widetilde b^{\ve, j}_{e}|^{\frac p 2} \|^2_{L^1(\Omega_e^\ve)}
 + C_2^p \|{\bf e}(\widetilde u_{e}^{\ve, j-1})\|^p_{L^{1+ \frac{1}\sigma}(0,s; L^{2}(\Omega_e^\ve))}
\\ & \qquad \qquad  +  \delta^p  \|{\bf e}(\partial_t \widetilde u_{f}^{\ve, j-1})\|^p_{L^{2}(\Omega_{f,s}^\ve)}
 +  C_\delta^p  \|\partial_t \widetilde u_{f}^{\ve, j-1}\|^p_{L^{2}(\Omega_{f,s}^\ve)}.
\end{aligned}
\end{equation*}
Using \eqref{eq_diff_bc_weak} and iterating in $p=2^k$, for $k=2,3,\ldots$, similarly to \cite[Lemma 3.2]{Alikakos}, we obtain the estimate stated in Lemma.
\end{proof}


\begin{thebibliography}{99}

\bibitem{Acerbi}
Acerbi E.,  Chiad\'o Piat V., Dal Maso G.,  Percivale D. {An extension theorem from connected sets, and homogenization in general periodic domains}. \emph{Nonlin. Anal. Theory, Methods, Applic.}, \textbf{18}, 481--496, 1992.

\bibitem{Alikakos}
Alikakos N.D. $L^p$ bounds of solutions of reaction-diffusion equations. \emph{Comm. Partial Differential  Equations}, \textbf{4}, 827-868, 1979.

\bibitem{allaire}
Allaire G. {Homogenization and two-scale convergence}. \emph{SIAM J.~Math.~Anal.}, \textbf{23}, 1482--1518, 1992.

\bibitem{Allaire_1996}
Allaire, G.,  Damlamian, A.,  Hornung, U.  {\em  Two-scale convergence on periodic surfaces and applications}, in Proceedings of the International Conference on Mathematical Modelling of Flow through Porous Media, A. Bourgeat et al., eds., World Scientific, Singapore,  15--25, 1996.


\bibitem{Auriault}
Auriault J.-L.,  Sanchez-Palencia E. Etude du comportement macroscopic d\'un mileu poreux satur\'e d\'eformable, \emph{J. M\'ecanique}, {\bf 16} 1977, 575--603.



\bibitem{Badia}
Badia S., Quaini A.,  Quarteroni A. Coupling Biot and Navier-Stokes equations for modelling fluid-poroelastic media interaction, \emph{Journal of Computational Physics}, \textbf{228}, 7986--8014, 2009.


\bibitem{BLP}
Bensoussan A., Lions J.-L., Papanicolaou G. \emph{Asymptotic Analysis for Periodic Structures}. North Holland, Amsterdam, New York, 1978.



\bibitem{Biot_1972}
Biot M.A. Theory of finite deformations of porous solids.  \emph{Indiana University Mathematics J}, \textbf{21}, 597--620, 1972.

\bibitem{Biot_1962}
Biot, M.A.  Generalized theory of acoustic propagation in porous dissipative media. \emph{J. Acoust. Soc. Amer.},  {\bf 34}, 1256--1264, 1962.

\bibitem{Biot_1941}
Biot M.A. General theory of three-dimensional consolidation. \emph{J. Appl. Phys.},   {\bf 12}, 155--164, 1941.

\bibitem{Broedersz}
Broedersz C.P., MacKintosh F.C. Modeling semiflexible polymer networks. \emph{Rev. Mod. Phys.}, \textbf{86}, 995-1036, 2014.


\bibitem{Bukac}
Buka\v c M., Yotov I.,  Zakerzadeh R.,  Zunino P.
 Partitioning strategies for the interaction of a fluid with a poroelastic material based on a Nitsche's coupling approach.
\emph{Computer Methods in Applied Mechanics and Engineering}, \textbf{292}, 138Ð170, 2015.

\bibitem{Burridge}
Burridge R., Keller J. B. Poroelasticity equations derived from microstructure. \emph{J. Acoust. Soc.\ Amer.},  \textbf{70}, 1140--1146, 1981.


\bibitem{CC_1994}
Chrispeels M.J., Maurel C. Aquaporins: The molecular basis of facilitated water movement through living plant cells? \emph{Plant Physiol.}, \textbf{105}, 9--13, 1994.

\bibitem{CDDGZ}
Cioranescu D, Damlamian A., Donato P.,  Griso G, Zaki R. The periodic unfolding method in domains with holes.
\emph{SIAM J. Math.~Anal.}, \text{44}, 718--760, 2012.

\bibitem{CiorPaulin99}
Cioranescu D., Saint Jean Paulin J. \emph{Homogenization of Reticulated Structures}.  Springer-Verlag, New York, 1999.

\bibitem{Clopeau}
Clopeau T.,   Ferrin J.L., Gilbert R.P., Mikeli\'c A. Homogenizing the acoustic properties of the seebed. II. \emph{Math.~Comput.~Modelling}, \textbf{33}, 821--841, 2001.

\bibitem{Cosgrove}
Cosgrove D.-J.,Growth of the plant cell wall. \emph{Nature Reviews Molecular Cell Biology},  \textbf{6}, 850--861, 2005.


\bibitem{GM} Gilbert R.P.,  Mikeli\'c A. Homogenizing the acoustic properties of the seabed: Part I. \emph{Nonlinear Analysis}, \textbf{40}, 185--212, 2000.

\bibitem{JMN}
J\"ager W.,  Mikeli\'c A.,  Neuss-Radu M. Homogenization limit of a model system for interaction of flow, chemical reactions, and mechanics in cell tissues. \emph{SIAM J. Math. Anal.}, \textbf{43}, 1390--1435, 2011.

\bibitem{JKO}
Jikov V.V., Kozlov S.M., Oleinik O.A. \emph{Homogenization of Differential Operators and Integral Functionals}.
Springer-Verlag, Berlin, 1994.

\bibitem{LSU}
Ladyzhenskaya O.A., Solonnikov V.A., Ural'ceva N.N.  \emph{Linear and quasilinear equations of parabolic type}. American Mathematical Society, 1968.

\bibitem{Levy}
Levy T. Propagation waves in a fluid-saturated porous elastic solid. \emph{Internat. J.\ Engrg.\ Sci.},  \textbf{17}, 1005--1014, 1979.


\bibitem{Meirmanov_2007}
Meirmanov A.M.  Nguetseng's two-scale convergence method for filtration and seismic acoustic problems in elastic porous media. \emph{Siberian Math.\ J.},  {\bf 48}, 519--538, 2007.

\bibitem{Meirmanov_2008}
Meirmanov A.M. Homogenization models for a short-time filtration in elastic porous media. \emph{Electronic J.~Differential Equations},  {\bf 2008}, 1-18, 2008.


  \bibitem{MW}
 Mikeli\'c A., Wheeler M.F. On the interface law between a deformable porous medium containing a viscous fluid and an elastic body. \emph{Math. Models Methods Appl. Sci.},  \textbf{22}, 1250031, 2012.

\bibitem{Murad_2000}
Murad M.A., Guerreiro J.N., Loula A.F.D. Micromechanical computational modeling of reservoir compaction and surface subsidence. \emph{Math. Contemp.}, \textbf{19}, 41--69, 2000.

\bibitem{Murad_2001}
Murad M.A., Guerreiro J.N., Loula A.F.D. Micromechanical computational modeling of secondary consolidation and hereditary creep in soils. \emph{Comput. Methods Appl. Mech. Engrg.}, 1\textbf{90}, 1985--2016, 2001.


\bibitem{Necas} Nec\v{a}s, J. {\em \'Equations aux D\'eriv\'ees Partielles}. Presses de Universit\'e de Montr\'eal, Montreal, 1965.

\bibitem{neuss-radu}
Neuss-Radu M. Some extensions of two-scale convergence.  \emph{C.R.~Acad.~Sci.~Paris}, \textbf{332},  899--904, 1996.

\bibitem{Nguetseng}
Nguetseng G. {A general convergence result for a functional related to the theory of homogenization}. \emph{SIAM J.~Math.~Anal.},  \textbf{20}, 608--623, 1989.

\bibitem{Nguetseng_1990}
Nguetseng G. Asymptotic analysis for a stiff variational problem arising in mechanics. \emph{SIAM J. Math. Anal.},  {\bf 21}, 1394--1414, 1990.


\bibitem{OShY}
Oleinik O., Shamaev A.S., Yosifian G.A. \emph{Mathematical Problems in Elasticity and Homogenization}. North Holland, Amsterdam, 1992.

\bibitem{P2010}
Pelletier S.,  Van Orden J., Wolf S., Vissenberg K., Delacourt J., Ndong Y.A., Pelloux J., Bischoff V., Urbain A., Mouille G., Lemonnier G., Renou J.P., H\"ofte H. A role for pectin de-methylesterification in a developmentally regulated growth acceleration in dark-grown \textit{Arabidopsis} hypocotyls. \emph{New Phytol}, \textbf{188}, 726--739, 2010.

\bibitem{PB} Proseus T.E., Boyer J.S.: Calcium deprivation disrupts enlargement of \textit{Chara corallina} cells: further evidence for the calcium pectate cycle. \emph{Journal of Experimental Botany}, \textbf{63}, 3953--3958, 2012.

\bibitem{Proseus_2007}
Proseus, T.E.,  Boyer J.S.  Tension required for pectate chemistry to control growth in {\it Chara corallina}.  Journal of Experimental Botany, \textbf{58}, 4283--4292, 2007.

\bibitem{Proseus_2006}
Proseus, T.E.,  Boyer J.S. Calcium pectate chemistry controls growth rate of {\it Chara corallina}.  \emph{Journal of Experimental Botany}, \textbf{57},  3989--4002, 2006.

\bibitem{Proseus_2005}
Proseus, T.E.,  Boyer J.S. Turgor pressure moves polysaccharides into growing cell walls of {\it Chara corallina}.  {\it Annals of Botany}, \textbf{95}, 967--979,  2005.

\bibitem{Ptashnyk}
Ptashnyk, M., Seguin B. Homogenization of a system of elastic and reaction-diffusion equations modelling plant cell wall biomechanics, 
\emph{ESAIM: Math.~Model.~Numer.~Anal. (ESAIM: M2AN)}, \textbf{50}, 593--631, 2016.

\bibitem{Redlinger}
Redlinger R.,  Invariant sets for strongly coupled reaction-diffusion systems under general boundary conditions. \emph{Arch. Rational Mech. Anal.} \textbf{108},  281--291, 1989.

\bibitem{Rojas}
Rojas E.R.,  Hotton S.,  Dumais J. Chemically mediated mechanical expansion of the pollen tube cell wall. \emph{Biophys.~J.}, \textbf{101},  1844--1853, 2011.

\bibitem{SP}
Sanchez-Palencia E. \emph{Non-Homogeneous Media and Vibration Theory}. Springer-Verlag, Berlin,  New York, 1980.

\bibitem{Schuster_2012}
Schuster E., Lundin L.,  Williams M.A.K. Investigating the relationship between network mechanics and single-chain extension using biomimetic polysaccharide gels. \emph{Macromolecules}, \textbf{45}, 4863--4869, 2012.


\bibitem{Shamaev_2007}
Shamaev A.S., Samarin V.A. On propagation of acoustic waves in the medium consisting of a fluid and an elastic material. \emph{J Mathematical Sciences},  {\bf 144}, 4284--4291, 2007.


\bibitem{Showalter}
Showalter R.E. Diffusion in poro-elastic media. \emph{J~Math.\ Anal.~Appl.}, \textbf{251}, 310--340, 2000.

\bibitem{Showalter_2005}
Showalter R.E. Poroelastic filtration coupled to Stokes flow. {\em in Control theory of partial differential equations}, 229--241, Lect. Notes Pure Appl. Math., 242, Chapman and Hall/CRC, Boca Raton, FL, 2005.

\bibitem{Smoller}
Smoller J. Shock Waves and Reaction-Diffusion Equations.  2nd ed., \emph{Springer-Verlag},  New York, 1994.

\bibitem{Somerville}
Somerville Ch.,   Bauer S.  Brininstool G.,   Facette M.,  Hamann Th.,  Milne J.,  Osborne E., Paredez A., Persson S., Raab T.,  Vorwerk S.,   Youngs H.  Toward a systems approach to understanding plant cell walls.  \emph{Science},  \textbf{306}, 2206--2211, 2004.

\bibitem{White_2001}
White P.J. The pathways of calcium movement to the xylem. \emph{J. Exp. Bot.},  \textbf{52},  891--899, 2001.

\bibitem{WG} Wolf S., Greiner S. Growth control by cell wall pectins. \emph{Protoplasma},  \textbf{249}, 169--175, 2012.

\bibitem{WHH}  Wolf S.,   H\'ematy K., H\"ofte H.  Growth control and cell wall signaling in plants. \emph{Annu. Rev. Plant Biol.}, \textbf{63}, 381--407, 2012.


\bibitem{Xiong}
Xiong T.-C., Bourque S., Lecourieux, D., Amelot N., Grat S., Bri\`ere C., Mazars C., Pugin A., Ranjeva R. Calcium signaling in plant cell organelles delimited by a double membrane. \emph{Biochimica et Biophysica Acta},  \textbf{1763}, 1209--1215, 2006.


\bibitem{Zenisek_1984}
\v Zen\'i\v sek A. The existence and uniqueness theorem in Biot's consolidation theory, \emph{Appl. Mat.}, \textbf{29}, 194--211, 1984.

\bibitem{Zhikov}
Zhikov V.V. On an extension of the method of two-scale convergence and its applications. \emph{Sbornik Mathematics}, \textbf{191}, 973--1014, 2000.

 \bibitem{Zsivanovits_2004}
 Zsivanovits, G., MacDougall, A.J., Smith, A.C., Ring, S.G. Material properties of concentrated pectin networks. \emph{Carbohydr.~Res.}, \textbf{339}, 1217--1322, 2004.

\end{thebibliography}
\end{document}